\documentclass{amsart}
\usepackage[english]{babel}
\usepackage[T1]{fontenc}
\usepackage{indentfirst}
\usepackage{latexsym}
\usepackage{mathtools}
\usepackage{hyperref}
\hypersetup{colorlinks,breaklinks,
            linkcolor=[rgb]{0,0.55,0.9},
            citecolor=[rgb]{0,0.55,0.9},
            urlcolor=[rgb]{0,0,0}}
\usepackage[dvips]{graphicx}
\DeclareGraphicsExtensions{.pdf,.png,.jpg,.eps}
\usepackage{amsmath} 
\usepackage{amsfonts,amsthm,amssymb,amsbsy}
\usepackage{textcomp}
\usepackage[usenames,dvipsnames]{color}
\usepackage{color}
\usepackage{pgf}
\usepackage{enumitem}
\usepackage{pgfplots}
\usepackage{multirow}
\usepackage{esint}
\usepackage[hang,sf,small]{caption}
\usepackage{graphicx}
\usepackage{wrapfig}
\usepackage{verbatim}
\usepackage[utf8]{inputenc} 

\usepackage{fancyhdr}

\usepackage[color = white, linecolor = black, shadow]{todonotes}

\usepackage{stackengine}
\stackMath

\usepackage{color}
\usepackage[active]{srcltx}
\newcommand{\brd}[1]{\mathbb{#1}}

\newcommand{\R}{\brd{R}}
\newcommand{\N}{\brd{N}}
\newcommand{\CCC}{\brd{C}}

\newcommand{\abs}[1]{\left\lvert {#1} \right\rvert}
\newcommand{\norm}[2]{\left\Vert {#1} \right\Vert_{#2}}

\newcommand{\loc}{{{\tiny{\mbox{loc}}}}}

\newcommand\ddfrac[2]{\frac{\displaystyle #1}{\displaystyle #2}}
\newtheorem{teo}{Theorem}[section]
\newtheorem{corollary}[teo]{Corollary}
\newtheorem{lemma}[teo]{Lemma}
\newtheorem{theorem}[teo]{Theorem}
\newtheorem{proposition}[teo]{Proposition}
\theoremstyle{definition}
\newtheorem{definition}[teo]{Definition}

\usepackage{cfr-lm}
\usepackage[osf]{libertine}

\begin{document}

\title[On the nodal set of solutions to degenerate-singular elliptic equations]{On the nodal set of solutions to degenerate or singular elliptic equations with an application to $s-$harmonic functions}
\date{\today}

\author[Y. Sire]{Yannick Sire}\thanks{}
\address{Yannick Sire \newline \indent
Johns Hopkins University  \newline \indent
Department of Mathematics,
 \newline \indent
 Baltimore, MD 21218, USA}
\email{sire@math.jhu.edu }

\author[S. Terracini]{Susanna Terracini}\thanks{}
\address{Susanna Terracini \newline \indent
 Dipartimento di Matematica ``Giuseppe Peano'', Universit\`a di Torino, \newline \indent
Via Carlo Alberto, 10,
10123 Torino, Italy}
\email{susanna.terracini@unito.it}

\author[G. Tortone]{Giorgio Tortone}\thanks{}
\address{Giorgio Tortone \newline \indent
Dipartimento di Matematica ``Giuseppe Peano'', Universit\`a degli Studi di Torino  \newline \indent
Via Carlo Alberto, 10,
10123 Torino, Italy
 \newline \indent
 and Politecnico di Torino (Italy)}
\email{giorgio.tortone@unito.it}

\date{\today} 

\subjclass[2010] {
35J70, 
35J75,  
35R11, 
35B40, 
35B44, 
35B53, 
}

\keywords{Nodal set, degenerate or singular elliptic equations, nonlocal diffusion, monotonicity formulas, blow-up classification}
%
%
%
%
%
%
%
%
%
%
%

\thanks{Work partially supported by the ERC Advanced Grant 2013 n.~339958 Complex Patterns for Strongly Interacting Dynamical Systems - COMPAT }

\maketitle

\begin{abstract}
This work is devoted to the geometric-theoretic analysis of the nodal set of solutions to degenerate or singular equations involving a class of operators including
$$
L_a = \mbox{div}(\abs{y}^a \nabla),
$$
with $a\in(-1,1)$ and their perturbations.

As they belong to the Muckenhoupt class $A_2$, these operators appear in the seminal works of Fabes, Kenig, Jerison and Serapioni \cite{fkj,fjk2,fks} and have recently attracted a lot of attention in the last decade due to their link to the localization of the fractional Laplacian via the extension in one more dimension \cite{CS2007}. Our goal in the present paper is to develop a complete theory of the stratification properties for the nodal set of solutions of such equations in the spirit of the seminal works of Hardt, Simon, Han and Lin \cite{MR1010169,MR1305956,MR1090434}.
\end{abstract}

\tableofcontents

\section{Introduction}

In the last decades the question of the structure of the nodal set of solutions of elliptic equations has been brought to the attention of the scientific community (see e.g. \cite{MR943927,MR1305956,MR1639155, MR1090434}), with a special focus on the measure theoretical features of its singular part, also in connection with the validity of a strong unique continuation principle, in order to ensure the existence of a finite vanishing order, as pointed out in \cite{MR833393, MR882069, MR1090434}. Recently, major progress have been done on the study of nodal sets of eigenfunctions (or critical sets of harmonic functions) by Logunov and Malinnikova \cite{logu1,logu2,LM} in connection with conjectures by Yau and Nadirashvili. It would be a challenge to adapt their techniques to more general equations like ours.
\\

In this paper we consider the nodal set in $\R^{n+1}$ of solution of a peculiar class of degenerate-singular operators which have recently become very popular in the study of the fractional powers of the Laplacian,  and firstly studied in the pioneering works \cite{fkj,fjk2,fks}. Given $a \in (-1,1)$ and $X=(x,y) \in \R^{n}_x\times \R_y$ we consider a class of operators including
$$
L_a = \mbox{div}(\abs{y}^a \nabla),
$$
and their perturbations, where  we denote by $\mbox{div}$ and $\nabla$ respectively the divergence and the gradient operator in $\R^{n+1}$. Our main purpose is to fully understand the local behaviour of $L_a$-harmonic functions near their nodal set and to develop a geometric analysis of its structure and regularity, in order to comprehend how the degenerate or singular character of the coefficients  can affect the local picture of the nodal set itself. Thus, we introduce the notion of \emph{characteristic manifold} $\Sigma$ associated with the operator $L_a$, as the set of points where the coefficient either vanishes or blows up, and we study the properties of the nodal set $\Gamma(u)$ of solutions to equation
\begin{equation*}
  -L_a u =0 \quad \mbox{in }B_1\subset \R^{n+1}.
\end{equation*}
In particular, since the operator $L_a$ is locally uniformly elliptic on $\R^{n+1}\setminus \Sigma$, we restrict our attention on the structure of the nodal set neighbouring the characteristic manifold $\Sigma$, trying to understand the structural difference between the whole nodal set $\Gamma(u)=\left \{ x \in B_1,\,\,\,u(x)=0 \right \}$ and its restriction on  $\Sigma$.\\

As a further motivation,  this analysis will be the starting point of the study of competition-diffusion systems of populations under an anomalous diffusion. More precisely, we can imagine that the characteristic manifold $\Sigma$ is playing a major
role in the diffusion phenomenon by penalizing or encouraging the diffusion across $\Sigma$, according with the value of $a \in (-1,1)$. Inspired by \cite{tvz1, tvz2,
MR2393430, MR2146353, MR2599456}, in the case of strong competition, the limiting segregated configurations will satisfy a refection law which represents the
only interaction between the different densities through the common free boundary. Thanks to
this reflection property, the free boundary will be locally described as the nodal set of $L_a$-harmonic function.\\

As already mentioned, our  operators belong to the class introduced in the 80's by  Fabes, Jerison, Kenig and Serapioni in \cite{fkj,fjk2,fks}, where they established H\"older continuity of solutions within a general class of degenerate-singular elliptic operators $L= \mbox{div}(A(X)\nabla \cdot)$ whose coefficient $A(X)=(a_{ij}(X))$ are defined starting from a symmetric matrix valued function such that
$$
\lambda \omega(X)\abs{\xi}^2\leq (A(X)\xi, \xi)\leq \Lambda \omega(X)\abs{\xi}^2, \quad \mbox{for some } \lambda,\Lambda >0,
$$
where the weight $\omega$ may either vanish, or be infinite, or both. In particular, the prototypes of weights considered in their analysis where in the Muckenhoupt $A_2$-class, i.e. such that
$$
\sup_{B\subset \R^{n+1}}\left(\frac{1}{\abs{B}}\int_{B}\omega(X)\mathrm{d}X\right)\left(\frac{1}{\abs{B}}\int_{B}\omega^{-1}(X)\mathrm{d}X\right)<\infty.
$$

Our case correspond to the choice $\omega(X)=|y|^a$ and is  Muckenhoupt whenever  $a\in(-1,1)$.  Note however that this class of $A_2-$weights is not the optimal one to have H\"older regularity as noticed in \cite{fks}. However, it provides a good model with applications for our purposes.\\


Our approach is based upon the validity of an Almgren and Weiss type monotonicity formul\ae, the existence and uniqueness of non trivial tangent maps at every point of the nodal set, and on a complete classification of the possible homogenous configurations appearing at the blow-up limit. Nevertheless, the starting point of our analysis relies on the decomposition of an $L_a$-harmonic function with respect to the orthogonal direction to the characteristic manifold $\Sigma$. Indeed, denoting by $H^{1,\beta}(B_1)$ the Sobolev space w.r.t. the measure $|y|^{\beta}\,dy \,dx$, we have (see also \cite{GZ,conformal})
\begin{proposition}
Given $a \in (-1,1)$ and $u$ an $L_a$-harmonic function in $B_1$, there exist two unique  functions $u_e^a \in H^{1,a}(B_1), u_e^{2-a}\in H^{1,2-a}(B_1)$ symmetric with respect to $\Sigma$ respectively $L_a$ and $L_{2-a}$ harmonic in $B_1$ and locally smooth, such that
\begin{equation*}
  u(X)=u_e^a(X)+u_e^{2-a}(X)y\abs{y}^{-a}\quad\mbox{in }B_1.
\end{equation*}
\end{proposition}

With this decomposition in mind, we can reduce the classification of the possible blow-up limits to the symmetric ones and  eventually recover all the possible cases. In particular, it is worthwhile introducing a new notion of \emph{tangent field} $\Phi^{X_0}$ of $u$ at a nodal point, which takes care of the different behaviour of both the symmetric and antisymmetric part of an $L_a$-harmonic function. Namely, by the decomposition and the Definition \ref{tangent.map} of the notion of tangent map, i.e.  the unique nonzero map $\varphi^{X_0} \in \mathfrak{B}^a_k(u)$ such that
  $$
  u_{X_0,r}(X)=\frac{u(X_0+rX)}{r^k} \longrightarrow \varphi^{X_0}(X),
  $$
with $k$ the vanishing order of $u$ at $X_0$, we introduce the following concept.
\begin{definition}
  Let $a \in (-1,1), u$ be an $L_a$-harmonic function in $B_1$ and $X_0 \in \Gamma_k(u)\cap\Sigma$, for some $k \geq \min\{1,1-a\}$. We define as \emph{tangent field} of $u$ at $X_0$ the unique nontrivial vector field $\Phi^{X_0} \in (H^{1,a}_\loc(\R^{n+1}))^2$ such that
  $$
  \Phi^{X_0}= (\varphi^{X_0}_e,\varphi^{X_0}_o),
  $$
  where $\varphi^{X_0}_e$ and $\varphi^{X_0}_o$ are respectively the tangent map of the symmetric part $u_e$ of $u$ and of the antisymmetric one $u_o$.
\end{definition}
At first, the notion of tangent field allows us  to describe the topology of the nodal set by proving in Proposition \ref{generalized.upper} a \emph{vectorial} counterpart of the classic result of upper semi-continuity of the vanishing order. In order to define properly the relevant subsets, we define
$$
\partial^a_y u = \begin{cases}
                   \abs{y}^a \partial_y u & \mbox{if } X \not \in \Sigma\\
                   \lim_{y \to 0} \abs{y}^a \partial_y u(x,y) & \mbox{if }X \in \Sigma
                 \end{cases}.
$$
This quantity, as observed already in previous works, is the nontrivial one to be considered as far as the derivative in $y$ is concerned.

In the light of this observation, it is natural to define the regular part $\mathcal{R}(u)$ and the singular part $\mathcal{S}(u)$ of the nodal set as follows:

\begin{align*}
\mathcal{R}(u)&= \left\{ X \in \Gamma(u)\colon \abs{\nabla_x u(X)}^2 + \abs{\partial^a_y u(X)}^2 \neq 0 \right\},\\
\mathcal{S}(u)&= \left\{ X \in \Gamma(u)\colon \abs{\nabla_x u(X)}^2 + \abs{\partial^a_y u(X)}^2 = 0 \right\},
\end{align*}
we developed a blow-up analysis in order to fully understand the structure of $\Gamma(u)$ in $\R^{n+1}$ and its restriction on $\Sigma$. The following is a summary of our main result describing the stratified structure of both the regular and singular parts of the nodal set.
\begin{theorem}
   Let $a \in (-1,1), a \neq 0$ and $u$ be an $L_a$-harmonic function in $B_1$. Then the regular set $\mathcal{R}(u)$ is locally a $C^{k,r}$ hypersurface on $\R^{n+1}$ in the variable $(x,y\abs{y}^{-a})$ with
   $$
   k=\left\lfloor \frac{2}{1-a}\right\rfloor \quad\mbox{and}\quad r = \frac{2}{1-a} - \left\lfloor \frac{2}{1-a}\right\rfloor.
   $$
   On the other hand, there holds
  $$
  \mathcal{S}(u)\cap \Sigma = \mathcal{S}^*(u) \cup \mathcal{S}^a(u)
  $$
  where $\mathcal{S}^*(u)$ is contained in a countable union of $(n-2)$-dimensional $C^1$ manifolds and $\mathcal{S}^a(u)$ is contained in a countable union of $(n-1)$-dimensional $C^1$ manifolds. Moreover
  $$
   \mathcal{S}^*(u)=\bigcup_{j=0}^{n-2} \mathcal{S}^*_j(u)\quad\mbox{and}\quad
   \mathcal{S}^a(u)=\bigcup_{j=0}^{n-1}\mathcal{S}^a_j(u),
  $$
  where both $\mathcal{S}^*_j(u)$ and $\mathcal{S}^a_j(u)$ are contained in a countable union of $j$-dimensional $C^1$ manifolds.
 \end{theorem}

A key step will be the complete classification of the spectrum of the tangent field. Recently, in \cite{toroengel} the authors studied the geometry of sets that admit arbitrarily good local approximations
by zero sets of harmonic polynomials. In the light of the previous Theorem, it would be interesting to adapt their strategies to our degenerate-singular framework.\\

In the last Section of this paper, we  provide applications of our theory in the context of nonlocal elliptic equations. In particular, inspired by \cite{CS2007,conformal,stingatorrea}, we exploit the local realisation of the fractional Laplacian, and more generally of fractional power of divergence form operator $L$ with Lipschitz leading coefficient, in order to study the structure and the regularity of the nodal set of $(-L)^s$-harmonic functions, for $s \in (0,1)$. More precisely, we combine the extension technique with a geometric reduction introduced in
\cite{MR0140031} and exploited in the seminal papers \cite{MR833393,MR882069}. This  will allow us to extend our analysis to fractional powers $(-\Delta_g)^s$ of the Laplace-Beltrami operator on a Riemannian
manifold $M$, also for the case of Lipschitz metrics $g$. Our techniques are quite robust and, we believe, can apply to a wider class of operators on manifolds like the conformally covariant ones of fractional order introduced via the scattering work in \cite{GZ} and reformulated in \cite{conformal} via a suitable extension property on some asymptotic hyperbolic manifolds. However, the equations under consideration in \cite{conformal} involve curvature terms that are to be controlled and this is probably not a completely trivial adaptation of our techniques in this context.  Finally, in \cite{audrito} the authors developed a similar analysis in the context of solutions to a class of nonlocal parabolic equations involving fractional powers of the Heat operator.\\

Our results show some genuinely nonlocal features in the Taylor expansion of $(-L)^s$-harmonic functions near their zero set and their deep impact on the structure of the nodal set itself. We prove that the first term of the Taylor expansion of an $(-L)^s$-harmonic function is either an homogeneous harmonic polynomial or any possible homogeneous polynomial. In particular, this implies
\begin{theorem}
Given  $L$, a divergence form operator with Lipschitz leading coefficients, and $s\in (0,1)$, let $u$ be $(-L)^s$-harmonic in $B_1$. Then there holds
  $$
  \mathcal{S}(u)= \mathcal{S}^*(u) \cup \mathcal{S}^s(u)
  $$
  where $\mathcal{S}^*(u)$ is contained in a countable union of $(n-2)$-dimensional $C^1$ manifolds and $\mathcal{S}^s(u)$ is contained in a countable union of $(n-1)$-dimensional $C^1$ manifolds. Moreover
  $$
   \mathcal{S}^*(u)=\bigcup_{j=0}^{n-2} \mathcal{S}^*_j(u)\quad\mbox{and}\quad
   \mathcal{S}^s(u)=\bigcup_{j=0}^{n-1}\mathcal{S}^s_j(u),
  $$
  where both $\mathcal{S}^*_j(u)$ and $\mathcal{S}^s_j(u)$ are contained in a countable union of $j$-dimensional $C^1$ manifolds.
\end{theorem}
Finally, we prove what can be seen as the nonlocal counterpart of a conjecture that Lin proposed in \cite{MR1090434}. Following his strategy, we give an explicit estimate on the $(n-1)$-Hausdorff measure of the nodal set $\Gamma(u)$ of $s$-harmonic functions in terms of the Almgren frequency previously introduced. We have

\begin{theorem}\label{barLin}
    Given $s\in (0,1)$, let $u$ be an $s$-harmonic function in $B_1$ and $0 \in \Gamma(u)$. Then
  $$
  \mathcal{H}^{n-1}\left(\Gamma(u)\cap B_{\frac{1}{2}}\right) \leq C(n,s) N,
  $$
where $v$ is the $L_a$-harmonic extension of $u$ in $B_1^+$ and $N=N(0,v,1)$ is the frequency defined by
  $$
  N=\ddfrac{\int_{B_1^+}{\abs{y}^a \abs{\nabla v}^2\mathrm{d}X}}{\int_{\partial B_1^+}{\abs{y}^a v^2\mathrm{d}\sigma}}.
  $$
\end{theorem}

The paper is organized as follows. In Section \ref{sec.dec} we prove some preliminary result about $L_a$-harmonic functions. Next, in Section \ref{sec.Alm}, we prove the validity of an Almgren's type monotonicity formula which allows in Section \ref{section.blowup} to prove the existence of blow-up limit in every point of the nodal set $\Gamma(u)$.\\
Finally, in Section \ref{sec.weiss} we prove a Weiss type monotonicity formula, which allows to introduce the notion of tangent map and tangent field at every point of the nodal set. In Section \ref{sec.strat} we
present some useful result on the stratification of the nodal set and finally in Section \ref{singular.sec} we prove a general result on the regularity of the whole nodal set $\Gamma(u)$ and on its restriction on the characteristic manifold $\Sigma$.
The last two Sections are devoted to the applications of the previous results to solutions of fractional powers of divergence form operator, with Lipschitz leading coefficient. In particular, in Section \ref{section.diverg} we apply our technique in order to study the nodal set of $s$-harmonic function and, more generally, of solutions of $(-L)^s$ operators and more general nonlocal equations, and in Section \ref{sec.measure} we give a new estimate of the Hausdorff measure of the nodal set of $s$-harmonic functions.

\section{Decomposition of $\uppercase{L}_a$-harmonic functions}\label{sec.dec}

In this Section we state some general results on $L_a$-harmonic function and we introduce some basic auxiliary concepts that will be often use through the paper in order to describe the structure of the nodal set $\Gamma(u)$.
In particular, given the  definition of characteristic manifold $\Sigma$ for a degenerate-singular operator, we consider the  decomposition of $L_a$-harmonic function with respect to the orthogonal direction to $\Sigma$,
which will turn out to be crucial in proving regularity of $L_a$-harmonic functions. It is worthwhile pointing put here that our theory is linear in nature because of this decomposition between symmetric and anti-symmetric solutions w.r.t. the characteristic manifold. Even though this can be seen as a major drawback, it still provides us a complete picture of the nodal set for our equation. We refer also to the work \cite{CNV} where the authors consider also linear equations, introducing a quantitative stratification technique.\\

To start with, we need recalling the already cited  pioneering works \cite{fkj,fks}, where the authors introduced a class of degenerate-singular operator strictly correlated to some weighted Sobolev spaces with Muckenhoupt $A_p$-weights. 
In \cite[Section 2]{fks} they gave six general properties that the weight must satisfy in order to have existence of weak solutions, Sobolev embeddings, Poincar\'e inequality, Harnack inequality, local solvability in H\"older spaces and estimates on the Green's function and in particular they found a sufficient condition in the definition of the Muckenhoupt $A_2$-class. Hence, they introduced for $a \in (-1,1)$ the weighted Sobolev spaces $H^{1,a}(B_1)$ as the closure of $C^{\infty}(\overline{B_1})$ functions under the norm
$$
\norm{u}{H^{1,a}(B_1)}^2 = \int_{B_1}{\abs{y}^a u^2\mathrm{d}X}+ \int_{B_1}{\abs{y}^a \abs{\nabla u}^2\mathrm{d}X}.
$$
Anyway, as the authors in \cite{fks} pointed out in the study of a special classes of elliptic problem associated with quasi-conformal maps, properties as the Sobolev embeddings, Poincar\'e inequality, Harnack inequality and local solvability in H\"older spaces still hold for every $a \in (-1,+\infty)$. Thus, the following definition is  consistent  for every $a \in (-1,+\infty)$.
\begin{definition}[\cite{fks}]\label{energy}
  Given $F=(f_1,\cdots,f_n)$ on $B_1$ such that $\abs{F} \in L^{2,-a}(B_1)$, we say that $u\in H^{1,a}(B_1)$ is a (energy) solution of $L_a u = \mbox{div}F$ if for every $\varphi \in C^\infty_c(B_1)$ we have
  $$
  \int_{B_1}{\abs{y}^a \langle \nabla u , \nabla \varphi\rangle \mathrm{d}X} = \int_{B_1}{\langle F,\nabla\varphi \rangle \mathrm{d}X}.
  $$
\end{definition}
Similarly, a function $u \in H^{1,a}(B_1)$ is said to be $L_a$-harmonic in $B_1$ if for every $\varphi \in C^\infty_c(B_1)$ we have
$$
\int_{B_1}{\abs{y}^a \langle \nabla u, \nabla \varphi \rangle \mathrm{d}X}=0.
$$

Since the operator $L_a$ is uniformly elliptic on every compact subset of $\R^{n+1}\setminus \Sigma$,  we shall concentrate on the study of the nodal set near the \emph{characteristic manifold $\Sigma$} associated with $L_a$. This remark explains why, throughout the paper, we will focus on the case $X_0 \in \Sigma$ and we will simply compare the result on $\Sigma$ with the case $\R^{n+1}\setminus \Sigma$, avoiding all the technical details.\\

%

In order to better understand the structure of the nodal set and the local behaviour of  $L_a$-harmonic functions, it is convenient to decompose them into their symmetric and antisymmetric parts with respect to the orthogonal direction to $\Sigma$. For the sake of precision,  $u$ is said to be \emph{symmetric} with respect to $\Sigma$ if
  $$
  u(x,-y) = u(x,y) \quad \mbox{in }\R^{n+1}.
  $$
  Conversely, the function $u$ is said to be \emph{antisymmetric} with respect to $\Sigma$ if
  $$
  u(x,-y) = -u(x,y) \quad \mbox{in }\R^{n+1}.
  $$
We shall denote
$$
u_e(x,y)=\frac{u(x,y)+u(x,-y)}{2} \quad\mbox{and} \quad u_o(x,y)=\frac{u(x,y)-u(x,-y)}{2}
$$
the symmetric and antisymmetric (with respect to $\Sigma$) parts of $u$ such that
$$
u(X) = u_e(X) + u_o(X).
$$
Note that, $u$ is $L_a$-harmonic if and only if so are $u_e$ and $u_o$.\\

At first sight, the previous decomposition seems to be harmless and disconnected from the degenerate-singular character of the operator, but with the following Propositions it would be clear the complete picture of how the presence of a set where the coefficients take value zero of infinite affect the local behaviour of the solutions.
Next proposition shows that it is enough to  characterise only the blow-up limits of the symmetric $L_a$-harmonic functions.
\begin{proposition}\label{even.odd}
  Let $a \in (-1,1)$ and $u$ be an $L_a$-harmonic function in $B_1$ antisymmetric with respect to $\Sigma$. Thus, there exists a unique $L_{2-a}$-harmonic function $v$ symmetric with respect to $\Sigma$ such that
  $$
  u(x,y)=v(x,y)y\abs{y}^{-a}\quad \mbox{in }\R^{n+1}.
  $$
\end{proposition}
\begin{proof}
Given $v(x,y)=u(x,y)\abs{y}^a y^{-1}$, let us first prove that $v \in H^{1,2-a}(B_1)$, where $2-a \in (1,3)$. By direct computations we obtain
\begin{equation}\label{L2norm.even.odd}
  \int_{B_1}{\abs{y}^{2-a}v^2\mathrm{d}X}  = \int_{B_1}{\abs{y}^{a}u^2\mathrm{d}X},
\end{equation}
and similarly
\begin{align*}
  \int_{B_1}{\abs{y}^{2-a}\abs{\nabla v}^2\mathrm{d}X} & = \int_{B_1}{\abs{y}^{a}\abs{\nabla u}^2\mathrm{d}X}+(a-1)^2\int_{B_1}{\abs{y}^{a}\frac{u^2}{y^2}\mathrm{d}X}\\
  &\leq C \left(\int_{B_1}{\abs{y}^{a} u^2\mathrm{d}X}+\int_{B_1}{\abs{y}^{a}\abs{\nabla u}^2\mathrm{d}X}\right),
\end{align*}
where in the last inequality we used the validity of an Hardy type inequality (see \cite{kufner}). Next, let us prove that $v$ is $L_{2-a}$-harmonic in $B_1$ in the sense of Definition \ref{energy}. At first, observe that, for $X \in B_1\setminus \Sigma$, we have
\begin{equation}\label{equali}
L_{2-a}v =\mbox{div}(\abs{y}^{2-a}\nabla v)=(a-1)\partial_y u + \mbox{div}(y\nabla u) = y\abs{y}^{-a} L_{a}u.
\end{equation}
 For every $\varphi \in C^\infty_c(B_1)$ and $0<\delta<1$, let $\eta_\delta\in C^{\infty}(B_1)$ be a family of compactly supported cut-off functions such that $0\leq\eta_\delta\leq 1$ and
$$
\eta_\delta(x,y) =
\begin{cases}
  0 & \mbox{ on } \{(x,y) \in B_1 \colon \abs{y}\leq \delta\}, \\
  1 & \mbox{ on }\{(x,y) \in B_1 \colon \abs{y}\geq 2\delta\},
\end{cases}
$$
with $\abs{\nabla \eta_\delta} \leq 1/\delta$.
Thus, by testing \eqref{equali} with $\varphi \eta_\delta$ we obtain, for every $\delta \in (0,1)$,
\begin{align*}
  \int_{B_1}{\abs{y}^{2-a} \langle \nabla v, \nabla(\eta_\delta \varphi)\rangle \mathrm{d}X} & = -\int_{B_1}{\eta_\delta \varphi L_{2-a}v\mathrm{d}X} \\
   & = -\int_{B_1}{\left(y\abs{y}^{-a}\eta_\delta \varphi\right) L_{a}u\mathrm{d}X} = 0,
\end{align*}
where in the last equality we used that $y\abs{y}^{-a}\eta_\delta \varphi \in C^\infty_c(B_1)$. Now,  integration by parts yields
\begin{equation}\label{integration}
\int_{B_1}{\abs{y}^{2-a} \langle \nabla v, \nabla(\eta_\delta \varphi)\rangle \mathrm{d}X} = \int_{B_1}{\abs{y}^{2-a} \eta_\delta \langle \nabla v, \nabla \varphi\rangle \mathrm{d}X}+ \int_{B_1}{\abs{y}^{2-a} \varphi \langle \nabla v, \nabla \eta_\delta \rangle \mathrm{d}X},
\end{equation}
whereas by the the Dominated convergence we obtain that
$$
\lim_{\delta\to 0^+} \int_{B_1}{\abs{y}^{2-a} \eta_\delta\langle \nabla v, \nabla \varphi\rangle \mathrm{d}X} = \int_{B_1}{\abs{y}^{2-a} \langle \nabla v, \nabla \varphi\rangle \mathrm{d}X}
$$
and by H\"older inequality
\begin{align*}
\left|\int_{B_1}{\abs{y}^{2-a} \varphi \langle \nabla v, \nabla \eta_\delta \rangle \mathrm{d}X}\right| &\leq  \norm{\varphi}{L^\infty(B_1)}\left(\int_{B_1}{\abs{y}^{2-a} \abs{\nabla v}^2\mathrm{d}X}\right)^{1/2}\left( \int_{B_1}{\abs{y}^{2-a}\abs{\nabla \eta_\delta}^2\mathrm{d}X}\right)^{1/2}\\
&\leq  C\norm{\varphi}{L^\infty(B_1)}\norm{v}{H^{1,a}(B_1)}\frac{1}{\delta}\left( \int_{\delta}^{2\delta}{\abs{y}^{2-a}\mathrm{d}y}\right)^{1/2}\\
&\leq  C\norm{\varphi}{L^\infty(B_1)}\norm{v}{H^{1,a}(B_1)}\left(\frac{2^{3-a}-1}{3-a}\right)^{1/2}\delta^{\frac{1-a}{2}},
\end{align*}
which imply, passing through $\delta \to 0$ in \eqref{integration}, that
$$
\int_{B_1}{\abs{y}^{2-a}\langle \nabla v, \nabla \varphi \rangle \mathrm{d}X} = 0 \quad \mbox{ for }\varphi \in C^\infty_c(B_1),
$$
since we are dealing with $a<1$.
\end{proof}
As a straightforward consequence, for $a \in (-1,1)$ and every $L_a$-harmonic function $u \in H^{1,a}(B_1)$ there exist $u_e^a \in H^{1,a}(B_1)$ and $u_e^{2-a}\in H^{1,2-a}(B_1)$ two symmetric function with respect to $\Sigma$ respectively $L_a$ and $L_{2-a}$ harmonic in $B_1$ such that
\begin{equation}\label{decompos}
  u(X)=u_e^a(X)+u_e^{2-a}(X)y\abs{y}^{-a}\quad\mbox{in }B_1.
\end{equation}
Thus, through the following Sections we will restrict the classification of the blow-up limit, i.e. the entire homogenous $L_a$-harmonic functions, to the symmetric with respect to $\Sigma$ and in the final part of the work we will collect all the result for a generic $L_a$-harmonic function.\\
Secondly, the previous decomposition combined with the following result gives a complete picture of the regularity of an $L_a$-harmonic function.
\begin{proposition}[{\cite{tesivita}}]\label{smooth}
  Let $a \in (-1,1)$ and $u$ be an $L_a$-harmonic function in $B_1$. Then there holds:
  \begin{itemize}
    \item if $u$ is symmetric with respect to $\Sigma$, we have $u \in C^{1,\alpha}_\loc(B_1)$, for any $\alpha \in (0,1)$;
    \item if $u$ is antisymmetric with respect to $\Sigma$,  we have $u \in C^{0,\alpha}_\loc(B_1)$, for any $\alpha \in (0,\alpha^*)$ with $\alpha^*=\min\{1,1-a\}$.
  \end{itemize}
  Moreover, if $a \in (-1,+\infty)$ and $u$ is symmetric with respect to $\Sigma$, we  have that $u \in C^\infty_\loc(B_1)$.
\end{proposition}
\begin{proposition}[{\cite{tesivita}}]
Let $a \in (-1,1)$ and $u$ be $L_a$-harmonic in $B_1$. Then, for every $i=1,\dots,n$, we have that $\partial_{x_i} u$ is $L_a$-harmonic in $B_1$, while $\partial^a_y u$ is $L_{-a}$-harmonic in $B_1$, where
$$
\partial^a_y u = \begin{cases}
                   \abs{y}^a \partial_y u & \mbox{if } X \not \in \Sigma\\
                   \lim_{y \to 0} \abs{y}^a \partial_y u(x,y) & \mbox{if }X \in \Sigma
                 \end{cases}.
$$
\end{proposition}
These results have been recently obtained in \cite{tesivita,sttv} using
some new approximation technique and Liouville type theorem for a class of degenerate-singular elliptic problems.
\\
We recall here some general result about $L_a$-harmonic functions.
First we introduce the following Caccioppoli inequality, which enables us to give a priori estimates of the $L^{2,a}$ norm of the derivatives of the solution $u$ in terms of the $L^{2,a}$-norm of u.
\begin{proposition}[\cite{fks}]\label{caccioppoli}
  Let $a \in (-1,1)$ and $u$ an $L_a$-harmonic function in $B_1$. Then, for each $X_0 \in B_1\cap \Sigma$ and $0<r<R\leq 1-\abs{X_0}$ we have
  \begin{equation}\label{caccio}
      \int_{B_r(X_0)}{\abs{y}^a \abs{\nabla u}^2\mathrm{d}X}\leq \frac{C}{(R-r)^2}\int_{B_R(X_0)\setminus B_r(X_0)}{\abs{y}^a\abs{u-\lambda}^2\mathrm{d}X},
  \end{equation}
  for every $\lambda \in \R$.
\end{proposition}
Now, for $a\in(-1,1)$ let us fix
$$
\abs{S^{n}}_a=\int_{\partial B_1}{\abs{y}^a\mathrm{d}\sigma},
$$
which implies
$$
\abs{S^n_r}_a = r^{n+a} \abs{S^n}_a \quad \mbox{and}\quad\abs{B^{n+1}_r}_a = \frac{r^{n+a+1}}{n+a+1}\abs{S^n}_a.
$$
\begin{lemma}[{\cite[Lemma A.1]{ww}}]
Let $a\in(-1,1)$ ad $u$ be an $L_a$-harmonic function in $B_1$. Then, for each ball $B_r(X_0)$, with $X_0 \in \Sigma$ and $r \in (0,1-\abs{X_0})$, we have
$$
u(X_0)= \frac{1}{\abs{S^n}_a r^{n+a}}\int_{\partial B_r(X_0)}{\abs{y}^a u\mathrm{d}\sigma} = \frac{1}{\abs{B^{n+1}}_a r^{n+a+1}}\int_{B_r(X_0)}{\abs{y}^a u\mathrm{d}X}.
$$
\end{lemma}
We remark that in the case of $L_a$-subharmonic function, i.e. $-L_a u \leq 0$, the previous result holds true in the form of inequality. Finally, by standard Moser's iteration technique, the following bound holds true.
\begin{lemma}[{\cite[Lemma A.2.]{ww}}]\label{moser}
Let $a \in (-1,1)$ and $u$ be a $L_a$-subharmonic function in $B_1$. Then, for $X_0 \in B_1\cap\Sigma$ and $r \in (0,1-\abs{X_0})$ we obtain
$$
\norm{u}{L^\infty (B_{r/2}(X_0))} \leq C(n,a)\left( \frac{1}{r^{n+1+a}}\int_{B_r(X_0)}{\abs{y}^a u^2\mathrm{d}X}\right)^{1/2},
$$
where $C(n,a)$ is a constant depending only on $n$ and $a$.
\end{lemma}
\section{Almgren  monotonicity formula}\label{sec.Alm}
In this Section, following \cite{CS2007}, we introduce the degenerate-singular counterpart of the classical Almgren monotonicity formula for harmonic functions (cfr \cite{almgren}). Since we intend to investigate the structure and regularity of the nodal set of $L_a$-harmonic function near the characteristic manifold $\Sigma$, let us fix $X_0 = (x_0,0) \in \Sigma$. Hence, for every $r \in (0,R)$, where $R>0$ will be defined later, consider
\begin{align*}
  E(X_0,u,r) & = \frac{1}{r^{n+a-1}}\int_{B_r(X_0)}{\abs{y}^a \abs{\nabla u}^2 \mathrm{d}X} \\
  H(X_0,u,r) & = \frac{1}{r^{n+a}}\int_{\partial B_r(X_0)}{\abs{y}^a u^2 \mathrm{d}\sigma}
\end{align*}
and the Almgren  quotient (also called frequency),
\begin{equation}\label{Almgren.a}
N(X_0,u,r)= \ddfrac{E(X_0,u,r)}{H(X_0,u,r)}= \ddfrac{r\int_{B_r(X_0)}{\abs{y}^a \abs{\nabla u}^2 \mathrm{d}X} }{\int_{\partial B_r(X_0)}{\abs{y}^a u^2 \mathrm{d}\sigma}}.
\end{equation}
Since $u \in H^{1,a}_{\loc}(B_1)$, both the functions $r\mapsto E(X_0,u,r)$ and $r\mapsto H(X_0,u,r)$ are locally absolutely continuous on $(0,+\infty)$, that is that both their derivatives are in $L^1_{\loc}((0,+\infty))$.
 \begin{proposition}[\cite{CS2007}]\label{Almgren.formula}
  Let $a\in(-1,1)$ and $u$ be an $L_a$-harmonic function on $B_1$. Then, for every $X_0\in B_1 \cap \Sigma$ we have that the map $r \mapsto N(X_0,u,r)$ is absolutely continuous and monotone nondecreasing on $(0,1-\abs{X_0})$.\\
  Hence, there always exists finite the limit
$$
N(X_0,u,0^+) = \lim_{r\to 0^+}N(X_0,u,r) = \inf_{r>0}N(X_0,u,r),
$$
to which we will refer as the (Almgren) limiting frequency. Moreover,
\begin{equation}\label{doubling1}
\frac{d}{dr} \log H(X_0,u,r) = \frac{2}{r}N(X_0,u,r).
\end{equation}


\end{proposition}

  As a direct consequence of the monotonicity formula, we infer that the Almgren limiting frequency map $X\mapsto N(X,u,0^+)$ on $\Sigma$ is upper semi-continuous since it is defined as the infimum of a continuous family of functions.    Another simple consequence of the monotonicity result and \eqref{doubling1} is the following comparison property (which, with $r_2 = 2r_1$, is the so called doubling property, which is classical).
    \begin{corollary}\label{doubling.corollary.Sigma}
      Let $a \in (-1,1)$ and $u$ be $L_a$-harmonic on $B_1$. Hence, there given $N=N(X_0,u,1-\abs{X_0})$ such that for every $X_0 \in B_1 \cap \Sigma$,
      $$
      H(X_0,u,r_2) \leq H(X_0,u,r_1) \left(\frac{r_2}{r_1} \right)^{2N}
      $$
      for $0<r_1<r_2<1-\abs{X_0}$.
    \end{corollary}
     In other words, for every $X_0 \in B_1\cap \Sigma$
     $$
     \frac{1}{R^{n+a}}\int_{\partial B_R(X_0)}{\abs{y}^a u^2\mathrm{d}\sigma}\leq \left(\frac{R}{r}\right)^{2N}\frac{1}{r^{n+a}}\int_{\partial B_r(X_0)}{\abs{y}^a u^2\mathrm{d}\sigma}
     $$
     with $0<r<R<1-\abs{X_0}$ and $N=N(X_0,u,1-\abs{X_0})$, and integrating the previous inequality we obtain
     \begin{equation}\label{doublin.ball}
     \frac{1}{R^{n+a+1}}\int_{B_R(X_0)}{\abs{y}^a u^2\mathrm{d}X}\leq \left(\frac{R}{r}\right)^{2N-1}\frac{1}{r^{n+a+1}}\int_{ B_r(X_0)}{\abs{y}^a u^2\mathrm{d}X}.
     \end{equation}
     In order to justify the analysis of the local behaviour of $L_a$-harmonic functions, we prove the validity of the strong unique continuation property for the degenerate-singular operator $L_a$. In general, a function $u$, is said to
\emph{vanish of infinite order} at a point $X_0 \in \Gamma(u)$ if
$$
\int_{\abs{X-X_0}<r}{u^2\mathrm{d}X} = O(r^k), \quad \mbox{for every }k\in \N,
$$
as $r\to 0$. Given an elliptic operator $L$, $L$ is said to have the \emph{strong unique continuation property} in $B_1$ if the only solution of $Lu = 0$ in $H^{1}_\loc(B_1)$ which vanishes of infinite
order at a point $X_0 \in \Gamma(u)$ is $u = 0$. Moreover, $L$ is said to have the \emph{unique continuation
property} in $B_1$ if the solution of $Lu = 0$ in $H^{1}_\loc(B_1)$ which can vanish in an open
subset of $B_1$ is $u = 0$. (see \cite{MR833393,MR882069} for more details for the uniformly elliptic case).
\begin{corollary}[{\cite[Theorem 1.4]{MR882069}}]\label{unique.continuation}
       Let $a \in (-1,1)$ and $u$ be $L_a$-harmonic in $B_1$. Then $u$ cannot vanish of infinite order at $X_0 \in \Gamma(u)\cap B_1$ unless $u\equiv 0$ in $B_1$.
 \end{corollary}
 In \cite{MR882069} the authors stated the proof for analytic nonnegative weights and pointed out the validity for more general, even degenerate, weighted elliptic equations.\\
     The previous result implies that the nodal set $\Gamma(u)$ has empty interior in $\R^{n+1}$. Hence, as a consequence of our blow-up analysis, we will prove a posteriori unique continuation property for the restriction of $\Gamma(u)$ on $\Sigma$. Using \eqref{doubling1} and integrating fro $0$ to $R$, one easily obtains
          \begin{corollary}
       Let $a \in (-1,1)$ and $u$ be an $L_a$-harmonic function on $B_1$. Then, for every $X_0 \in B_1\cap \Sigma$ given $R=1-\abs{X_0}$ we have
       $$
       \frac{1}{n+a+1+2N}\int_{\partial B_R(X_0)}{\abs{y}^a u^2\mathrm{d}\sigma} \leq \int_{B_R(X_0)}{\abs{y}^a u^2\mathrm{d}X} \leq \frac{1}{n+a+1}\int_{\partial B_R(X_0)}{\abs{y}^a u^2\mathrm{d}\sigma},
       $$
       where $N=N(X_0,u,R)$.
     \end{corollary}

     The following result can be viewed as the degenerate-singular counterpart of  \cite[Theorem 1.6]{MR2351641}, which gives us a sufficient condition for the presence of the nodal set in the unitary ball.
     \begin{proposition}\label{import}
       Let $a \in (-1,1)$ and $u$ be an $L_a$-harmonic function on $B_1$. Then, for any $R\in (0,1)$ there exists $N_0=N_0(R)\ll 1$ such that the following holds:
       \begin{enumerate}
         \item if $N(0,u,1)\leq N_0$, then $u$ does not vanish in $B_R$;
         \item if $N(0,u,1) > N_0$, then
         $$
         N\left(X_0,u,\frac{1-R}{2}\right) \leq C N(0,u,1) \quad\mbox{for any }X_0 \in B_R \cap \Sigma,
         $$
         where $C$ is a positive constant depending only on $n,a$ and $R$.
       \end{enumerate}
       Moreover, the vanishing order, i.e. the Almgren frequency formula, of $u$ at any point of $B_R$ never exceeds $C N(0,u,1)$.
     \end{proposition}

    This proof follows directly from the one in \cite{MR2351641,MR1090434}, so we omit it.

    \begin{corollary}\label{lower.N}
      Let $u$ be $L_a$-harmonic on $B_1$, then for every $X_0 \in \Gamma(u) \cap \Sigma$ we have
      \begin{equation}\label{eqlower.N}
        N(X_0,u,0^+)\geq \min\{1,1-a \}.
      \end{equation}
      More precisely
      \begin{itemize}
        \item if $u$ is symmetric with respect to $\Sigma$, we have $N(X_0,u,0^+)\geq 1$,
        \item if $u$ is antisymmetric with respect to $\Sigma$ we have $N(X_0,u,0^+)\geq 1-a$.
      \end{itemize}
    \end{corollary}
    \begin{proof}
    This result follows by Proposition \ref{smooth}. More precisely, et $\alpha^* = \min\{1,1-a \}$ be the coefficient of optimal H\"older regularity for $L_a$-harmonic function, and suppose by contradiction that \eqref{eqlower.N} is not satisfied.\\ Since the limit $N(X_0,u,0^+)$ exists, we obtain the existence of $R>0$ and $\varepsilon>0$ such that $N(X_0,u,r) \leq \alpha^* -\varepsilon$ for all $0\leq r\leq R$. By \eqref{doubling1}, up to consider a smaller interval of $(0,R)$, we have
      $$
     \frac{d}{dr} \log H(X_0,u,r) = \frac{2}{r}N(X_0,u,r)\leq \frac{2}{r}(\alpha^*-\varepsilon).
      $$
      Integrating this inequality between $r$ and $R$ yields
      $$
      \frac{H(X_0,u,R)}{H(X_0,u,r)}\leq \left(\frac{R}{r} \right)^{2(\alpha^* - \varepsilon)}
      $$
      which, together with the fact that $u$ is $\alpha^*$-H\"older continuous and $u(X_0)=0$, implies
      $$
      C_1 r^{2(\alpha^*-\epsilon)}\leq H(X_0,u,r)\leq C_2 r^{2\alpha^*}.
      $$
      The contradiction follows for small value of $r>0$. If an addition we suppose that $u$ is symmetric or antisymmetric with respect to $\Sigma$, we obtain respectively that $u$ is Lipschitz continuous or $(1-a)$-H\"older continuous, and the lower bound on the Almgren frequency formula follows immediately.
    \end{proof}

Now we focus our attention to the balls centred outside the characteristic manifold; for $X_0 \in B_1\setminus \Sigma$, i.e. $y_0 \neq 0$, and $r\in (0,\abs{y_0})$, let us consider, as usual,

\begin{align*}
 E(X_0,u,r) =& \frac{1}{r^{n-1}}\int_{B_r(X_0)}{\abs{y}^a\abs{\nabla u}^2 \mathrm{d}X},\\
 H(X_0,u,r) =& \frac{1}{r^{n}}\int_{\partial B_r(X_0)}{\abs{y}^a u^2 \mathrm{d}\sigma},
\end{align*}
and consequently the associated Almgren frequency
\begin{equation}\label{Almgren.a}
N(X_0,u,r)= \ddfrac{E(X_0,u,r)}{H(X_0,u,r)}= \ddfrac{r\int_{B_r(X_0)}{\abs{y}^a \abs{\nabla u}^2 \mathrm{d}X} }{\int_{\partial B_r(X_0)}{\abs{y}^a u^2 \mathrm{d}\sigma}}.
\end{equation}

Now we prove the validity of a perturbed monotonicity formula also in the case of balls centred outside $\Sigma$. Note that,  by  the results in \cite{MR833393,MR882069} we already know the validity of an Almgren type monotonicity formula associated with uniformly elliptic operator with Lipschitz leading coefficients.  For completeness, we provide below a slightly different monotonicity result which holds for every $X_0 \in \R^{n+1}$,  bypassing the change of coordinates introduced in \cite{MR833393, MR882069}.
 \begin{proposition}\label{Almgren.formula.fuori}
  Let $a\in(-1,1)$ and $u$ be an $L_a$-harmonic function on $B_1$. Then, for every $X_0\in B_1 \setminus \Sigma$ there exists $C>0$ such that $r\mapsto e^{Cr}N(X_0,u,r)$ is absolutely continuous and monotone nondecreasing on $(0,\abs{y_0})$.\\
  Hence, there always exists finite the limit
$$
N(X_0,u,0^+) = \lim_{r\to 0^+}N(X_0,u,r),
$$
which we will call as the \emph{Almgren type frequency formula}.
\end{proposition}
\begin{proof}
The strategy of the proof is similar to the one for the case $X_0 \in B_1\cap \Sigma$. By \eqref{deriv.num} and \eqref{deriv.dem}, we already know that passing to the logarithmic derivatives we obtain from the Cauchy-Schwarz inequality
\begin{align*}
\frac{d}{dr}\log{N(X_0,u,r)} =& \frac{1}{r}+\ddfrac{\frac{d}{dr}\int_{B_r(X_0)}{\abs{y}^a \abs{\nabla u}^2 \mathrm{d}X} }{\int_{B_r(X_0)}{\abs{y}^a \abs{\nabla u}^2 \mathrm{d}X} } - \ddfrac{\frac{d}{dr}\int_{\partial B_r(X_0)}{\abs{y}^a u^2 \mathrm{d}\sigma}}{\int_{\partial B_r(X_0)}{\abs{y}^a u^2 \mathrm{d}\sigma}}\\
&\geq \frac{a  y_0\displaystyle\int_{\partial B_r(X_0)}{\frac{\abs{y}^a}{y} u^2\mathrm{d}\sigma}}{r\displaystyle\int_{\partial B_r(X_0)}{\abs{y}^a u^2\mathrm{d}\sigma}}-\frac{a  y_0\displaystyle\int_{B_r(X_0)}{\frac{\abs{y}^a}{y} \abs{\nabla u}^2\mathrm{d}\sigma}}{r\displaystyle\int_{B_r(X_0)}{\abs{y}^a \abs{\nabla u}^2\mathrm{d}\sigma}}
\end{align*}
for $r \in (r_1,r_2)$. This remainders reflect the presence of the weight $\omega(X)=\abs{y}^a$ and its homogeneity with respect to $\Sigma$. Now, for every $r \in (r_1,r_2)$
$$
\frac{d}{dr}\log N(X_0,u,r) \geq
\begin{cases}
\ddfrac{a y_0}{r}\left( \min_{\partial B_r(X_0)}{\frac{1}{y}} -  \max_{B_r(X_0)}{\frac{1}{y}}  \right) & \mbox{ if } a\cdot y_0>0\\
-\ddfrac{a y_0}{r}\left( \min_{B_r(X_0)}{\frac{1}{y}} -  \max_{\partial B_r(X_0)}{\frac{1}{y}}  \right) & \mbox{ if } a \cdot y_0<0
\end{cases}
$$
which is equivalent to
$$
\frac{d}{dr}\log N(X_0,u,r) \geq
-\ddfrac{2\abs{a y_0}}{y_0^2 -r^2}\geq -\ddfrac{2\abs{a y_0}}{y_0^2 -r^2_2} \quad \mbox{for }r \in (r_1,r_2),
$$
from which we learn that necessary $r_2<\abs{y_0}$. Consider now
$$
H(X_0,u,r)= \frac{1}{r^{n-1}}\int_{\partial B_r(X_0)}{\abs{y}^a u^2 \mathrm{d}\sigma}
$$
such that
$$
\frac{d}{dr}\log H(X_0,u,r)=\frac{2}{r}N(X_0,u,r) +\frac{a}{r}\left(1-y_0\frac{  \displaystyle\int_{\partial B_r(X_0)}{\frac{\abs{y}^a}{y} u^2\mathrm{d}\sigma}}{\displaystyle\int_{\partial B_r(X_0)}{\abs{y}^a u^2\mathrm{d}\sigma}}\right).
$$
Let us prove the existence of the limit of the Almgren frequency formula as $r\to 0^+$, so suppose by contradiction that $r_1 = \inf\{r>0: H(X_0,u,r)>0 \mbox{ on } (r,\abs{y_0})\}>0$ and consider $r \in (r_1,\abs{y_0})$. By the previous inequality, we have that there exists a positive constant $C>0$ such that
$$
r \mapsto e^{Cr} N(X_0,u,r)
$$
is monotone nondecreasing on $(r_1,\abs{y_0})$.
Then, let $r_1<r<2r_1 \leq \abs{y_0}$,  since
\begin{equation}\label{bound.R}
\frac{a  y_0\displaystyle\int_{\partial B_r(X_0)}{\frac{\abs{y}^a}{y} u^2\mathrm{d}\sigma}}{r\displaystyle\int_{\partial B_r(X_0)}{\abs{y}^a u^2\mathrm{d}\sigma}} \geq
\begin{cases}
\ddfrac{1}{r}\ddfrac{a y_0}{y_0+2r_1}& \mbox{if }a\cdot y_0 >0 \\
\ddfrac{1}{r}\ddfrac{a y_0}{y_0-2r_1}& \mbox{if }a\cdot y_0 <0
\end{cases}
\end{equation}
we have
\begin{equation}\label{doubling2}
\frac{d}{dr}\log H(X_0,u,r)
\leq  \frac{2}{r} e^{2C r_1}N(X_0,u,2r_1) 
\end{equation}
By integrating \eqref{doubling2}, it follows
$$
\frac{H(X_0,u,2r_1)}{H(X_0,u,r)} \leq \left(\frac{2r_1}{r} \right)^{2 e^{2Cr_1}N(X_0,u,2r_1)}
$$
and since $r\mapsto H(X_0,u,r)$ is continuous, $H(X_0,u,r_1)>0$ and we seek the contradiction.
%

  \end{proof}
  As before, a simple consequence of the monotonicity result 
  is the following comparison property (which, with $r_2 = 2r_1$, is the so called doubling property).
    \begin{corollary}\label{doubling.corollary}
      Let $u$ be an $L_a$-harmonic function in $B_1$. For every $X_0 \in B_1\setminus \Sigma$, there exists $C>0$ and $R>0$ such that
      $$
      H(X_0,u,r_2) \leq H(X_0,u,r_1) \left(\frac{r_2}{r_1} \right)^{2C}
      $$
      for every $0<r_1<r_2<R$.
    \end{corollary}
     Moreover, since the operator $L_a$ is uniformly elliptic outside $\Sigma$, we can apply the same reasoning using the Lipschitz optimal regularity in $\R^n\setminus \Sigma$ and prove
    \begin{corollary}
      Let $u$ be an $L_a$-harmonic function in $B_1$. For every $X_0 \in \Gamma(u) \setminus \Sigma$ we have $N(X_0,u,0^+)\geq 1$.
    \end{corollary}
\section{Compactness of blow-up sequences }\label{section.blowup}
    All the arguments exposed in the following Sections involve a local analysis of the solutions, which will be performed via a blow-up procedure. Fix $a\in (-1,1)$ and $u$ an $L_a$-harmonic function in $B_1$; let us consider now $X_0 \in \Gamma(u)$ a point on the nodal set of $u$,  and define, for any $r_k \downarrow 0^+$ , the normalized  blow-up sequence as
    \begin{equation*}\label{blow.up.almgren}
    u_k(X)= \ddfrac{u(X_0 + r_k X)}{\sqrt{H(X_0,u,r_k)}}\quad \mbox{for } X \in X\in B_{X_0,r_k}=\frac{B_1 - X_0}{r_k},
    \end{equation*}
    such that $L_a u_k = 0$ and $\norm{u_k}{L^{2,a}(\partial B_1)}=1$. We stress that we will always apply a blow-up analysis centered at point of the nodal set $\Gamma(u)$ on the characteristic manifold $\Sigma$, since as we already remarked the local behaviour of $L_a$-harmonic function is known outside the characteristic manifold.\\

    In this Section we will prove the convergence of the blow-up sequence and we provide the classification of the blow-up limits, starting from the following general convergence result.
    \begin{theorem}\label{blowup.convergence}
    Let $a \in (-1,1)$ and $\alpha^*=\min\{1,1-a\}$. Given $X_0 \in \Gamma(u)\cap \Sigma$ and a blow-up sequence $u_k$ centered in $X_0$ and associated with some $r_k \downarrow 0^+$, there exists $p \in H^{1,a}_{\loc}(\R^n)$ such that, up to a subsequence, $u_k\to p$ in $C^{0,\alpha}_{\loc}(\R^n)$ for every $\alpha \in (0,\alpha^*)$ and strongly in $H^{1,a}_{\loc}(\R^n)$. In particular, the blow-up limit is an entire $L_a$-harmonic function, i.e.
    $$
    L_a p = 0\quad\mbox{ in }\R^{n+1}.
    $$
    \end{theorem}
In particular, the previous result can be easily improved in the case of $L_a$-harmonic function purely symmetric with respect to $\Sigma$. More precisely, as suggested by Proposition \ref{smooth}, in the first case the convergence holds in $C^{1,\alpha}_\loc$ for every $\alpha \in (0,1)$, and this difference relies on the Liouville type theorems introduced in \cite{tesivita,sttv}.\\

The proof will be presented in a series of lemmata.
\begin{lemma}\label{bounded.H1}
  Let $X_0 \in \Gamma(u)\cap\Sigma$. For any given $R>0$, we have  $$\norm{u_{k}}{H^{1,a}(B_R)}\leq C\quad\mbox{and}\quad\norm{u_{k}}{L^\infty(\overline{B_R})}\leq C,
$$
where $C>0$ is a constant independent on $k>0$.
\end{lemma}
\begin{proof}
Let us consider $\rho^2_k = H(X_0,u,r_k)$, then by definition of the blow-up sequence $u_k$  and Corollary \ref{doubling.corollary.Sigma} we obtain
 \begin{align*}
    \int_{\partial B_R}{\abs{y}^a u^2_k \mathrm{d}\sigma} & = \frac{1}{\rho_k^2}\int_{\partial B_{R}}{\abs{y}^a u^2(X_0 + r_k X) \mathrm{d}\sigma} \\
&=\frac{1}{\rho_k^2 r_k^{n+a}}\int_{\partial B_{Rr_k}(X_0)}{\abs{y}^a u^2 \mathrm{d}\sigma}\\
    &= R^{n+a}\frac{H(X_0,u,Rr_k)}{H(X_0,u,r_k) }\\
    & \leq R^{n+a}\left( \frac{R r_k}{r_k}\right)^{\!2\tilde{C}}
\end{align*}
which gives us $\norm{u_k}{L^{2,a}(\partial B_R)}^2\leq C(R) R^{n+a}$. Similarly
\begin{align}\label{c.overline}
\begin{aligned}
\int_{B_{R}}{\abs{y}^a \abs{\nabla u_k}^2 \mathrm{d}\sigma} & = N(0,u_k,R)\frac{1}{R}\int_{\partial B_R}{\abs{y}^a u^2_k \mathrm{d}\sigma}\\
&\leq C(R) R^{n-1+a} N(X_0,u,Rr_k)\\
& \leq C(R) R^{n-1+a} N(X_0,u,R)
\end{aligned}
\end{align}
where in the last inequality we used the monotonicity result of Proposition \ref{Almgren.formula}. Since the map $u_k$ is $L_a$- harmonic, by \cite[Lemma A.2]{ww} we obtain
\begin{align*}
  \sup_{\overline{B_{R/2}}} u_k &\leq C(n,s) \left(\frac{1}{R^{n+1+a}}\int_{B_R}{\abs{y}^a u_k^2\mathrm{d}X} \right)^{1/2} \\
  &\leq   C(n,s) \left(\frac{H(0,u_k,R)}{n+a+1}\right)^{1/2},
\end{align*}
where in the second inequality we used the monotonicity of $r\mapsto H(0,u_k,r)$ in $(0,R)$. Finally, the estimate follows directly from the one the $L^{2,a}(\partial B_R)$-norm.
\end{proof}
So far we have proved the existence of a nontrivial function $p\in H^{1,a}_{\loc}(\R^{n+1})\cap L^\infty_{\loc}(\R^{n+1})$ such that, up to a subsequence, we have $u_k \rightharpoonup p$ weakly in $H^{1,a}_{\loc}(\R^{n+1})$ and $L_a p = 0$ in $\mathcal{D}'(\R^{n+1})$. The next step is to prove that for $X_0 \in \Gamma(u) \cap \Sigma$ the convergence $u_k \to p$ is indeed strong in $H^{1,a}_{\loc}$ and in $C^{0,\alpha}_{\loc}$ for $\alpha \in (0,\alpha^*)$. This is a standard argument based on testing the weak formulation and using the fact that $p$ belongs to $H^{1,a}_{loc}(\mathbb R^n)$.

\begin{lemma}
  For every $R>0$, up to a subsequence, $u_k \to p$ strongly in $H^{1,a}(B_R)$.
\end{lemma}

\begin{lemma}\label{holder}
  For every $R > 0$ there exists $C > 0$, independent of $k$, such that
  $$
  \left[ u_k \right]_{C^{0,\alpha}(B_R)} = \sup_{X_1,X_2 \in \overline{B}_R}\frac{\abs{u(X_1)-u(X_2)}}{\abs{X_1-X_2}^\alpha} \leq C
  $$
  for every $\alpha \in (0,\alpha^*)$.
\end{lemma}
\begin{proof}
    The proof follows essentially the ideas of the similar results in \cite{tvz1, tvz2}:  the critical exponent $\alpha^*=\min\{1,1-a\}$ is related to a Liouville type theorem for $L_a$-harmonic function as given in \cite{tesivita,sttv}.
\end{proof}
If instead we consider the general case $X_0 \in \Gamma(u)$ we can prove, as a direct consequence of the blow-up analysis of \cite[Theorem 3.3]{MR2984134}, the following general result
    \begin{theorem}
     Let $u$ be an $L_a$-harmonic function in $B_1$ and $X_0 \in \Gamma(u)$ a point on its nodal set. Given the blow-up sequence $u_k$ centered in $X_0$ and associated with $r_k \downarrow 0^+$ we have these two cases:
     \begin{enumerate}
     \item if $X_0 \in \Sigma$, there exists $p \in H^{1,a}_{\loc}(\R^n)$ such that $u_k \to p$ in $C^{0,\alpha}_{\loc}(\R^n)$ for every $\alpha \in (0,\alpha^*)$ and strongly in $H^{1,a}_{\loc}(\R^n)$. In particular the blow-up limit solves
         $$
         -L_a p=0 \ \mbox{ in }\  \R^n.
         $$
     \item if $X_0 \not\in \Sigma$, there exists $p \in H^{1}_{\loc}(\R^n)$ such that $u_k \to p$ in $C^{0,\alpha}_{\loc}(\R^n)$ for every $\alpha \in (0,1)$ and strongly in $H^{1}_{\loc}(\R^n)$. In particular the blow-up limit solves
         $$
         -\Delta p=0 \ \mbox{ in }\  \R^n.
         $$
     \end{enumerate}
    \end{theorem}

Next we focus our attention on the blow-up limit itself in the more challenging case $X_0 \in \Gamma(u)\cap \Sigma$ and we investigate the connection between the value of the  limiting frequency and the local behaviour of the solution. More precisely, we have
\begin{proposition}\label{blow.N}
Let $X_0 \in \Gamma(u)\cap \Sigma$ and $p$ be a blow-up limit of $u$ centered in $X_0$, as previously defined. Then $N(0,p,r)=N(X_0,u,0^+)=:k$ for every $r>0$ and $p$ is $k$-homogeneous, i.e.
$$
p(X)= \abs{X}^k p\left(\frac{X}{\abs{X}} \right)\ \mathrm{for}\ \mathrm{every}\ X \in \R^{n+1}.
$$
\end{proposition}
\begin{proof}
  First of all we prove that $r \mapsto N(0,p,r)$ is constant. Let us observe that
$N(0,u_k,r)=N(X_0,u,r r_k)$ and that Theorem \ref{blowup.convergence} yields that $N(0,p,r)=\lim_k N(0,u_k,r)$. Similarly, for the right hand side we obtain $\lim_k N(X_0,u,r r_k )=N(X_0,u,0^+)$ by Proposition \ref{Almgren.formula}.\\
We now compute the derivative of $r\mapsto N(0,p,r)$, in order to prove that $p$ is $k$-homogeneous, where obviously $k=N(X_0,u,0^+)$ is the Almgren frequency formula. As in the proof of Proposition \ref{Almgren.formula}, we know that
$$
\frac{d}{dr} H(0,p,r)= \frac{2}{r^{n+a-1}}\int_{\partial B_r}{\abs{y}^a p\partial_r p \mathrm{d}\sigma}
$$
and by integration by parts that
$$
\frac{d}{dr} E(0,p,r)= \frac{1}{r^{n+a-1}}\int_{\partial B_r}{\abs{y}^a \left(\partial_r p \right)^2 \mathrm{d}\sigma}.
$$
Hence, these two equalities imply
$$
0=\frac{d}{dr}N(0,p,r)= \frac{2}{r^{2n+2a-2}}\frac{1}{H^2(0,p,r)}\left[\int_{\partial B_r}{\abs{y}^a p^2 \mathrm{d}\sigma} \int_{\partial B_r}{\abs{y}^a \abs{\partial_r p}^2 \mathrm{d}\sigma}  - \left(\int_{\partial B_r}{\abs{y}^a p\partial_r p \mathrm{d}\sigma} \right)^2\right]
$$
for $r>0$. This equality yields the existence of $C=C(r)>0$ such that $\partial_r p=C(r)p$ for every $r>0$. Using this fact in \eqref{doubling1} we infer
$$
2C(r)=\frac{\int_{\partial B_r}{\abs{y}^a p \partial_r p \mathrm{d}\sigma}}{\int_{\partial B_r}{\abs{y}^a p^2 \mathrm{d}\sigma}}=\frac{d}{dr}\log{H(0,p,r)}=\frac{2}{r}{N(0,p,r)}=\frac{2}{r}k
$$
and thus $C(r)=k/r$ and $p$ is $k$-homogenous as we claimed.
\end{proof}
In the final part of this Section we classify the possible values of the Almgren limiting frequencies on the restriction $\Gamma(u) \cap \Sigma$ and consequently the possible blow-up limits. This classification will be the key to  understand the structure and the stratification of the nodal set of $u$. A crucial consequence of our analysis is that
the blow-up process discriminates between the symmetric and antisymmetric cases (with respect to $\Sigma$).
\\

At this point, we already know that given an $L_a$-harmonic function $u$ on $B_1$, for every $X_0 \in \Gamma(u) \cap \Sigma$ and $r_k \downarrow 0^+$ we have, up to a subsequence, that
$$
u_k(X)= \ddfrac{u(X_0 + r_k X)}{\sqrt{H(X_0,u,r_k)}} \to p(X),
$$
where $p\in H^{1,a}_{\loc}(\R^{n+1})$ is a nonconstant entire $L_a$-harmonic function homogenous of order $k \in \R$ with $\norm{p}{L^{2,a}(\partial B_1)}=1$. In particular, by Proposition \ref{blow.N} we already know that $k= N(X_0,u,0^+)$.\\
Inspired by Proposition \ref{even.odd}, let us consider separately the case when $u$ is symmetric with respect to $\Sigma$ and the antisymmetric one.
\begin{lemma}\label{gap.even}
  Let $a \in (-1,1)$ and $u$ be an $L_a$-harmonic function symmetric with respect to $\Sigma$. Then, for every $X_0 \in \Gamma(u) \cap \Sigma$, we have $$N(X_0,u,0^+) \in 1+\N.$$
\end{lemma}
\begin{proof}
Let $X_0 \in \Gamma(u) \cap \Sigma$ and $k=N(X_0,u,0^+)$ be the limiting frequency at $X_0$. For every $r_k \to 0^+$ we already know that, up to a subsequence, by Theorem \ref{blowup.convergence} and Proposition \ref{blow.N} that
$$
u_k(X)= \ddfrac{u(X_0 + r_k X)}{\sqrt{H(X_0,u,r_k)}} \to p(X),
$$
where $p$ is an $L_a$-harmonic $k$-homogenous function symmetric with respect to $\Sigma$.\\
Since, by Corollary \ref{lower.N} we already know that $k\geq 1$, let us suppose by contradiction that there exists an homogenous $L_a$-harmonic function of order $k >1$ such that $k \not\in \N$. Since for every $i = 1,\dots ,n$, we have
$$
L_a (\partial_{x_i} p ) = \partial_{x_i} L_a p =0,
$$
fixed $k= \lfloor k\rfloor$, by Euler's homogeneous function Theorem, we already know that any $k$-order partial derivative of $p$ with respect to the variables $x_1,\dots ,x_n$ must be an homogenous $L_a$-harmonic of order $\alpha=k- \lfloor k\rfloor \in (0,1)$. The contradiction follows from Proposition \ref{lower.N}, since the homogeneity of an homogenous function is equal to the Almgren frequency formula evaluated in the origin, hence in the symmetric case it must be greater or equal to 1.
\end{proof}
\begin{lemma}\label{gap.odd}
  Let $a \in (-1,1)$ and $u$ be an $L_a$-harmonic function antisymmetric with respect to $\Sigma$. Then, for every $X_0 \in \Gamma(u) \cap \Sigma$, we have
  $$N(X_0,u,0^+) \in 1-a+\N.$$
\end{lemma}
\begin{proof}
As in the previous Lemma, let $X_0 \in \Gamma(u) \cap \Sigma$ and $k=N(X_0,u,0^+)$ be the limiting frequency at $X_0$. For every $r_k \to 0^+$ we already know that, up to a subsequence, we have by Theorem \ref{blowup.convergence} and Proposition \ref{blow.N} that $u_k \to p$ where $p$ is an $L_a$-harmonic $k$-homogenous function antisymmetric with respect to $\Sigma$.\\
By Proposition \ref{even.odd}, there exists $q \in H^{1,2-a}_\loc(\R^{n+1})$ and $L_{2-a}$-harmonic function symmetric with respect to $\Sigma$, such that $p=q y\abs{y}^{-a}$. Since $p$ is $k$-homogenous, we already know that $q$ must be $(k-1+a)$-homogenous, i.e.
$$
q(X)=p(X)y^{-1}\abs{y}^{a} = \abs{X}^{k-1+a}p\left(\frac{X}{\abs{X}}\right)\frac{y^{-1}\abs{y}^{a}}{\abs{X}^{-1+a}} = \abs{X}^{k-1+a}q\left(\frac{X}{\abs{X}}\right)
$$
for every $X \in \R^{n+1}$.\\
Obviously if $q(0)\neq 0$, then $k=1-a$ and $q$ is zero-homogenous, i.e. $q\equiv q(0)$ on $\R^{n+1}$, instead, if $q(0)=0$, by Lemma \ref{gap.even} we know that $N(0,q,0^+) \in 1+\N$ and consequently $k \in 2-a+\N$. Similarly, since these two cases correspond to $N(0,q,0^+)=0$ and $N(0,q,0^+)\in 1+\N$, the final result on $k$ can be formulated as $N(X_0,u,0^+)\in 1-a+\N$.
\end{proof}
The proof is a direct consequence of the previous Lemmas and we leave it to the readers.
\begin{proposition}\label{gap}
  Let $a \in (-1,1)$ and $u$ be an $L_a$-harmonic function. Given $X_0 \in \Gamma(u)\cap \Sigma$ and a blow-up sequence $u_k$ centered in $X_0$ and associated with some $r_k\downarrow 0^+$. Then the blow-up limit $p \in H^{1,a}_\loc(\R^{n+1})$ is either symmetric or antisymmetric with respect to $\Sigma$ and
  $$
  N(X_0,u,0^+) \in
  \begin{cases}
    1 + \N, & \mbox{if } $p$ \mbox{ is symmetric}, \\
    1-a+\N, & \mbox{if } $p$ \mbox{ is antisymmetric}.
  \end{cases}
  $$
\end{proposition}

In order to understand the local behaviour of our solution,  let's  now proceed with the explicit construction of homogeneous $L_a$-harmonic functions, first examining the  symmetric  ones.
\begin{lemma}\label{yornot}
  Let $p \in H^{1,a}_\loc(\R^{n+1})$ be a nonconstant homogeneous $L_a$-harmonic function, symmetric with respect to $\Sigma$. Then $p$ does not depend on the variable $y$ if and only if it is harmonic in the variable $(x_1,\dots,x_n)\in\R^n$.
\end{lemma}
The proof is trivial and the main consequence is that for every $k\in 1+\N$ an homogenous harmonic function in the variable $x_1, \dots,x_n$ of order $k$ is an admissible blow-up limit. For this reason, let us concentrate our attention on the case of blow-up limits that depend on the variable $y$.
\begin{lemma}\cite[Lemma 2.7]{MR2367025}\label{caff.silv.salsa}
Let $p \in H^{1,a}_{\loc}(\R^{n+1})$ be an entire $L_a$-harmonic function symmetric with respect to $\Sigma$, such that
$$
\abs{p(X)}\leq C \left(1+ \abs{X}^k\right) \quad \mathrm{in}\ \R^{n+1},
$$
for some $k \in \N$. Then $p$ is a polynomial.
\end{lemma}
In order to give an explicit expression of the blow-up limits, we start with the case $n+1=2$, and we remark that if $p$ is a $k$-homogenous $L_a$-harmonic function, then for every $i=1,\dots,n$ the functions $\partial_{x_i}u$ are $(k-1)$-homogeneous $L_a$-harmonic function and $\partial^2_{yy} u +ay^{-1}\partial_y u$ is a $(k-2)$-homogenous $L_a$-harmonic function. More precisely a straightforward computation leads to
\begin{lemma}\label{iteration}
  Let $p\in H^{1,a}_{\loc}(\R^{n+1})$ be an $L_a$-harmonic homogenous polynomial of degree $k\geq2$ symmetric with respect to $\Sigma$. Then $p$ is $L_a$-harmonic if and only if $\partial_{x_i} p$ and $\partial^2_{yy}p + a y^{-1}\partial_y p$ are $L_a$-harmonic, for every $i=1,\dots,n$.
\end{lemma}

The following Proposition gives a complete picture of the possible entire configurations in $\R^2$. This profiles will be useful in the stratification result of Section \ref{singular.sec}.
\begin{proposition}\label{example}
  Let $p \in H^{1,a}_{\loc}(\R^2)$ be a nonconstant entire $L_a$-harmonic function symmetric with respect to $\Sigma$ such that $N(0,p,r)=k$ for every $r>0$. Suppose that $p$ depends on the variable $y$, then if $k \in 2\N$ we have
  \begin{equation}\label{even}
  p(x,y)=\ddfrac{(-1)^{\frac{k}{2}} \Gamma\left(\frac{1}{2}+\frac{a}{2}\right)}{2^{k}\Gamma\left(1+\frac{k}{2}\right)\Gamma\left(\frac{1}{2}+\frac{a}{2}+\frac{k}{2}\right)} \ _2F_1\left(-\frac{k}{2}, -\frac{k}{2}-\frac{a}{2}+\frac{1}{2}, \frac{1}{2}, -\frac{x^2}{y^2}\right)y^{k},
  \end{equation}
and if $k \in 2\N +1$ we obtain
\begin{equation}\label{odd}
  p(x,y)=-\ddfrac{(-1)^{\frac{k}{2}+\frac{1}{2}} \Gamma\left(\frac{1}{2}+\frac{a}{2}\right)}{2^{k-1}\Gamma\left(\frac{1}{2}+\frac{k}{2}\right)\Gamma\left(\frac{a}{2}+\frac{k}{2}\right)} \ _2F_1\left(\frac{1}{2}-\frac{k}{2}, 1-\frac{k}{2}-\frac{a}{2}, \frac{3}{2}, -\frac{x^2}{y^2}\right)xy^{k-1},
\end{equation}
where $_2F_1$ is the hypergeometric function. 
\end{proposition}
\begin{proof}
The proof is  based on the properties related to the derivatives of homogenous $L_a$-harmonic functions and proceeds by induction. By Lemma \ref{caff.silv.salsa}, we already know that every homogenous $L_a$-harmonic function symmetric with respect to $\Sigma$ is a polynomial $p(x,y)$ such that, for every $x \in \Sigma$ the map $y \mapsto p(x,y)$ is a polynomial of even degree.\\
Fix $k=2m$ with $m\in\N$, consider
\begin{equation}\label{coeff}
c(m,a,t) = \frac{(-1)^{m-t}}{2t!}\frac{1}{2^{m-t} (m-t)!}\prod_{i=1}^{m-t}\frac{1}{2i+a-1}= \ddfrac{(-1)^{m-t}\Gamma\left(\frac{1}{2}+\frac{a}{2}\right)}{2t! (m-t)! 2^{2m-2t}\Gamma\left(m-t+\frac{1}{2}+\frac{a}{2}\right)}
\end{equation}
and consequently
$$
p(x,y)=\frac{x^{2m}}{2m!}+\sum_{t=0}^{m-1}{c(m,a,t)x^{2t}y^{2m-2t}}.
$$
which is equivalent to \eqref{even}. By a direct computation, it is easy to see that $L_a p(x,y)=0$ for every $(x,y)\in \R^2$. Now, let us prove by induction on $k\geq 2$ that every homogenous $L_a$-harmonic function is of the form \eqref{even}. Since the case $k=0$ is trivial, let us take $k=2$. Since $p$ must be of degree $2$ and even in the variable $y$, the polynomial must be like $p(x,y)=a_1x^2+a_2y^2$ and consequently
$$
L_a p = 0 \quad \leftarrow \! \rightarrow\quad a_2 = -\frac{1}{1+a}a_1,
$$
and for $a_1=1/2$ we obtain the formula in \eqref{coeff}.\\ Suppose \eqref{coeff} are true for $k \in 2\N$, and consider a $L_a$-harmonic polynomial $p$ of degree $k+2$, i.e.
$$
p(x,y)=a_{m+1} x^{2m+2} + \sum_{t=0}^{m}{a_t x^{2t} y^{2m-2t} }.
$$
Since $\partial^2_x p$ is a $L_a$-harmonic polynomial of degree $k$, we must have by the inductive hypothesis
$$
a_{m+1}(2m+2)(2m+1) =\frac{1}{2m!},\quad 2t(2t-1)a_{t} = c(m,a,t-1)\ \mbox{for }t = 1,\dots,m.
$$
which imply, by definition \eqref{coeff}, that
$$
a_{m+1}=\frac{1}{(2m+2)!},\quad a_{t} = \frac{c(m,a,t-1)}{2t(2t-1)}=c(m+1,a,t)
$$
for $t=1,\dots,m$. Finally, let $w=-\partial^2_{yy}p - a y^{-1}\partial_y p$ be a polynomial of degree $k$. By Lemma \ref{iteration} $w$ is $L_a$-harmonic and, by the inductive hypothesi, we obtain by linearity that
$$
-\partial^2_{yy}(a_0 y^{2m+2}) - a y^{-1}\partial_y (a_0 y^{2m+2}) = c(m,a,0)y^{2m},
$$
or in other words that $-(2m+2)(2m+1+a)a_0 = c(m,a,0)$, which implies that
$$
a_0 = \frac{c(m,a,0)}{2(m+1)(2m+1+a)} = c(m+1,a,0).
$$
We have already proved the formula for the case $k \in 2\N$, while the other one is obtained via an integration respect to the variable  $x$.
\end{proof}
Before considering the general case $n\geq 3$, we complete the Section with some concrete examples of blow-up profiles in $2$-dimensional case. This example, and more generally the class of homogeneous function described by the previous Proposition, will summarize all the possible behaviour of the $(n-2)$-dimensional singular set, as we will see in Section \ref{singular.sec}.
For $n\geq 3$, we can not give an explicit formula for the blow- up limits which depend on the variable $y$, but we can prove that every polynomial in $\R^n$ admits a unique $L_a$-harmonic extension symmetric with respect to $\Sigma$.
Since we want to classify the possible blow-up limit of $s$-harmonic functions on the nodal set, this result suggests that $s$-harmonic functions can vanish like any polynomial. We will discuss in the following Sections the implication of this classification.
\begin{lemma}\cite[Lemma 5.2]{GaRo}\label{garofalo}
  Let $p(x)$ be an homogeneous polynomial of degree $d$ in $\R^n$.Then, there exists a unique polynomial $q(X)=q(x,y)$ of degree $d$ in $\R^{n+1}$ such that
  $$
  \begin{cases}
  L_a q = 0 & \mbox{in }\R^{n+1}\\
  q(x,y)=q(x,-y) & \mbox{in } \R^{n+1}\\
  q(x,0)=p(x) & \mbox{on }\R^n.
  \end{cases}
  $$
\end{lemma}
In particular, it can proved that this extension is obtained by
$$
q(x,y)= \sum_{k\geq 0}^{\,d/2\!}{(-1)^k c_{2k} \Delta^k \frac{x^\alpha}{\alpha!} \frac{y^{2k}}{(2k!)}}, \quad c_{2k} = \prod_{i=1}^k \frac{2i -1}{2i-2s},
$$
where $\alpha=(\alpha_1,\dots,\alpha_n)\in \N^d$, $x^\alpha = x_1^{\alpha_1}\cdots x_n^{\alpha_n}$ and $\alpha! = \alpha_1 ! \cdots \alpha_n!$.\\
Inspired by the previous results, let us introduce the following classes of blow-up limit.
\begin{definition}\label{blowup.class}
  Given $a \in (-1,1)$ and $k \in \R$, we define the set of all possible blow-up limit of order $k$, i.e. the set of all $L_a$-harmonic symmetric polynomials of degree $k$, as
  $$
  \mathfrak{B}^a_k(\R^{n+1}) = \left\{p \in H^{1,a}_{\loc}(\R^{n+1})\left\lvert \quad
  \begin{aligned}
  &L_a p =0 \mbox{ in }\R^{n+1}\\
  &p(X)=\abs{X}^k p\left(\frac{X}{\abs{X}}\right) \mbox{ in }\R^{n+1}
  \end{aligned}
  \right.
  \right\}.
  $$
  Similarly, the set of blow-up limit of order $k$ respectively symmetric or antisymmetric with respect to $\Sigma$ are defined as
  \begin{align*}
  \mathfrak{sB}_{k}^a(\R^{n+1}) &= \Big\{p \in \mathfrak{B}^a_k(\R^{n+1}) \left\lvert\, p\mbox{ symmetric with respect to }\Sigma\right.\Big\},\\
  \mathfrak{aB}_{k}^a(\R^{n+1}) &= \Big\{p \in \mathfrak{B}^a_k(\R^{n+1}) \left\lvert\, p\mbox{ antisymmetric with respect to }\Sigma\right.\Big\} .
  \end{align*}
\end{definition}
By Proposition \ref{even.odd} and Lemma \ref{yornot} we can classify even more the structure of the previous classes emphasizing two subclasses of blow-up limit.
\begin{definition}\label{blowup.class.y}
  Given $a \in (-1,1)$ and $k \in \R$, let us define $\mathfrak{sB}^*_k(\R^{n+1})=\mathfrak{B}^0_k(\R^{n}_x)$ the set of functions $p\in \mathfrak{B}_k^a(\R^{n+1})$ such that $\Delta_x p =0$, namely $p(x,y)=p(x)$ in $\R^{n+1}= \R^n_x \times \R_y$.
\end{definition}
By the previous Section, we already know that for $a\in (-1,1)$ we have $\mathfrak{B}^a_1(\R^{n+1})=\mathfrak{B}^*_1(\R^{n+1})$ and for $k\geq 2$ we have $\mathfrak{sB}^a_k(\R^{n+1})\setminus\mathfrak{B}^*_k(\R^{n+1}) \neq \emptyset$ and it consists of all blow-up limit which depends on the variable $y$. Finally
\begin{corollary}\label{anti.sym}
  For $a \in (-1,1)$, let $u$ be an $L_a$-harmonic function in $B_1$ and $X_0 \in \Gamma_k(u)$, for some $k \in 1+\N$ or $k\in 1-a+\N$. Then, every blow-up limit $p$ centered in $X_0 \in \Gamma_k(u)$ is either in $\mathfrak{sB}_k^a(\R^{n+1})$ or in $\mathfrak{aB}^a_k(\R^{n+1})$. Moreover, for every $a \in (-1,1)$ we have
  $$
  \mathfrak{aB}_k^a(\R^{n+1})= \mathfrak{sB}_{k+a-1}^{2-a}(\R^{n+1})y\abs{y}^{-a}.
  $$
\end{corollary}
 \section{Uniqueness and continuity of tangent maps and tangent fields}\label{sec.weiss}
In this Section we start introducing a Weiss type
monotonicity formula, which is a fundamental tool for the blow-up analysis at the nodal points $X_0 \in \Gamma(u)$ where $N(X_0, u,0^+) = k$. Starting from this, we will improve our knowledge of the blow-up convergence by proving the existence of a unique non trivial blow-up limit at every point of the nodal set $\Gamma(u)$, which will be called the tangent map $\varphi^{X_0}$ of $u$ at $X_0$. In particular, driven by the decomposition in \eqref{decompos},
 we introduce the notion of tangent \emph{field} at the nodal point $\Phi^{X_0}$, which have a major role in our blow-up analysis. 
\begin{definition}
Given $u$, an $L_a$-harmonic function in $B_1$, for a real number $k\geq \min\{1,1-a\}$, we define
$$
\Gamma_k(u)  \coloneqq  \{X_0 \in \Gamma(u) \colon N(X_0,u,0^+)=k\}.
$$
\end{definition}
One has to point out that the sets $\Gamma_k(u)$ may be nonempty only for $k$ in a certain range of values, depending on $u$ and $k$ itself. Indeed, by Proposition \ref{gap}, we already know that $\Gamma_k(u)\cap \Sigma\neq \emptyset$ implies $k \in 1+ \N$ or $k \in 1-a+\N$. Notice that the Weiss formula used for the uniformly elliptic case is usually different. We remark that all the following results are well known for the case $X_0 \in \Gamma_k(u)\setminus \Sigma$ since the $L_a$ operator is uniformly elliptic outside $\Sigma$. A direct computation gives
\begin{proposition}\label{weiss.formula}
Let $u$ be a nontrivial $L_a$-harmonic function in $B_1$. For $X_0 \in \Gamma_k(u)\cap \Sigma$, we introduce the $k$-Weiss function
$$
W_k(X_0,u,r) = \frac{1}{r^{n+a-1+2k}}\int_{B_r(X_0)}{\abs{y}^a \abs{\nabla u}^2 \mathrm{d}X}- \frac{k}{r^{n+a+2k}}\int_{\partial B_r(X_0)}{\abs{y}^a \abs{u}^2\mathrm{d}\sigma}.
$$
For $r\in (0,1-\abs{X_0})$ we have
\begin{equation}\label{weiss}
\frac{d}{dr}W_k(X_0,u,r)=\frac{2}{r^{n+a+1+2k}}\int_{\partial B_r(X_0)}{\abs{y}^a \left(\langle \nabla u, X-X_0\rangle - k u \right)^2 \mathrm{d}\sigma}.
\end{equation}
which implies that $r\mapsto W_k(X_0,u,r)$ is monotone nondecreasing in $(0,1-\abs{X_0})$. Furthermore, the map $r\mapsto W_k(X_0,u,r)$ is constant if and only if $u$ is homogeneous of degree $k$.
\end{proposition}

By a integration by parts, we can rewrite the $k$-Weiss function as
$$
W_k(X_0,u,r)=\frac{1}{r^{n+a+2k}}\int_{\partial B_r(X_0)}{\abs{y}^a u \left( \langle \nabla u , X-X_0\rangle - u \right)\mathrm{d}\sigma}.
$$
\begin{proposition}\label{Monneau}
  Let $a\in (-1,1)$ and $u$ be an $L_a$-harmonic function in $B_1$ and $X_0 \in \Gamma_k(u)\cap \Sigma$. For every homogenous $L_a$-harmonic polynomial $p\in \mathfrak{B}^a_k(\R^{n+1})$, the map
  $$
  r \mapsto \frac{H(X_0,u-p_{X_0},r)}{r^{2k}}= \frac{1}{r^{n+a+2k}}\int_{\partial B_r(X_0)}{\abs{y}^a \left(u -p_{X_0}\right)^2\mathrm{d}\sigma}
  $$
  is monotone non decreasing in $(0,1-\abs{X_0})$, where $p_{X_0}(X)=p(X-X_0)$.
\end{proposition}
Through the following Section, we will use the notation $r \mapsto M(X_0,u,p_{X_0},r)$ for the previous map.
\begin{proof}
  Since $X_0 \in \Gamma_k(u)\cap \Sigma$ and $p$ is a $k$-homogenous $L_a$-harmonic function, we already know that $W_k(X_0,u,r)\geq 0$ and $W_k(X_0,p_{X_0},r)=0$ for every $r \in (0,1-\abs{X_0})$. Let $w = u - p_{X_0}$, then
  \begin{align*}
    W_k(X_0,u,r) &= W_k(X_0,u,r) + W_k(X_0,p_{X_0},r) \\
     &= \frac{1}{r^{n+a-1+2k}}\left(\int_{B_r(X_0)}{\!\!\!\!\abs{y}^a \abs{\nabla w}^2 +2\abs{y}^a \langle \nabla w,\nabla p\rangle\mathrm{d}X}- \frac{k}{r}\int_{\partial B_r(X_0)}{\!\!\!\!\abs{y}^a w^2 +2 \abs{y}^a w p \mathrm{d}\sigma}\right) \\
     &= W_k(X_0,w,r) + \frac{2}{r^{n+a+2k}}\int_{\partial B_r(X_0)}{\abs{y}^a w( \langle \nabla p_{X_0},X-X_0\rangle - k p) \mathrm{d}\sigma}  \\
    &= W_k(X_0,u-p_{X_0},r).
  \end{align*}
Hence , by \eqref{weiss.h}, we finally obtain
\begin{align*}
\frac{d}{dr} \frac{H(X_0,u-p_{X_0},r)}{r^{2k}}&=2\frac{H(X_0,u-p_{X_0},r)}{r^{2k+1}}\left(N(X_0,u-p_{X_0},r)-k\right)\\
&=\frac{2}{r}W_k(X_0,u-p_{X_0},r)\geq 0.
\end{align*}
\end{proof}
Now, we apply the previous monotonicity formul\ae to study the vanishing order of the $L_a$-harmonic function at the points of the nodal set. In particular, we prove a nondegeneracy and uniqueness result of the blow-up limit, for every points of the nodal set.
\begin{lemma}\label{growth}
Let $a\in (-1,1)$ and $u$ be an $L_a$-harmonic function in $B_1$. Then, for every $X_0 \in \Gamma_k(u)\cap \Sigma$, there exists $C>0$ such that
$$
\abs{u(X)}\leq C\abs{X-X_0}^k\quad \mbox{in }B_{R/2}(X_0).
$$
where $R=1- \mbox{dist}(X_0,\partial B_1)$.
\end{lemma}
\begin{proof}
Since whenever $X_0\in \Gamma_k(u)$ we have $N(X_0,u,r)\geq N(X_0,u,0^+)=k$, then for every $r\in (0,R)$
$$
\frac{d}{dr}{\log{H(X_0,u,r)}} \geq \frac{2}{r}N(X_0,u,r) \geq \frac{2k}{r}
$$
and similarly
$$
\log \frac{H(X_0,u,R)}{H(X_0,u,r)} \geq 2k\log \frac{1}{r},
$$
which implies $H(X_0,u,r)\leq H(X_0,u,R) r^{2k}$. Now, by \cite[Lemma A.2.]{ww} and the previous estimate, we obtain for every $r \in (0,R)$
\begin{align*}
 \sup_{\overline{B_{r/2}}} u &\leq C(n,a) \left(\frac{1}{r^{n+1+a}}\int_{B_r}{\abs{y}^a u^2\mathrm{d}X} \right)^{1/2} \\
  &\leq   C(n,a) \left(\frac{H(0,u,R)}{n+a+1}\right)^{1/2},
\end{align*}
where in the second inequality we used the monotonicity of $r\mapsto H(0,u_k,r)$ in $(0,R)$.
\end{proof}
\begin{lemma}[Nondegeneracy]\label{nondegeneracy}
Let $a\in(-1,1)$ and $u$ be an $L_a$-harmonic function in $B_1$. Then, for every $X_0 \in \Gamma_k(u)\cap \Sigma$  there exists $C>0$ such that
$$
\sup_{\partial B_r(X_0)}{\abs{u(X)}}\geq C r^{k}\quad \mathrm{for}\,\,0<r<R
$$
where $R=1-\mbox{dist}(X_0,\partial B_1)$.
\end{lemma}
\begin{proof}
Fix $X_0\in \Gamma_k(u)$ and suppose by contradiction, given a decreasing sequence $r_j\downarrow 0$, that
$$
\lim_{j\to \infty}\frac{H(X_0,u,r_j)^{1/2}}{r_j^{k}}=\lim_{j\to\infty}{\bigg(\frac{1}{r_j^{n+a+2k}}\int_{\partial B_{r_j}(X_0)}{\abs{y}^a u^2\,d\sigma\bigg)^{1/2}}}\!\!=0.
$$
Consider now the blow-up sequence
$$
u_j(X)=\frac{u(X_0+r_j X)}{\rho_j}\quad \mbox{where }\, \rho_j = H(X_0,u,r_j)^{1/2}
$$
constructed starting from $r_j$ and centered in $X_0 \in \Gamma_k(u)$. By Theorem \ref{blowup.convergence}, up to a subsequence $u_j \to p$ uniformly, where $p$ is a nontrivial $L_a$-harmonic homogenous polynomial of degree $k$ such that $H(0,p,1)=1$.\\
Let us focus our attention on the functional $M(X_0,u,p_{X_0},r)$ with $p_{X_0}$ as above. By the assumption on the growth of $u$ it follows
\begin{align*}
M(X_0,u,p_{X_0},0^+) &= \lim_{r\to 0} \frac{1}{r^{n+a+2k}}\int_{\partial B_r(X_0)}{\abs{y}^a(u-p_{X_0})^2\,d\sigma}\\
&=\lim_{r\to 0} \int_{\partial B_1}{\abs{y}^a\Big(\frac{u(X_0+rX)}{r^k}-p(X) \Big)^2\,d\sigma}\\
&= \int_{\partial B_1}{\abs{y}^a p^2\,d\sigma}\\
&=\frac{1}{r^{n+a+2k}}\int_{\partial B_r(X_0)}{\abs{y}^a {p}_{X_0}^2\,d\sigma}.
\end{align*}
By the monotonicity result of Proposition \ref{Monneau} on the map $r\mapsto M(X_0,u,p_{X_0},r)$, we obtain
$$
\frac{1}{r^{n+a-1+2k}}\int_{\partial B_r(X_0)}{\abs{y}^a(u- p_{X_0})^2\,d\sigma}\geq\frac{1}{r^{n+a-1+2k}}\int_{\partial B_r(X_0)}{\abs{y}^a p_{X_0}^2\,d\sigma}
$$
and similarly
$$
\int_{\partial B_r(X_0)}{\abs{y}^a (u^2-2u p_{X_0})\,d\sigma}\geq 0.
$$
On the other hand, rescaling the previous inequality and using the blow-up sequence $u_k$ defined as above, we obtain
$$
\int_{\partial B_1}{\abs{y}^a\left(H(X_0,u,r_j)u_j^2-2H(X_0,u,r_j)^{1/2}r_j^k  u_j p\right)\,d\sigma}\geq 0
$$
and
$$
\int_{\partial B_1}{\abs{y}^a\left(\frac{H(X_0,u,r_j)^{1/2}}{r_j^k}u_j^2-2 u_j p\right)\,d\sigma}\geq 0.
$$
One gets a contradiction passing to the limit for $j\to\infty$; indeed by the previous inequality we obtain
$$
\int_{\partial B_1}{\abs{y}^a p^2\,d\sigma}\leq 0
$$
in contradiction with $p \not\equiv 0$.
\end{proof}

\begin{theorem}[Uniqueness of the blow-up limit]\label{uniqueness}
Given $a \in (-1,1)$ and $u$ be an $L_a$-harmonic function in $B_1$, let us consider $X_0\in \Gamma_k(u)\cap \Sigma$, i.e. $N(X_0,u,0^+)=k$. Then there exists a unique nonzero $p \in \mathfrak{B}_k^a(\R^{n+1})$ blow-up limit such that
\begin{equation}\label{blow.up.homogenous}
  u_{X_0,r}(X) =\frac{u(X_0+rX)}{r^k} \longrightarrow  p(X).
\end{equation}
\end{theorem}
\begin{proof}
Up to a subsequence $r_j\to 0^+$, we have that $u_{X_0,r_j}\to p$ in $\mathcal{C}^{0,\alpha}_{\loc}$. The existence of such limit follows directly from the previous growth estimate $\abs{u(X)} \leq C\abs{X}^k$ and by Lemma \ref{nondegeneracy} we have $p$ is not identically zero. Now, for any $r>0$ we have
$$
W_k(0,p,r) = \lim_{j\to \infty}{W_k(0,u_{X_0,r_j},r)}= \lim_{j\to \infty}{W_k(X_0,u,r r_j)}= W_k(X_0,u,0^+)=0.
$$
In particular, Proposition \ref{weiss.formula} implies that the $L_a$-harmonic function $p$ is $k$-homogeneous and consequently $p \in \mathfrak{B}^a_k(\R^{n+1})$. By Proposition \ref{Monneau} the limit $M(X_0,u,p_{X_0},0^+)$ exists and can be computed by
\begin{align*}
M(X_0,u,p_{X_0},0^+) &= \lim_{j\to\infty}{M(X_0,u,p_{X_0},r_j)}\\
&= \lim_{j\to\infty}{M(0,u_{X_0,r_j},p,1)}\\
&= \lim_{j\to\infty}{\int_{\partial B_1}{\abs{y}^a(u_{X_0,r_j} -p)^2\,d\sigma}}=0.
\end{align*}
Moreover, let us suppose by contradiction that for any other sequence $r_i\to 0^+$ we have that the associated sequence converges to another blow-up limit, i.e. $u_{X_0,r_i}\to q\in \mathfrak{B}^a_k(\R^{n+1}), q\not\equiv p$, then
\begin{align*}
0=M(X_0,u,p_{X_0},0^+) &= \lim_{i\to \infty}M(X_0,u,p_{X_0},r_i)\\
&=\lim_{i\to\infty}\int_{\partial B_1}{\abs{y}^a (u_{r_i} - p)^2\,d\sigma}\\
&=\int_{\partial B_1}{\abs{y}^a (q- p)^2\,d\sigma}.
\end{align*}
As we claim, since $q$ and $p$ are both homogenous of degree $k$ they must coincide in $\R^n$.
\end{proof}
Inspired by the previous uniqueness and nondegeneracy results, we introduce the notion of tangent map at every point on the nodal set $\Gamma(u)$.
\begin{definition}\label{tangent.map}
   Given $a \in (-1,1)$, let $u$ be an $L_a$-harmonic function in $B_1$ and $X_0 \in \Gamma_k(u)\cap \Sigma$, for $k\geq \min\{1,1-a\}$. We define as \emph{tangent map} of $u$ at $X_0$ the unique nonzero map $\varphi^{X_0} \in \mathfrak{B}^a_k(u)$ such that
  $$
  u_{X_0,r}(X)=\frac{u(X_0+rX)}{r^k} \longrightarrow \varphi^{X_0}(X).
  $$
  Moreover, we define as \emph{normalized} tangent map of $u$ at $X_0$, the unique nonzero map $p^{X_0} \in \mathfrak{B}^a_k(u)$ normalized with respect to the $L^{2,a}(\partial B_1)$ norm, i.e. the map obtained as
  $$
  u_{X_0,r}(X)=\frac{u(X_0+rX)}{\sqrt{H(X_0,u,r)}} \longrightarrow  p^{X_0}.
  $$
\end{definition}
Exploiting the deep connection between the existence and uniqueness of the tangent map and the Taylor expansion of an $L_a$-harmonic function, we can find another characterization of the sets $\Gamma_k(u)$.
\begin{corollary}\label{gamma_k}
   For $a\in (-1,1)$, let $u$ be an $L_a$-harmonic function in $B_1$ and $X_0 \in \Gamma_k(u)\cap \Sigma$, with $k\geq \min\{2,2-a\}$. Then
   \begin{itemize}
     \item if $k \in 2+\N$, we have $D^{\nu}u(X_0)=0$ for every $\abs{\nu}\leq k-1$ and there exists $\abs{\nu_0}=k$ such that $D^{\nu_0}u(X_0)\neq0$;
     \item if $k \in 2-a+\N$, we have $D^{\nu}(u y\abs{y}^{-a})(X_0)=0$ for every $\abs{\nu}\leq k-1$ and there exists $\abs{\nu_0}=k$ such that $D^{\nu_0}(u y\abs{y}^{-a})(X_0)\neq0$.
   \end{itemize}
\end{corollary}
Finally, we can prove the validity of the weak unique continuation principle for the restriction of $\Gamma(u)$ on $\Sigma$. This result will improve the study of the nodal set of $u$ by showing that its restriction on the characteristic manifold $\Sigma$ is either with empty interior in $\Sigma$ or is $\Sigma$ itself. While in \cite{Ruland} the author proved a similar weak unique continuation property using a boot strap argument based on some regularity estimates for the $L_a$-operator, in our case we want to emphasize how our blow-up analysis and the classification of the tangent maps allow to study several local property of $L_a$-harmonic function.
\begin{proposition}\label{ucp.sigma}
  Let $a\in (-1,1)$ and $u$ be an $L_a$-harmonic function in $B_1$. If there exists $X_0 \in B_1\cap \Sigma$ and $R<1-\abs{X_0}$ such that $u = 0$ on $B_R(X_0)\cap \Sigma$, then $u\equiv 0$ on $B_1\cap \Sigma$.
\end{proposition}
\begin{proof}
  Let $X_0 \in \Gamma(u)\cap \Sigma$ and $R< 1- \abs{X_0}$. Since we are focusing the attention on the restriction of the nodal set on $\Sigma$, by definition of the symmetric part of $u$ with respect to $\Sigma$, we can assume that $u=u_e$ is purely symmetric with respect to $\Sigma$.\\
  The idea of the proof is to prove that $u$ is identically zero in the whole ball $B_R(X_0)$ in order to apply the Strong Unique continuation property Corollary \ref{unique.continuation}, which is actually a stronger result since it does not only concern the trace of $u$ on $\Sigma$.\\
  Suppose by contradiction that $u\not \equiv 0$ on $B_R(X_0)$, then
  $$
  H(X_0,u,r)=\frac{1}{r^{n+a}}\int_{\partial B_r(X_0)}{\abs{y}^a u^2 \mathrm{d}X} >0
  $$
  for all $r\in (0,R)$. Now, since $X_0 \in \Gamma(u)$, there exists by Theorem \ref{uniqueness} a unique nontrivial tangent map $\varphi^{X_0}\in \mathfrak{B}^a_k(\R^{n+1})$ of $u$ at $X_0$, where $k=N(X_0,u,0^+)$. Since $u$ is symmetric with respect to $\Sigma$, by Corollary \ref{gap.even} we know that $\varphi^{X_0}\in \mathfrak{sB}^a_k(\R^{n+1})$, with $k\in 1+\N$.
  \\Let us see the points in $B_R(X_0)\cap \Sigma$ as the collection of point $X_0+r \nu$ for $r<R$ and $\nu \in S^{n}\cap \Sigma$. By the $L^\infty_{\loc}$ convergence of the blow-up sequence we obtain that
  $\varphi^{X_0}(\nu)=0$ for all $\nu \in S^n\cap \Sigma$, i.e. $\varphi^{X_0}\equiv 0$ on $\Sigma$. Let us prove now that $\varphi^{X_0} \equiv 0$ on $\R^{n+1}$ by induction on the homogeneity $k=N(0,\varphi^{X_0},0^+)$. \\Let $k=1$, then up to a rotation $\varphi^{X_0}(x,y)=C\langle X,e_1\rangle = Cx_1$, where $x=(x_1,\cdots,x_n)$ and consequently $C=0$. Now let us suppose that every $k$-homogenous $L_a$-harmonic polynomial symmetric with respect to $\Sigma$ which is zero on $\Sigma$ is actually identically zero in $\R^{n+1}$ and consider the case $k+1$. Given $v_i = \partial_{x_i}\varphi^{X_0} \in H^{1,a}(B_1)$ we have that
  $$
  \begin{cases}
    L_a v_i =0 & \mbox{in } \R^{n+1}, \\
    v_i = 0 & \mbox{on } \Sigma, \\
    N(0,v_i,0^+)\leq k.
  \end{cases}
  $$
  By the induction hypothesis we have that for every $i=1,\cdots,n$ $v_i\equiv 0$ on $\R^{n+1}$, i.e. $\partial_{x_i} \varphi^{X_0}\equiv 0$ and consequently $\varphi^{X_0}$ does not depend on $x\in \Sigma$. The absurd follows immediately since the only $L_a$-harmonic polynomial in the $y$-variable is purely antisymmetric and equal, up to a multiplicative constant, to $f(y)=y\abs{y}^{-a}$.
\end{proof}
Inspired by the doubling estimate in \cite{Ruland}, we obtain
\begin{proposition}\label{empty.interior}
       Let $a \in (-1,1)$ and $u$ a $L_a$-harmonic function in $B_1$. Then $\Gamma(u)$ has empty interior in $\R^{n+1}$ and its restrictions $\Gamma(u) \cap \Sigma$ is either equal to $\Sigma$ or it has empty interior in $\Sigma$ itself. More generally,
       $$
       \Gamma(u)\cap \Sigma = \Gamma(u_e)\cap \Sigma.
       $$
     \end{proposition}
     \begin{proof}
       Assume by contradiction that there exists $X_0 \in \Gamma(u)$ such that $d = \mbox{dist}(X_0,\partial \Gamma(u)) < R$, where $R=1-\abs{X_0}$. By definition of $d$, we have $H(X_0,u,r)>0$ for $r \in (d,d+\varepsilon)$, for some $\varepsilon>0$. By \eqref{doubling1}, the map $r\mapsto H(X_0,u,r)$ solves the Cauchy problem
        \begin{equation}\label{diff.eq}
        \begin{cases}
          H'(r) = a(r) H(r), & \mbox{for } r \in (d,d+\varepsilon) \\
          H(d)=0,
        \end{cases}
        \end{equation}
        where $a(r)=2N(X_0,u,r)/r$, which is continuous at $d$ by the monotonicity result of $r\mapsto N(X_0,u,r)$, i.e. Proposition \ref{Almgren.formula}. Then by uniqueness, $H(r)\equiv 0$ for $r>d$, which contradicts the definition of $d$ and the assumption that $u$ is not identically zero in $B_1$. \\Now, let us consider $\Gamma(u)\cap \Sigma$. By definition of $u_e,u_o$ we easily obtain
        $$
        \Gamma(u)\cap \Sigma = \Gamma(u_e)\cap \Sigma.
        $$
        Hence, let us suppose that $u \not\equiv u_o$, i.e. $\Gamma(u)\cap \Sigma \varsubsetneq \Sigma$, and assume as before that there exists $X_0 \in \Gamma(u)\cap \Sigma =\Gamma(u_e)\cap \Sigma$ such that $d = \mbox{dist}(X_0,\partial \Gamma(u_e)\cap \Sigma) < R$, where $R=1-\abs{X_0}$. In other words, the symmetric part $u_e$ of $u$ solves for every $r<d$,
        $$
        \begin{cases}
          L_a u_e=0 & \mbox{on } B_r(X_0) \\
          u_e=0 & \mbox{on } B_r(X_0)\cap \Sigma \\
          \partial_y^a u_e =0 & \mbox{on } B_r(X_0)\cap \Sigma,
        \end{cases}
        $$
        which implies that $u_e \equiv 0$ in $B_d(X_0)$, i.e. $H(X_0,u_e,d)=0$. As before, by the uniqueness of the Cauchy problem \eqref{diff.eq}, we obtain that $u_e$ is identically zero in $B_1$, in contradiction with the assumption $\Gamma(u)\cap \Sigma \varsubsetneq \Sigma$.
     \end{proof}

Looking again to the blow-up sequence, we can establish an auxiliary result concerning the convergence with respect to the Hausdorff distance $d_{\mathcal{H}}$. In particular, we will prove that given the blow-up sequence $(u_{X_0,r})_r$ of $u$ at $X_0$, then the nodal sets $\Gamma(u_{X_0,r})$ converge to $\Gamma(\varphi^{X_0})$ with respect to the Hausdorff distance. More precisely, given two sets $A,B$, the Hausdorff distance $d_{\mathcal{H}}(A,B)$ is defined as
$$
d_{\mathcal{H}}(A,B) := \max\left\{\sup_{a\in A} \mbox{dist}(a, B),\, \sup_{b\in B} \mbox{dist}(A, b)\right\}.
$$
Notice that $d_{\mathcal{H}}(A,B)\leq \varepsilon$ if and only if $A\subseteq N_{\varepsilon}(B)$ and $B\subseteq N_{\varepsilon}(A)$, where $N_{\varepsilon}(\cdot)$ is the closed $\varepsilon$-neighborhood of a set, i.e.
$$
N_{\varepsilon}(A) = \left\{X\in \R^{n+1} : \mbox{dist}(X, A) \leq \varepsilon\right\}.
$$
\begin{proposition}\label{H.convergence.nodal.set}
  Let $u$ be an $L_a$-harmonic function in $B_1$ and $X_0\in \Gamma_k(u) \cap \Sigma $. Given, $u_{X_0,r}$ the blow-up sequence at $X_0$, i.e.
  $$
  u_{X_0,r}(X)=\frac{u(X_0+rX)}{r^k} \to \varphi^{X_0}(X).
  $$
  Then $\Gamma(u_{X_0,r})\cap \Sigma \to \Gamma(\varphi^{X_0}) \cap \Sigma $ with respect to the Hausdorff distance $d_{\mathcal{H}}$ in $B_1$. More precisely, for every $k\geq \min\{1,1-a\}$ we have that
  $$
  \Gamma_k(u_{X_0,r})\cap \Sigma  \to \Gamma_k(\varphi^{X_0})\cap \Sigma
  $$ with respect to the Hausdorff distance $d_{\mathcal{H}}$ in $B_1$
\end{proposition}
\begin{proof}
Let $r_i\to 0^+$ and $u_i = u_{X_0,r_i}$ be the blow-up sequence of $u$ at $X_0$ associated with $r_i$ and $\Gamma_k(u_i)$ be the sequence of nodal sets associated with the blow-up sequence. Through the proof, we will omit the fact that we are just focusing on the restriction of the nodal sets on $\Sigma$ and we will call $\Gamma_k(\varphi^{X_0})$ as the tangent cone of $\Gamma_k(u)$ at $X_0$. By Theorem \ref{uniqueness} we already know that $\varphi^{X_0}$ and $\Gamma(\varphi^{X_0})$ do not depend on the choice of the sequence $r_k$. By the definition of Hausdorff distance, the claimed result
$$
d_{\mathcal{H}}\left(\Gamma_k(u_i)\cap B_1,\Gamma_k(\varphi^{X_0})\cap B_1\right)\to 0
$$
is equivalent to prove that for every $\varepsilon>0$ there exists $\overline{i}>0$ such that for every $i\geq \overline{i}$
\begin{align*}
  \Gamma_k(u_i) \cap B_1& \subseteq N_\varepsilon\left(\Gamma_k(\varphi^{X_0})\cap B_1 \right)\\
    \Gamma_k(\varphi^{X_0})\cap B_1& \subseteq N_\varepsilon\left(\Gamma_k(u_i)\cap B_1\right).
\end{align*}
Supposing by contradiction that the first inclusion is not true, then there exist $\overline{\varepsilon}>0$ and a sequence $X_i \in \Gamma_k(u_i)\cap B_1$ such that $\mbox{dist}\left(X_i, \Gamma_k(\varphi^{X_0})\cap B_1\right) > \overline{\varepsilon}$.
Up to a subsequence, $X_i \to \overline{X} \in \Gamma(\varphi^{X_0}) \cap \overline{B_1}$ by the $L^\infty_{\loc}$ convergence of $u_i \to \varphi^{X_0}$.
Since $X_i \in \Gamma_k(u_i)$ is equivalent to $X_0 + r_i X_i \in \Gamma_k(u)$, given $\Omega\subset\subset B_1$ such that $(X_0+r_i X_i)_i \subset \Omega$, let us consider
\begin{align*}
R_1 =& \min_{p \in \overline{\Omega}}{\mbox{dist}(p,\partial B_1)}<1,\\
\widetilde{C} =& \sup_{p \in \overline{\Omega}}{N(p,u,R_1)}.
\end{align*}
Hence, by the monotonicity result Proposition \ref{Almgren.formula} and Corollary \ref{doubling.corollary.Sigma}, for $p \in \Omega\cap \Gamma_k(u)$ and $r<R_1$ we obtain that $N(p,u,r)\geq k$ and
$$
N(p,u,r)\leq N(p,u,R_1)\left(\frac{R_1}{r}\right)^{n+a-1+2\tilde{C}}\leq \tilde{C}\frac{1}{r^{n+a-1+2\tilde{C}}}.
$$
In particular, from the second inequality we can easily state that for every $\varepsilon>0$ there exists $\overline{R}=\overline{R}(n,a,\Omega,\varepsilon)>0$ such that   $$N(p,u,r)\leq k+\varepsilon,$$
for every $p \in \Omega \cap \Gamma_k(u)$ and $r<\overline{R}$.\\
Now, since for $i>0$ sufficiently large $N(X_i,u_i,r) \leq N(X_0+r_i X_i ,u, r)$, if we take $p=X_0+r_i X_i$ in the previous inequality, we obtain that there exists $\overline{R}=\overline{R}(n,a,X_0)>0$ sufficiently small, such that for $r<\overline{R}$ we have
$$
k\leq N(X_i,u_i,r)\leq k + \min\left(\frac{1}{2},\frac{1-a}{2},\frac{\abs{a}}{2}\right).
$$
Since $\lim_{i} N(X_i,u_i,r) = N(\overline{X},\varphi^{X_0},r)$ for sufficiently small $r$, we directly obtain from Proposition \ref{gap} that $N(\overline{X},\varphi^{X_0},0^+)=k$, i.e. $\overline{X} \in \Gamma_k(\varphi^{X_0}) \cap \overline{B_1}$.
Finally, the absurd follows immediately since $\Gamma_k(\varphi^{X_0}) \cup \{0\}$ is an homogeneous cone passing through the origin and hence it implies that $\mbox{dist}(\overline{X},\Gamma_k(\varphi^{X_0}) \cap B_1) = 0$.\\
Now let us consider the second inclusion, i.e. for every $\varepsilon>0$ there exists $\overline{i}>0$ such that for every $i\geq \overline{i}$
$$
\Gamma_k(\varphi^{X_0})\cap B_1 \subseteq N_\varepsilon\left(\Gamma_k(u_i)\cap B_1\right).
$$
Let us start by proving that
given $\overline{X} \in \Gamma_k(\varphi^{X_0})$ and $\delta >0$ such that $B_\delta(\overline{X}) \cap \Gamma(\varphi^{X_0}) = B_\delta(\overline{X}) \cap \Gamma_k(\varphi^{X_0})$ there exists $\overline{i}>0$ such that for every $i\geq \overline{i}$ the function $u_i$ must admit a zero of order $k$ in $B_\delta(\overline{X})$, $\Gamma_k(u_i)\cap B_\delta(\overline{X})$. Suppose it is not true, we would have two possibilities: first that $u_i >0 $ in $B_\delta(\overline{X})$ for every $k>0$ or secondly that every zeros of $u_i$ is not of order $k$. In the first case, the positivity implies that $\varphi^{X_0}$ must be an homogeneous $L_a$-harmonic function nonnegative in $B_\delta(\overline{X})$ with $\varphi^{X_0}(\overline{X})=0$, and therefore $\varphi^{X_0} \equiv 0 $ in $\R^{n+1}$. In this case the contradiction follows by Lemma \ref{nondegeneracy} and Theorem \ref{uniqueness}.\\
Secondly, since up to a subsequence there exists a sequence $X_i \in \Gamma_h(u_i)\cap B_\delta(\overline{X})$ for $h\neq k$, by arguing as in the proof of the other inclusion, we can prove that $X_i \to \widetilde{X} \in \Gamma_h(\varphi^{X_0})\cap B_\delta(\overline{X})$, in contradiction with the definition of $\delta>0$.\\
Finally, suppose the existence of $\overline{\varepsilon}> 0$ and $X_i \in \Gamma_k(\varphi^{X_0})\cap B_1, X_i \to X \in \overline{\Gamma_k(\varphi^{X_0})}\cap \overline{B_1}$, such that
$\mbox{dist}(X_i, \Gamma_k(\varphi^{X_0})\cap B_1) > \overline{\varepsilon}$. Since $\overline{X}=\{0\}$ is a trivial case, let us focus on the case $\overline{X} \in \Gamma_k(\varphi^{X_0}) \cap \overline{B_1}$. By definition, $\Gamma_k(\varphi^{X_0})\cup \{0\}$ is an homogenous cone passing through the origin and hence we can take $\overline{X} \in \Gamma_k(\varphi^{X_0})\cap B_1$ such that $\abs{X - \overline{X}}\leq \overline{\varepsilon}/4$. Moreover, by the previous paragraph, there exist a sequence $\overline{X}_i \in \Gamma(u_i) \cap B_1$ and $\overline{i}>0$, such that for $i\geq \overline{i}$ we have $\abs{ \overline{X}_i - \overline{X}}\leq \min\{\delta,\overline{\varepsilon}\}/4$
Hence, we obtain
$$
\mbox{dist}(X_i, \Gamma_k(\varphi^{X_0})\cap B_1) \leq \abs{X_i - \overline{X}_i} \leq \abs{X_i - X} + \abs{X - \overline{X}} + \abs{\overline{X} - \overline{X}_i} 
< \overline{\varepsilon},
$$
which leads a contradiction for large $i>0$.
\end{proof}

The following result will be a fundamental tool in the study of $\Gamma(u)\cap \Sigma$. Indeed, by using the continuation of the tangent map with respect to the $L^{2,a}(\partial B_1)$, we will prove a separation property for the set $\Gamma_k(u)\cap \Sigma$, for $k\geq \min\{2,2-a\}$.
\begin{theorem}[Continuation of the tangent map on $\Gamma_k(u)$]
\label{continuation}
Let $X_0\in \Gamma_{k}(u) \cap \Sigma$ and $\varphi^{X_0}$ be the tangent map of $u$ at $X_0$, such that
\begin{equation}\label{eq.continuation}
u(X)=\varphi^{X_0}(X-X_0) + o(\abs{X-X_0}^{k}).
\end{equation}
Then, the map $X_0 \mapsto \varphi^{X_0}$ (from $\Gamma_k(u)$ to $\mathfrak{B}_k^a(\R^{n+1})$) is continuous.
Moreover, for any compact set $K \subset \Gamma_k(u) \cap B_1$ there exists a modulus of continuity $\sigma_K$ such that $\sigma_K(0)=0$ and
$$
\abs{u(X)-\varphi^{X_0}(X-X_0)}\leq \sigma_K\left(\abs{X-X_0}\right)\abs{X-X_0}^k,
$$
for any $X_0\in K$.
\end{theorem}
\begin{proof}
Since $\mathfrak{B}^a_k(\R^{n+1})$ is a finite-dimensional linear space, all  norms are equivalent and hence we can then endow it with the norm of $L^{2,a}(\partial B_1)$.
Fixed $X_0\in\Gamma(u) \cap \Sigma$, by Theorem \ref{uniqueness} we have the following expansion
$$
u(X)=\varphi^{X_0}(X-X_0) + o(\abs{X-X_0}^{k}).
$$
where $\varphi^{X_0}$ is the unique blow-up limit of $u$ in $X_0$. Given $\varepsilon>0$, consider $r_\varepsilon=r_\varepsilon(X_0)$ such that
$$
M(X_0,u,\varphi^{X_0},r_\varepsilon)= \frac{1}{r^{n+a+2k}_\varepsilon}\int_{\partial B_{r_\varepsilon}}{\abs{y}^a\big(u(X_0+X)-\varphi^{X_0}(X)\big)^2\,d\sigma}<\varepsilon.
$$
There exists also $\delta_\varepsilon=\delta_\varepsilon(X_0)$ such that if $X_1\in \Gamma_k(u)\cap \Sigma$ and $\abs{X_1-X_0}<\delta_\varepsilon$ then
$$
\frac{1}{r^{n+a+2k}_\varepsilon}\int_{\partial B_{r_\varepsilon}}{\abs{y}^a\big(u(X_1+X)-\varphi^{X_0}(X))^2\,d\sigma}<2\varepsilon
$$
or similarly
$$
\int_{\partial B_{1}}{\abs{y}^a\left( \frac{u(X_1+r_\varepsilon X)}{r^k_\varepsilon} - \varphi^{X_0}(X)\right)^2\,d\sigma}<2\varepsilon
$$
From Proposition \ref{Monneau}, we have that $M(X_1,u,\varphi^{X_0},r)<2\varepsilon$ for $r\in (0,r_\varepsilon)$, which implies
\begin{align*}
M(X_1,u,\varphi^{X_0},0^+)&=\lim_{r\to 0}{M(X_1,u,\varphi^{X_0},r)}\\
&=\lim_{r\to 0}\int_{\partial B_1}{\abs{y}^a\left( \frac{u(X_1+rX)}{r^k} - \varphi^{X_0}(X)\right)^2\,d\sigma}\\
&=\int_{\partial B_1}{\abs{y}^a\left( \varphi^{X_1}- \varphi^{X_0}\right)^2\,d\sigma}\leq 2\varepsilon.
\end{align*}
Now, by the previous computations, for $\abs{X_1-X_0}<\delta_\varepsilon, 0<r<r_\varepsilon$ we obtain
$$
\norm{u_{X_1,r}-\varphi^{X_1}}{L^{2,a}(\partial B_1)} \leq \norm{u_{X_1,r}-\varphi^{X_0}}{L^{2,a}(\partial B_1)}+ \norm{\varphi^{X_0}-\varphi^{X_1}}{L^{2,a}(\partial B_1)} \leq 2\sqrt{2\varepsilon},
$$
where $u_{X_1,r}$ and $u_{X_0,r}$ are the blow-up sequences defined in \eqref{blow.up.homogenous} centered respectively in $X_1$ and $X_0$. Now, covering the compact set $K\subset \Gamma_k(u)\cap B_1$ with finitely many balls $B_{\delta_\varepsilon(X_0^i)}(X_0^i)$, for some points $X_0^i \in K, i=1,\dots,N$, we obtain that the previous inequality is satisfied for all $X_1 \in K$ with $r<r_\varepsilon^K = \min\{r_\varepsilon(X_0^i)\colon i=1,\dots,N\}$.\\ Now, since $u_{X_1,r} -\varphi^{X_1}$ is an $L_a$-harmonic function in $B_1$, by \cite[Lemma A.2]{ww} and \eqref{doubling1}, we obtain
\begin{align*}
\sup_{B_{1/2}}\abs{u_{X_1,r}-\varphi^{X_1}}&\leq C(n,a)\left(\int_{B_1}{\abs{y}^a (u_{X_1,r}-\varphi^{X_1})^2\mathrm{d}X}\right)^{1/2}\\
& \leq 2C(n,a) \sqrt{\frac{2\varepsilon}{n+a+1}}
\end{align*}
for all $X_1 \in K, 0<r<r_\varepsilon^K$, which immediately implies the second part of the Theorem.
\end{proof}
The following definition allows us to study the structure of the restriction $\Gamma(u)\cap \Sigma$. Inspired by Proposition \ref{empty.interior}, since $\Gamma(u)\cap \Sigma = \Gamma(u_e)\cap \Sigma$, where $u_e$ is the symmetric part of $u$ with respect to $\Sigma$, we characterize the sets $\Gamma_k(u)$ starting from the unique tangent map of $u_e$. Moreover, since we are dealing with a purely symmetric function, we will see that the structure of the nodal set on $\Sigma$ is completely defined starting from the blow-up classes $\mathfrak{sB}^a_k(\R^{n+1})$ and $\mathfrak{B}^*_k(\R^{n+1})$.
\begin{definition}
  Given $u$ an $L_a$-harmonic function on $B_1$, for $k\geq \min\{1,1-a\}$ we define on $\Sigma$
  $$
  \Gamma^*_k(u)=\left\{X_0 \in \Gamma_k(u)\cap \Sigma\colon \varphi^{X_0}_e \in \mathfrak{sB}^*_k(\R^{n+1})\right\}\ \mbox{ and } \ \Gamma^a_k(u)= \Gamma_k(u)\setminus \Gamma^*_k(u),
  $$
  where $\varphi^{X_0}_e \in \mathfrak{sB}^a_k(\R^{n+1})$ is the unique tangent map of $u_e$ at $X_0$.
\end{definition}
In particular
  $ \Gamma_1(u)=\Gamma_1^*(u)$
  and for $k\geq 2$ the points in $\Gamma^a_k(u)$ are the ones whose tangent map depends on the variable $y$.
\begin{corollary}\label{disjoint}
  For every $k\geq 2$ we have that $\overline{\Gamma^*_k(u)}\cap \Gamma^a_k(u)=\emptyset = \Gamma^*_k(u)\cap\overline{\Gamma^a_k(u)}$.
\end{corollary}
\begin{proof}
The proof of this result is based on the continuation of the tangent map of $u$ on $\Gamma_k(u)\cap \Sigma$ with respect to the norm $L^{2,a}(\partial B_1)$. \\First, suppose by contradiction that there exists a sequence $(X_i)_i\subset \Gamma^*_k(u)$ such that $X_i \to X_0 \in \Gamma^a_k(u)$. Let $\varphi^{X_i}=\varphi^{X_i}_e$ and $\varphi^{X_0}=\varphi^{X_0}_e$ be respectively the tangent map of $u_e$ at $X_i$ and $X_0$, then by Theorem \ref{continuation} we obtain that $\varphi^{X_i} \to \varphi^{X_0}$ strongly in $L^{2,a}(\partial B_1)$, i.e. for every $\varepsilon >0$ there exists $N=N(\varepsilon)>0$ such that if $i>N$, then
$$
\int_{\partial B_1}{\abs{y}^{a}\left(\varphi^{X_i} -\varphi^{X_0}\right)^2\mathrm{d}\sigma} \leq \varepsilon.
$$
Hence, fixed $w_i=\varphi^{X_i} -\varphi^{X_0}$ we obtain that $L_a w_i=0$ in $B_1$ and $\norm{w_i}{L^{2,a}(\partial B_1)}\to 0$. Since $w_i$ is homogenous of degree $k$, we have
\begin{align*}
\int_{B_1}{\abs{y}^a w_i^2\mathrm{d}X} & = \frac{1}{n+a+2k+1}\int_{\partial B_1}{\abs{y}^a w_i^2 \mathrm{d}\sigma}\\
\int_{B_1}{\abs{y}^a \abs{\nabla w_i}^2\mathrm{d}X} &=k \int_{\partial B_1}{\abs{y}^a w_i^2 \mathrm{d}\sigma}
\end{align*}
which implies that $w_i \to 0$ strongly in $H^{1,a}(B_1)$. In particular, for every $\phi \in H^{1,a}(B_1)$ we have
$$
\int_{B_1}{\abs{y}^a \langle\nabla \varphi^{X_0},\nabla \phi\rangle\mathrm{d}X} = \lim _{i\to \infty} \int_{B_1}{\abs{y}^a \langle\nabla \varphi^{X_i},\nabla \phi\rangle\mathrm{d}X}
$$
If $\phi = \phi(y) \in C^\infty_c((-1,1))$ we obtain $\nabla \phi = \partial_y \phi\, e_y$ and consequently, since $\varphi^{X_i}\in \mathfrak{sB}^*_k(\R^{n+1})$, that
$$
\int_{B_1}{\abs{y}^a \partial_y \varphi^{X_0}\partial_y\phi \mathrm{d}X} = \lim _{i\to \infty} \int_{B_1}{\abs{y}^a \partial_y\varphi^{X_i}\partial_y \phi \mathrm{d}X} = 0,
$$
in contradiction with the fact that $\varphi^{X_0}\in \mathfrak{sB}^a_k(\R^{n+1})\setminus \mathfrak{sB}^*_k(\R^{n+1})$.\\
Similarly, suppose now there exists a sequence $(X_i)_i\subset \Gamma^a_k(u)$ such that $X_i \to X_0 \in \Gamma^*_k(u)$. As before, let $\varphi^{X_i}$ and $\varphi^{X_0}$ be respectively the tangent map of $u_e$ at $X_i$ and $X_0$, fixed $w_i = \varphi^{X_i}-\varphi^{X_0}$ we obtain
$$
\int_{B_1}{\abs{y}^a \phi\Delta_x w_i\mathrm{d}X} + \int_{B_1}{\abs{y}^a \left(-\partial^2_{yy}\varphi^{X_i}-\frac{a}{y}\partial_y \varphi^{X_i}\right)\phi\mathrm{d}X}= k\int_{\partial B_1}{\abs{y}^a w_i \phi \mathrm{d}\sigma}
$$
for every $\phi \in H^{1,a}(B_1)$. The idea now is to reach the contradiction by induction on $k$, proving that it is impossible that the sequence of $L_a$-harmonic polynomials in $\mathfrak{sB}^a_k(\R^{n+1})\setminus \mathfrak{sB}^*_k(\R^{n+1})$ converges strongly in the $L^{2,a}(\partial B_1)$-topology to a function in $\mathfrak{sB}^*_k(\R^{n+1})$.\\
First, for $\phi \in H^{1,a}_0(B_1)$, we have
$$
\abs{\int_{B_1}{\abs{y}^a \left(\partial^2_{yy}\varphi^{X_i}+\frac{a}{y}\partial_y \varphi^{X_i}\right)\phi\mathrm{d}X}}\leq C(n,k,a)\left(\norm{\nabla w_i}{L^{2,a}(B_1)}+ \norm{w_i}{L^{2,a}(B_1)}\right)\norm{\phi}{H^{1,a}(B_1)},
$$
which gives us that
\begin{equation}\label{psi}
\psi_i= -\partial^2_{yy}\varphi^{X_i}-\frac{a}{y}\partial_y \varphi^{X_i} \rightharpoonup 0 \ \mbox{ in } L^{2,a}(B_1),
\end{equation}
where $\psi_i$ is a sequence of homogeneous $L_a$-harmonic polynomial of degree $k-2\geq 0$. Since $X_i \mapsto\psi_i$ is continuous, by Lebesgue's dominated convergence theorem, we obtain $\norm{\psi_i}{L^{2,a}(B_1)}\to 0$, i.e. $\psi_i \to 0$ strongly in $L^{2,a}(B_1)$. \\Hence, let $k=2$ and $\varphi^{X_i}\in \mathfrak{sB}^a_2(\R^{n+1})\setminus \mathfrak{sB}^*_2(\R^{n+1})$ be the sequence that converges to some $\varphi^{X_0}\in \mathfrak{sB}^*_2(\R^{n+1})$. As in \eqref{psi}, let us consider the associate sequence $\psi_i$ of $L_a$-harmonic polynomial of degree $k-2=0$, i.e. a sequence of nonzero constants. Since $\varphi^{X_i}\in \mathfrak{sB}^a_2(\R^{n+1})\setminus \mathfrak{sB}^*_2(\R^{n+1})$, by the reasoning in Section \ref{section.blowup}, there exists, up to a multiplicative constant, a unique homogeneous polynomial $u^{X_i}=u^{X_i}(x)$ of degree $2$, such that
$$
\varphi^{X_i}(x,y)=u^{X_i}(x)-y^2 \ \mbox{ in }\R^{n+1},
$$
where $\Delta_x u^{X_i} = 2(1+a)$ in $\R^n$. In particular, by \eqref{psi} we obtain $\psi_i\equiv 2(1+a)$, and the contradiction follows immediately since $a\in (-1,1)$.\\Suppose now that we have proved the statement for every $k\leq K$ and let us consider the case $K+1$. By contradiction, let us suppose that $\mathfrak{sB}^a_{K+1}(\R^{n+1})\setminus \mathfrak{sB}^*_{K+1}(\R^{n+1})$ is not closed in the $L^{2,a}(\partial B_1)$ topology and $\varphi^{X_i}\to \varphi^{X_0}$ strongly in $L^{2,a}(\partial B_1)$, with $\varphi^{X_0}\in \mathfrak{sB}^*_{K+1}(\R^{n+1})$. \\Thus, we already know that the sequence $\psi_i$ defined by \eqref{psi} strongly converges to the zero function with respect to the $L^{2,a}(B_1)$ topology. Now, since $(\psi_i)_i$ are $(K-1)$-homogenous, we have that the $L^{2,a}(B_1)$ and $L^{2,a}(\partial B_1)$ topologies are equivalent. Finally, given that $0<K-1\leq K$, we have constructed a sequence of $(K-1)$-homogenous $L_a$-harmonic polynomials $\psi_i$ that converges to the zero function $0 \in \mathfrak{sB}^*_{K-1}(\R^{n+1})$, which contradicts the inductive hypothesis.
\end{proof}

In the uniformly elliptic case, as well known, the Almgren and Weiss monotonicity formulas allow us to prove uniqueness and non degeneracy of the tangent map and also to construct the generalized Taylor expansion of $u$ at $X_0$. In our degenerate-singular setting, since, as we already pointed out, the symmetric and antisymmetric cases are complementary, the notion of \emph{tangent field} of $u$ at a nodal point should be introduced in order to capture both this aspect of the solution $u$.
\begin{definition}
  Let $a \in (-1,1), u$ be an $L_a$-harmonic function in $B_1$ and $X_0 \in \Gamma_k(u)\cap\Sigma$, for some $k \geq \min\{1,1-a\}$. We define as \emph{tangent field} of $u$ at $X_0$ the unique nontrivial vector field $\Phi^{X_0} \in (H^{1,a}_\loc(\R^{n+1}))^2$ such that
  $$
  \Phi^{X_0}= (\varphi^{X_0}_e,\varphi^{X_0}_o),
  $$
  where $\varphi^{X_0}_e$ and $\varphi^{X_0}_o$ are respectively the tangent map of the symmetric part $u_e$ of $u$ and of the antisymmetric one $u_o$.
\end{definition}
The notion of tangent field will allow us to obtain a better understanding of the regularity features of the nodal set $\Gamma(u)$. Indeed, the main weakness of the concept of tangent map in this context is that it takes care either of the symmetric part of $u$ or of the even one since they do not share the same optimal regularity and even the same possible vanishing orders. More precisely, by Definition \ref{tangent.map}, for every $X_0 \in \Gamma_k(u)$
\begin{align*}
u_{X_0,r}(X) & = \frac{u_e(X_0+rX)}{r^k} + \frac{u_o(X_0+rX)}{r^k}\\
& = \frac{u_e^a(X_0+rX)}{r^k}+\frac{u_e^{2-a}(X_0+rX)}{r^{k-1+a}}y\abs{y}^{-a}
\end{align*}
where both $u_e^a$ and $u_e^{2-a}$ are symmetric with respect to $\Sigma$. By Proposition \ref{gap} we already know that the tangent map of $u$ at $X_0$ is either the tangent map of $u_e$ or the one of $u_o$.
\begin{definition}\label{Almgren.vectorial}
  Let $a \in (-1,1), u$ be an $L_a$-harmonic function in $B_1$ and $X_0 \in \Gamma_k(u)$, for some $k \geq \min\{1,1-a\}$. We define the vectorial Almgren frequency associated with the tangent field $\Phi^{X_0}$ of $u$ at $X_0$
  $$
  N(X,\Phi^{X_0},r)= \left(N(X,\varphi^{X_0}_e,r),N(X,\varphi^{X_0}_o,r)\right).
  $$
\end{definition}
Obviously, the \emph{vectorial} notion of the Almgren frequency formula can be naturally extended to the $L_a$-harmonic function $u$ as
$$
  N(X_0,u,r)= \left(N(X_0,u_e,r),N(X_0,u_o,r)\right).
$$
for every $X_0 \in \Sigma$, but we will avoid this ambiguity on this notion. However, if the function $u$ is symmetric or antisymmetric with respect to $\Sigma$, the Almgren frequencies associated with $\Phi$ is equal to the one of the tangent map $\varphi^{X_0}$ of $u$ at $X_0$ and it does not contain further information on the local behaviour of $u$ at $X_0$. In general, proving uniqueness result on both the symmetric and antisymmetric part of $u$ with respect to $\Sigma$ gives the following generalized Taylor expansion
\begin{corollary}\label{taylor.generalized}
  Given $a\in(-1,1)$, let $u$ be an $L_a$-harmonic function in $B_1$ and $X_0 \in \Gamma(u)\cap \Sigma$. Then
  $$
  u(X) = \varphi^{X_0}_e(X-X_0) +  \varphi^{X_0}_o(X-X_0) + o(\abs{X-X_0}^{k})
  $$
  where $\varphi^{X_0}_e \in \mathfrak{sB}^a_k(\R^{n+1})$ and $\varphi^{X_0}_o \in \mathfrak{aB}^a_k(\R^{n+1})$ are respectively the tangent maps of $u_e$ and $u_o$ at $X_0$ and $k=\max\{N(0,\varphi^{X_0}_e,0^+),N(0,\varphi^{X_0}_o,0^+)\}$.
\end{corollary}
\begin{lemma}\label{Fsigma}
Let $u$ be an $L_a$-harmonic function in $B_1$ and $\mathcal{R}(u)$ the set
\begin{align*}
  \mathcal{R}(u) &=\left\{X_0 \in \Gamma(u) \colon N(X_0,u_e,0^+)=1 \quad\mbox{or}\quad N(X_0,u_o,0^+)=1-a\right\}\\
  &=\left\{X_0 \in \Gamma(u) \colon N(X_0,\Phi^{X_0},0^+)= (1,1-a)\right\}.
\end{align*}
Then $\mathcal{R}(u)\cap \Sigma$ is relatively open in $\Gamma(u)\cap \Sigma$, while for $k\geq 2$ the set $\Gamma_k(u)$ is $F_\sigma$, i.e. it is a union of countably many closed sets.
\end{lemma}
\begin{proof}
The first part of the Lemma is a direct consequence of the upper semi-continuity of $X\mapsto N(X,u,0^+)$ on $\Sigma$. More precisely, since $\Gamma(u)\cap \Sigma = \Gamma(u_e)\cap \Sigma$, we can restrict our attention on functions symmetric with respect to $\Sigma$ and hence, we have
$$
\mathcal{R}(u)\cap \Sigma = \left\{X_0 \in \Gamma(u)\cap \Sigma\colon N(X_0,u_e,0^+)=1\right\}
$$
Now, by Lemma \ref{gap.even} we obtain
$$
\left\{X_0\in \Gamma(u)\colon
N(X_0,u_e,0^+)=1\right\} = \left\{X_0\in \Gamma(u)\colon N(X_0,u_e,0^+)\leq \frac{3}{2}\right\}.
$$
Hence, let us focus our attention on the case $\Gamma_k(u_e)\cap \Sigma$, with $k\geq 2$.
For $j\in\N$,  let us define with $E_j$ the set of points of $\Sigma$ such that
$$
E_j = \left\{X_0 \in \Gamma_k(u_e) \cap \Sigma \cap \overline{B_{1-1/j}} \colon \,\,\frac{1}{j}\rho^k \leq \sup_{\abs{X-X_0}= \rho}{\abs{u_e(X)}} < j \rho^k, \ 0<\rho<1-\abs{X_0} \right\}.
$$
By Lemma \ref{growth} and Lemma \ref{nondegeneracy}
we have that
$$
\Gamma_k(u) \cap \Sigma = \bigcup_{j=1}^\infty{E_j}.
$$
The result follows immediately once we prove that $E_j$ is a collection of closed sets. Given $X_0\in \overline{E_j}$, since it satisfies
\begin{equation}\label{serve}
\frac{1}{j}\rho^k \leq \sup_{\abs{X-X_0}= \rho}{\abs{u_e(X)}} < j \rho^k,
\end{equation}
we need only to show that $X_0\in \Gamma_k(u) \cap \Sigma$, i.e. $N(X_0,u_e,0^+)=k$. Since $X\mapsto N(X,u_e,0^+)$ is upper semi-continuous on $\Sigma$, we readily have $N(X_0,u_e,0^+)\geq k$. On the other hand, if $N(X_0,u_e,0^+)=k'> k$, we would have
$$
\abs{u_e(X)}\leq C \abs{X-X_0}^{k'}\quad \mathrm{in }\ B_{1-\abs{X_0}}(X_0)\cap \Sigma,
$$
which contradicts Lemma \ref{growth} and implies that $X_0 \in E_j$.
\end{proof}
An other relevant consequence of our analysis of the tangent field of $u$ at some nodal point $X_0 \in \Gamma(u)\cap \Sigma$ is
the following a posteriori result about the ``quasi'' upper semi-continuity of the
Almgren frequency $X \to N(X,u,0^+)$ in the whole $\R^{n+1}$.\\
Obviously, the restriction of this map
on the characteristic manifold $\Sigma$ and the one on its complementary are both
upper semi-continuous, but in general in the whole space $\R^{n+1}$ the upper semi-continuity is not an immediate consequence of the Almgren monotonicity formula.\\
This result is based on the decomposition \eqref{decompos} of $L_a$-harmonic functions and on the regularity result of Proposition \ref{smooth} for $L_a$-harmonic function symmetric with respect to $\Sigma$.\\
Moreover, the following result can be seen as the \emph{vectorial} counterpart of the classic one, since it will establish the validity of an upper semi-continuity property for the Almgren frequency in the vectorial sense of Definition \ref{Almgren.vectorial}. In particular, it allows to relate the notion of vanishing order on $\Sigma$ to the one on $\R^{n+1}\setminus \Sigma$, which is a fundamental step in order to comprehend the complete topology of the nodal set near $\Sigma$.
\begin{proposition}\label{generalized.upper}
  Let $u$ be an $L_a$-harmonic function in $B_1$. Given $(X_i)_i \in \Gamma_k(u)\setminus \Sigma$, with $k \in 1+\N$ such that $X_i \to X_0 \in \Gamma(u)\cap \Sigma$, then
  $$
  N(X_i,u,0^+) \leq
  \begin{cases}
    N(X_0,u_e,0^+),\\
    N(X_0,u_o,0^+)+a.
  \end{cases}
  $$
  \end{proposition}
\begin{proof}
By Proposition \ref{even.odd}, we already know that there exist $f \in H^{1,a}(B_1), g\in H^{1,2-a}(B_1)$ symmetric with respect to $\Sigma$ and respectively $L_a$ and $L_{2-a}$-harmonic in $B_1$, such that
\begin{equation}\label{da.derivare}
u(x,y)=f(x,y)+g(x,y)y\abs{y}^{-a} \quad \mbox{in }B_1,
\end{equation}
where respectively the first term is the symmetric part $u_e$ of $u$ with respect to $\Sigma$ and the second one the antisymmetric part $u_o$.\\
Throughout this proof, let us suppose that up to a subsequence $y_i>0$. Since $(X_i)_i \in \Gamma_k(u)\setminus \Sigma$ and the operator $L_a$ is locally uniformly elliptic on $\R^{n+1}\setminus \Sigma$ we know that $D^{\nu}u(X_i)=0$, for any $\abs{\nu}<k$ and there exists $\abs{\nu_0}=k$ such that $D^{\nu_0}u(X_i)\neq 0$. Let us prove the main result by induction on $k\geq 2$. If $k=N(X_i,u,0^+)=2$, then for every $j=1,\cdots,n$ we obtain from \eqref{da.derivare} that
\begin{gather*}
\partial_{x_j} f(x_i,y_i)= - \partial_{x_j}g(x_i,y_i)y_i^{1-a},\\
-y_i^a\partial_{y} f(x_i,y_i)= (1-a)g(x_i,y_i) + y_i \partial_y g(x_i,y_i),
\end{gather*}
where the maps $X \mapsto \partial_{x_j}f(X), X\mapsto \partial_y f(X), X \mapsto \partial_{x_j}g(X)$ and $X\mapsto \partial_y g(X)$ are all smooth in $B_1$ thanks to Proposition \ref{smooth}. Passing through the limit as $i\to \infty$ we obtain
$$
\partial_{x_j} f(X_0)= 0 \quad \mbox{and}\quad -\partial^a_{y} f(X_0)= (1-a)g(X_0).
$$
First, since $\partial^a_y f$ is antisymmetric with respect to $\Sigma$ we obtain that $g(X_0)=0$ and consequently $N(X_0,u_o,0^+)=N(X_0,g,0^+)+1-a\geq 2-a$. Similarly, if $\varphi^{X_0}(f)$ is the tangent map of $f$ at $X_0$, we obtain $\abs{\nabla_X (\varphi^{X_0}(f))(X_0)}=0$ and consequently that $N(X_0,f,0^+)=N(0,\varphi^{X_0},0^+)\geq 2$, as required.\\
Now, let us prove the inductive step $k-1\mapsto k$. Let us consider for $j=1,\dots,n$ the collection of symmetric $L_a$-harmonic functions $v_j = \partial_{x_j} u$ and the antisymmetric $L_{-a}$-harmonic function $w= \partial^a_y u$. Since $N(X_i,u,0^+)=k$, we obtain
$$
N(X_i,v_j,0^+)=k-1\quad\mbox{for } i=1,\dots,n\quad\!\mbox{and }N(X_i,w,0^+)=k-1,
$$
where we remark that since $X_i \not \in \Sigma$ it is the same to consider the order of vanishing of $\partial_y u$ or of the covariant derivative $w=\partial^a_y u$.\\ By the inductive hypothesis, passing through the limit as $i\to \infty$ we obtain
$$
\begin{cases}
  N(X_0,v_{j,e},0^+) \geq k-1\\
  N(X_0,v_{j,o},0^+) + a \geq k-1\\
\end{cases}\mbox{for }j=1,\dots,n
\quad\mbox{and}\quad
\begin{cases}
  N(X_0,w_{e},0^+) \geq k-1\\
  N(X_0,w_{o},0^+) - a \geq k-1\\
\end{cases}.
$$
Hence, comparing this result with  the notations in \eqref{da.derivare}, since $v_j = \partial_{x_j}f$ and $w = (1-a)g$ on $\Sigma$, we obtain
$$
\begin{cases}
  N(X_0,\partial_{x_j}f,0^+) \geq k-1\\
  N(X_0,\partial_{x_j}g,0^+) + a \geq k-1\\
\end{cases}\mbox{for }j=1,\dots,n
\quad\mbox{and}\quad
\begin{cases}
  N(X_0,g,0^+) \geq k-1\\
  N(X_0,\partial_y^a f,0^+) - a \geq k-1\\
\end{cases},
$$
which directly imply that $N(X_0,u_e,0^+)\geq k$ and $N(X_0,u_o,0^+)\geq k-a$, as required.
\end{proof}
\begin{lemma}\label{Fsigma}
Let $u$ be an $L_a$-harmonic function in $B_1$ and $\mathcal{R}(u)$ the set
\begin{align*}
  \mathcal{R}(u) &= \left\{X_0 \in \Gamma(u)\left\lvert
  \begin{array}{cc}
    N(X_0,u,0^+)=1 & \emph{ if }X_0 \not\in \Sigma \\
    N(X_0,u_e,0^+)=1 \emph{ or } N(X_0,u_o,0^+)=1-a & \emph{ if }X_0 \in \Sigma
  \end{array} \right.\right\}\\
  &=\left\{X_0 \in \Gamma(u) \left\lvert
  \begin{array}{cc}
 N(X_0,u,0^+)=1 & \emph{ if }X_0 \not\in \Sigma \\
    N(0,\Phi^{X_0},0^+)=(1,1-a) & \emph{ if }X_0 \in \Sigma
  \end{array} \right. \right\},
\end{align*}
is relatively open in $\Gamma(u)$, while for $k\geq \min\{2,2-a\}$ the set $\Gamma_k(u)$ is $F_\sigma$, i.e. it is a union of countably many closed sets.
\end{lemma}
\begin{proof}
The first part of the Lemma is a direct consequence of the upper semi-continuity of $X\mapsto N(X,u,r)$ restricted to $\Sigma$ and to $\R^{n+1}\setminus \Sigma$ and of the Proposition \ref{generalized.upper}. Hence, let us focus our attention on the case $\Gamma_k(u)$, with $k\geq \min\{2,2-a\}$.
For $j\in\N$,  let us define with $E_j$ the set of points of $\Sigma$ such that
$$
E_j = \left\{X_0 \in \Gamma_k(u) \cap \Sigma \cap \overline{B_{1-1/j}} \colon \,\,\frac{1}{j}\rho^k \leq \sup_{\abs{X-X_0}= \rho}{\abs{u(X)}} < j \rho^k, \ 0<\rho<1-\abs{X_0} \right\}.
$$
By Lemma \ref{growth} and Lemma \ref{nondegeneracy}
we have that
$$
\Gamma_k(u) \cap \Sigma = \bigcup_{j=1}^\infty{E_j}.
$$
The result follows immediately once we prove that $E_j$ is a collection of closed sets. Given $X_0\in \overline{E_j}$, since it satisfies
$$
\frac{1}{j}\rho^k \leq \sup_{\abs{X-X_0}= \rho}{\abs{u(X)}} < j \rho^k,
$$
we need only to show that $X_0\in \Gamma_k(u) \cap \Sigma$, i.e. $N(X_0,u,0^+)=k$. Since $X\mapsto N(X,u,0^+)$ is upper semi-continuous on $\Sigma$, we readily have $N(X_0,u,0^+)\geq k$. On the other hand, if $N(X_0,u,0^+)=k'> k$, we would have
$$
\abs{u(X)}\leq C \abs{X-X_0}^{k'}\quad \mathrm{in }\ B_{1-\abs{X_0}}(X_0),
$$
which contradicts Lemma \ref{growth} and implies that $X_0 \in E_j$.
\end{proof}

\section{Hausdorff dimension estimates for the nodal set}\label{sec.strat}
In this Section we prove different estimates on the Hausdorff dimension of the sets $\Gamma(u)$ and $\Gamma(u)\cap \Sigma$. In the latter, we improve our analysis taking care of the regular and singular part of the restricted nodal set $\Gamma(u)\cap \Sigma$. To start with, given $a \in (-1,1)$ and $u$ be an $L_a$-harmonic function in $B_1$, let us split the nodal set $\Gamma(u)$ in its regular part
\begin{equation}\label{regular.serve}
 \mathcal{R}(u) = \left\{X_0 \in \Gamma(u)\left\lvert
  \begin{array}{cc}
    N(X_0,u,0^+)=1 & \mbox{ if }X_0 \not\in \Sigma \\
    N(X_0,u_e,0^+)=1 \mbox{ or } N(X_0,u_o,0^+)=1-a & \mbox{ if }X_0 \in \Sigma
  \end{array} \right.\right\},
\end{equation}
 and its singular part
\begin{equation}\label{singular.serve}
 \mathcal{S}(u)   = \left\{X_0 \in \Gamma(u)\left\lvert
  \begin{array}{cc}
    N(X_0,u,0^+)\geq 2 & \mbox{ if }X_0 \not\in \Sigma \\
    N(X_0,u_e,0^+)\geq 2 \mbox{ and } N(X_0,u_o,0^+)\geq 2-a & \mbox{ if }X_0 \in \Sigma
  \end{array} \right.\right\}.
\end{equation}
The main idea is to apply a version of the Federer's Reduction Principle as stated in \cite[Appendix A]{Simon83}. More precisely, given a class $\mathcal{F}$ of functions invariant under rescaling and translation and a map $\mathcal{S}$ which associates to each function a subset of $\R^n$, by the Reduction principle we can establish conditions on $\mathcal{F}$ and $\mathcal{S}$ which imply that to control the Hausdorff dimension of $\mathcal{S}(u)$ for every $u \in \mathcal{F}$, we just need to control the Hausdorff dimension of $\mathcal{S}(u)$ for elements which are homogeneous of some degree.
\begin{theorem}[Federer's Reduction Principle] Let $\mathcal{F}\subseteq L^\infty_{\loc}(\R^{n+1})$ and define, for any given $u\in\mathcal{F}$, $X_0\in\R^{n+1}$ and $r>0$, the rescaled and translated function
$$
u_{X_0,r}:=u(X_0+r\,\cdot).
$$
We say that $u_n\to u$ in $\mathcal{F}$ if and only if $u_n\to u$ uniformly on every compact set of $\R^{n+1}$. Moreover, let us assume that $\mathcal{F}$ satisfies the following conditions:
\begin{enumerate}
\item[$\mathrm{(F1)}$]  \emph{(Closure under rescaling, translation and normalization)} Given any $\abs{X_0}\leq 1-r, 0< r, \rho>0$ and $u\in\mathcal{F}$, we have that $\rho\,u_{X_0,r}\in\mathcal{F}$.
\item[$\mathrm{(F2)}$]  \emph{(Existence of a homogeneous blow- up)}
Given $\abs{X_0}<1, r_k \searrow 0$ and $u\in\mathcal{F}$, there exists a sequence $\rho_k\in (0,\infty)$, a real number $\alpha\geq 0$ and a function $\overline{u}\in\mathcal{F}$ $\alpha$-homogenous such that, if we define $u_k(x)=u(X_0 +r_k x)/\rho_k$ then ,up to a subsequence, we have
$$
u_k\to \overline{u} \quad \mbox{in }\mathcal{F}.
$$
\item[$\mathrm{(F3)}$]  \emph{(Singular Set hypotheses)}
There exists a map $\overline{\mathcal{S}}\colon \mathcal{F}\to \mathcal{C}$, where
$$
\mathcal{C}:=\{A\subset \R^{n+1}: A\cap B_1(0) \mbox{ is relatively closed in }B_1(0) \}
$$
such that
\begin{itemize}
\item[$\mathrm{(1)}$] Given $\abs{X_0}\leq 1-r, 0<r<1$ and $\rho>0$, there holds
$$
\overline{\mathcal{S}}(\rho\,u_{X_0,r})=\left(\overline{\mathcal{S}}(u)\right)_{X_0,r}:=\frac{\overline{\mathcal{S}}(u)-X_0}{r}.
$$
\item[$\mathrm{(2)}$] Given $\abs{X_0}<1, r_k \searrow 0$ and $u,\overline{u}\in\mathcal{F}$ such that there exists $\rho_k>0$ satisfying $u_k:=\rho_k u_{X_0,r_k}\to \overline{u}$ in $\mathcal{F}$, the following property holds:
    \begin{gather*}
    \forall \varepsilon >0,\, \exists k=k(\varepsilon)>0 \emph{ such that for every }k\leq k(\varepsilon)\\
    \overline{\mathcal{S}}(u_k) \cap B_1(0) \subseteq \{x\in\R^{n+1}: \mbox{dist}(x,\overline{\mathcal{S}}(\overline{u}))<\varepsilon \}.
    \end{gather*}
\end{itemize}
\end{enumerate}
Then, if we define
\begin{equation}\label{federer}
\begin{aligned}
d :=&\,\, \max\big\{ \mathrm{dim}\,V: V \emph{ is a vector subspace of }\R^{n+1}\emph{ and there exists } u\in\mathcal{F}\emph{ and }\alpha\geq 0\\
&\,\, \emph{such that }\overline{\mathcal{S}}(u)\neq \emptyset \emph{ and } u_{y,r}=r^\alpha u,\,\forall y \in V, \, r>0\big\},
\end{aligned}
\end{equation}
either $\overline{\mathcal{S}}(u)\cap B_1(0) = \emptyset$ for every $u\in\mathcal{F}$ or else $\mathrm{dim}_{\mathcal{H}}\left(\overline{\mathcal{S}}(u) \cap B_1(0)\right)\leq d$ for every $u\in\mathcal{F}$. Furthermore in the latter case there exists a function $\varphi\in\mathcal{F}$, a d-dimensional subspace $V \subseteq \R^{n+1}$ and a real number $\alpha\geq 0$ such that
$$
\varphi_{Y,r}=r^\alpha \varphi\ \emph{ for all }\ Y\in V, r>0 \ \emph{ and }\ \overline{\mathcal{S}}(\varphi)\cap B_1(0) = V \cap B_1(0)
$$
At last if $d=0$ then $\overline{\mathcal{S}}(u)\cap B_{\rho}(0)$ is a fine set for each $u\in\mathcal{F}$ and $0<\rho<1$.
\end{theorem}
We will apply this general result due to Federer in order to deive some estimates on the Hausdorff dimension of the nodal set $\Gamma(u)$ and on its restriction $\Gamma(u)\cap \Sigma$. In the second case, we improve our analysis introducing its regular and singular part on $\Sigma$.

  \begin{theorem}\label{Federer.La}
  Let $a \in (-1,1)$ and $u$ be an $L_a$-harmonic function in $B_1$. Then $\mathrm{dim}_{\mathcal{H}}(\Gamma(u))\leq n.$
  \end{theorem}
  \begin{proof}
  A preliminary remark is that we only need to prove the Hausdorff dimensional estimates for the localization of the sets in $K\subset\subset B_1$, since the general statement follows because a countable union of sets with Hausdorff dimension less than or equal to some $n\in \R^+_0$ also has Hausdorff dimension less than or equal to $n$. Let us consider the class of functions $\mathcal{F}$ defined as
$$
\mathcal{F}=\Big\{ u \in L^\infty_\loc(\R^{n+1})\setminus \{ 0\}\colon L_a u =0 \mbox{ in }B_r(X_0), \mbox{ for some }r \in \R,\, X_0 \in \R^{n+1}\mbox{ with }B_r(X_0)\subset B_1
\Big\}.
$$
By the linearity of the $L_a$ operator, we already know that the closure under rescaling, translation and normalization and assumption $\mathrm{(F1)}$ are all satisfied. \\On the other hand, let $\abs{X_0}< 1, r_k \downarrow 0^+$ and $u \in \mathcal{F}$, and choose $\rho_k = \norm{u(X_0+r_k\cdot x)}{L^{2,a}(\partial B_1)}$. Theorem  \ref{blowup.convergence} and Proposition \ref{blow.N} yield the existence of a blow-up limit $\varphi^{X_0}\in \mathcal{F}$, i.e. normalized tangent map of $u$ at $X_0$, such that, up to a subsequence,
$u_k \to \varphi^{X_0}$ in $\mathcal{F}$ and $\varphi^{X_0}$ is a homogeneous function of degree $k= N(X_0,u, 0+)\geq \min\{1,1-a\}$.
Hence also $\mathrm{(F2)}$ holds.\\
Now, let us consider $\overline{\mathcal{S}}\colon u \mapsto \Gamma(u)$. By the continuity of $u$, we already know that the set $\Gamma(u)\cap B_1$ is obviously closed in $B_1$ and it is quite straightforward
to check that the two hypotheses in $\mathrm{(F3)}$ are satisfied.\\
Hence, in order to conclude the analysis, the only thing left to prove is that the integer $d$ in \eqref{federer} is equal to $n$.
Suppose by contradiction that $d = n+1$, then this would imply the existence of $\varphi \in \mathcal{F}$ with $\overline{\mathcal{S}}(\varphi) = \R^{n+1}$ i.e., $\varphi \equiv 0$ on $\R^{n+1}$, which contradicts the fact the fact the $\Gamma(\varphi)$ has empty interior.\\
Actually, taking $V= \R^{n-1}\times \{0\}\times \R$ and $\varphi(X)=\langle X, e_n\rangle$, we obtain the claimed estimate on $d$.
\end{proof}

Next we prove a different stratification result for the set $\Gamma(u)\cap \Sigma$, which will enlighten the different structure of the nodal set in comparison with the uniformly elliptic case. In particular, with this analysis we want to point out how the different classes of blow-up influence the stratification on the characteristic manifold $\Sigma$.
Obviously, by Proposition \ref{empty.interior} we already know that
$$
\Gamma(u)\cap \Sigma = \Gamma(u_e)\cap \Sigma,
$$
and it is either equal to $\Sigma$ or with empty interior in $\Sigma$. Inspired by this fact, since we are dealing with the restriction of the nodal set on the characteristic manifold $\Sigma$, we will concentrate our attention on the trace of $u$ on $\Sigma$, which is actually equal to the trace of $u_e$ itself.
\begin{theorem}\label{federer.theorem}
  Let $a \in (-1,1)$ and $u$ be an $L_a$-harmonic function in $B_1$. If $\Gamma(u)\cap \Sigma \neq \Sigma$, then, under the previous notations, we have
  $\mathrm{dim}_{\mathcal{H}}(\Gamma(u)\cap \Sigma) \leq n-1$ and more precisely
  $$
  \mathrm{dim}_{\mathcal{H}}(\mathcal{R}(u)\cap \Sigma) = n-1 \quad \mbox{and} \quad \mathrm{dim}_{\mathcal{H}}(\mathcal{S}(u)\cap \Sigma) \leq n-1.
  $$
\end{theorem}
\begin{proof}
Let us consider the class of functions $\mathcal{F}$ defined as
$$
\mathcal{F}=\Bigg\{ u \in L^\infty_\loc(\R^{n+1})\setminus \{ 0\}\Bigg\vert \,\,
\begin{aligned}
 &L_a u =0 \mbox{ in }B_r(X_0), \mbox{ for some }r \in \R,\, X_0 \in \R^{n+1}\\
  &u\mbox{ symmetric with respect to }\Sigma
 \end{aligned}
 \Bigg\}.
$$
Since the functions in $\mathcal{F}$ are symmetric with respect to $\Sigma$ and nontrivial, the condition $\Gamma(u)\cap B_r(X_0)\cap \Sigma \neq \Sigma$ is always satisfied.\\
As before, we already know that the closure under rescaling, translation and normalization and assumption $\mathrm{(F1)}$ and $\mathrm{(F2)}$ are all satisfied.
Moreover, by \eqref{regular.serve} and \eqref{singular.serve} we obtain
\begin{align*}
  \mathcal{R}(u)\cap \Sigma &=  \{X_0 \in \Gamma(u)\cap \Sigma\colon N(X_0,u,0^+)=1\}, \\
  \mathcal{S}(u)\cap \Sigma &= \bigcup_{k\geq 2}\Gamma_k(u)\cap \Sigma=\bigcup_{k\geq 2}\left\{X_0 \in \Gamma(u)\cap \Sigma\colon  N(X_0,u,0^+)=k \right\}
  \end{align*}
since we are dealing with functions in the class $\mathcal{F}$. \\Now, we choose the map $\overline{\mathcal{S}}$ in $\mathrm{(F3)}$ according to our needs.
\\
{\textbf{1. Dimensional estimate of }${\mathbf{\Gamma(u)\cap \Sigma}}$}\\
First, let us consider $\overline{\mathcal{S}}\colon u \mapsto \Gamma(u)\cap \Sigma$. By the continuity of $u$, we already know that the set $\Gamma(u)\cap \Sigma\cap B_1$ is obviously closed in $B_1$ and it is quite straightforward
to check the two hypothesis in $\mathrm{(F3)}$.
Therefore, in order to conclude the analysis of $\Gamma(u)\cap \Sigma$, the only thing left to prove is that the integer $d$ in \eqref{federer} is equal to $n - 1$.\\
Suppose by contradiction that $d = n$, this would implies the existence of $\varphi \in \mathcal{F}$ such that $\overline{\mathcal{S}}(\varphi) = \R^n$ i.e., $\varphi \equiv 0$ on $\Sigma$. Since $\varphi$ solves
$$
\begin{cases}
  L_a \varphi=0 & \mbox{in }\R^{n+1} \\
  \varphi =0 & \mbox{on }\Sigma \\
  \partial^a_y \varphi =0 & \mbox{on }\Sigma,
\end{cases}
$$
it implies that $\varphi \equiv 0$ on the whole $\R^{n+1}$, which contradicts the fact the $0 \not\in \mathcal{F}$.
Actually, by taking $V= \R^{n-1}\times {(0,0)}\subset \Sigma$ and $\varphi(X)=\langle X, e_n\rangle$, we obtain the claimed estimate on $d$.
\\
{\textbf{2. Dimensional estimate of }${\mathbf{\mathcal{R}(u)\cap \Sigma}}$}\\
Let us consider $\overline{\mathcal{S}}\colon u \mapsto \mathcal{R}(u)\cap \Sigma$. Since we are dealing just with symmetric function with respect to $\Sigma$, by Lemma \ref{gap.even} we obtain that necessary $N(X_0,u,0^+)=1$ for every $X_0 \in \mathcal{R}(u)$. By the inclusion, we already know that
$$
\mathrm{dim}_{\mathcal{H}}(\mathcal{R}(u)\cap \Sigma\cap B_1)\leq n-1.
$$
Finally, we can apply the Reduction principle since $\mathrm{(F3)}$ is completely satisfied. More precisely, for $X_0 \in \Sigma\cap B_1, \rho>0$ and $t>0$ if $X \in \mathcal{R}(\rho u_{X_0,t}) \cap \Sigma$ then obviously $X_0 +t X \in \mathcal{R}(u) \cap \Sigma$, i.e. $N(X_0+tX,u,0^+)=1$. Secondly, given $u_i,\overline{u} \in \mathcal{F}$ as in $\mathrm{(F3)}$, suppose by contradiction that there exists a sequence $X_i \in \Sigma \cap B_1$ and $\overline{\varepsilon}>0$ such that
$$
N(X_i,u_i,0^+)=1
$$
and $\mbox{dist}(X_i,\overline{\mathcal{S}}(\overline{u}))\geq \overline{\varepsilon}$. Since, up to a subsequence, $X_i \to \overline{X}$, by the upper semi-continuity of the Almgren frequency formula, we already know that $N(\overline{X},\overline{u},0^+)\geq 1$. Moreover, up to a subsequence, $X_i \to \overline{X} \in \Gamma(\overline{u})\cap \Sigma \cap \overline{B_1}$ by the $L^\infty_{\loc}$ convergence of $u_i \to \overline{u}$.
The contradiction follows from the same argument of the proof of the second case of Theorem \ref{Federer.La}.
\\ More precisely, since $\Gamma(\overline{u})\cap \Sigma$ is a conical set, i.e. for every $\lambda >0$ and $\overline{X}\in \Gamma(\overline{u})\in \Sigma$ we have $\lambda \overline{X} \in \Gamma(\overline{u})\cap \Sigma$, we deduce that if we can prove $\overline{X} \in \mathcal{R}(\overline{u})\cap \overline{B_1}\cap \Sigma$ we provide a contradiction, more precisely we obtain  $\mbox{dist}(\overline{X}, \overline{\mathcal{S}}(\overline{u})\cap B_1) = 0$.
Since there exists $\Omega \subset\subset B_1\setminus \Sigma$ such that $(X_0+r_i X_i)_i \subset \Omega$, if we consider
\begin{align*}
R_1 &= \min_{p \in \overline{\Omega}}\mbox{dist}(p,\partial B_1),\\
\overline{C} &= \sup_{p \in \overline{\Omega}} N(p,u,R_1),
\end{align*}
we easily obtain from Corollary \ref{doubling.corollary.Sigma} that for $p\in \Omega \cap \mathcal{R}(u)$ and $r<R_1$ we have
$$
N(p,u,r)\leq N(p,u,R_1) \left(\frac{R_1}{r}\right)^{n+a-1+2\overline{C}}\leq \overline{C}\frac{1}{r^{n+a-1+2\overline{C}}}.
$$
In particular, from the previous inequality we obtain that there exists $\overline{R}=\overline{R}(n,a,X_0,\varepsilon)>0$ sufficiently small, such that for $r<\overline{R}$ we have
$$
1\leq N(X_i,u_i,r) \leq 1+ \frac{1}{4}.
$$
Since $\lim_i N(X_i,u_i,r) = N(\overline{X},\overline{u},r)$ for sufficiently small $r$, we directly obtain from Lemma \ref{gap.even} that $N(\overline{X},\overline{u},0^+)=1$, as we claimed.\\
\\
As before, let us suppose now that there exist $\varphi \in \mathcal{F}$ and a $d$-dimensional subspace $V\subset \R^{n+1}$, with $d\leq n-1$, and $k\geq 0$ such that
$$
\varphi_{Y,r}=r^k \varphi\ \mbox{ for all }\ Y\in V, r>0 \ \mbox{ and }\ \mathcal{R}(\varphi)\cap \Sigma\cap B_1 = V \cap B_1
$$
Since $\varphi \in \mathfrak{sB}^a_k(\R^{n+1})$ is homogenous of degree $k$ with respect to any $Y \in V=\mathcal{R}(\varphi)\cap \Sigma$, namely $N(Y,\varphi,0^+)=k$, we obtain that necessary $k=1$ and that $\mathcal{R}(\varphi)\cap \Sigma$ is $d$-dimensional. Since every homogenous $L_a$-harmonic function of order $k=1$ is one dimensional, i.e. there exists $\nu \in S^{n-1}$ and $C>0$ such that either
$$
\varphi(X)=C\langle X,(\nu,0)\rangle,\mbox{ for every } X=(x,y) \in \R^{n+1},
$$
we obtain that $\mathcal{R}(\varphi)\cap \Sigma$ must be $(n-1)$-dimensional, and consequently that
$$
\mathrm{dim}_{\mathcal{H}}(\mathcal{R}(u)\cap \Sigma\cap B_1)= n-1.
$$
\\
{\textbf{3. Dimensional estimate of }${\mathbf{\mathcal{S}(u)\cap \Sigma}}$}\\
Let us focus on the singular strata
$$
\mathcal{S}(u)\cap \Sigma = \bigcup_{k\geq 2}\left\{X_0 \in \Gamma(u)\cap \Sigma\colon  N(X_0,u,0^+)=k \right\}
$$
Hence, given $\overline{\mathcal{S}}\colon u \mapsto \mathcal{S}(u)$, the map satisfies $\mathrm{(F3)}$, since for $X_0 \in \Sigma\cap B_1, \rho>0$ and $t>0$, if $X \in \overline{\mathcal{S}}(\rho u_{X_0,t})$ we obtain
$$
N(X,\rho u_{X_0,t}, 0^+) =k
\longleftrightarrow
N(X_0 + t X,u, 0^+) =k,
$$
which is equivalent to $X_0 + t X \in \Gamma_k(u)\subset \mathcal{S}(u)$. Now, given $u_i = \rho_i u_{X_0,r_i},\overline{u} \in \mathcal{F}$ as in $\mathrm{(F3)}$, suppose by contradiction that there exists a sequence $X_i \in B_1$ and $\overline{\varepsilon}>0$ such that, up to a subsequence, $X_i \to \overline{X}$ and
\begin{equation}\label{eq2}
N(X_i,u_i,0^+)=k
\end{equation}
and $\mbox{dist}(X_i,\overline{\mathcal{S}}(\overline{u}))\geq \overline{\varepsilon}$. By the upper semi-continuity of the Almgren frequency formula, we already know that $N(\overline{X},\overline{u},0^+)\geq k$. Since $X_i \in \Gamma_k(u_i)$, there exists $\Omega \subset\subset B_1\setminus \Sigma$ such that $(X_0+r_i X_i)_i \subset \Omega$, if we consider
\begin{align*}
R_1 &= \min_{p \in \overline{\Omega}}\mbox{dist}(p,\partial B_1),\\
\overline{C} &= \sup_{p \in \overline{\Omega}} N(p,u,R_1),
\end{align*}
we easily obtain from Corollary \ref{doubling.corollary.Sigma} that for $p\in \Omega \cap \Gamma_k(u)$ and $r<R_1$ we have
$$
N(p,u,r)\leq N(p,u,R_1) \left(\frac{R_1}{r}\right)^{n+a-1+2\overline{C}}\leq \overline{C}\frac{1}{r^{n+a-1+2\overline{C}}}.
$$
In particular, from the previous inequality we obtain that there exists $\overline{R}=\overline{R}(n,a,X_0,\varepsilon)>0$ sufficiently small, such that for $r<\overline{R}$ we have
$$
k\leq N(X_i,u_i,r) \leq k+ \frac{1}{4}.
$$
Since $\lim_i N(X_i,u_i,r) = N(\overline{X},\overline{u},r)$ for sufficiently small $r$, we directly obtain from Lemma \ref{gap.even} that $N(\overline{X},\overline{u},0^+)=k$, as we claimed.\\
Since $\mathcal{S}(u)\cap \Sigma \subseteq \Gamma(u)\cap \Sigma$, we already know that
$$
\mathrm{dim}_{\mathcal{H}}(\mathcal{S}(u)\cap \Sigma\cap B_1) \leq n-1,
$$
which is actually the optimal bound even for the singular set. Indeed, since there exists $\varphi \in \mathcal{F}$, a $(n-1)$-dimensional subspace $V\subset \Sigma$ and $k\geq 0$ such that
$$
\varphi_{Y,r}=r^k \varphi\ \mbox{ for all }\ Y\in V, r>0 \ \mbox{ and }\ \mathcal{S}(\varphi)\cap \Sigma\cap B_1 = V \cap B_1.
$$
In particular, for every $k\geq 2$, $n\geq 1$ it can be seen by taking $V=\R^{n-1}\times \{0,0\}$ and
$$
\varphi(X)=\ddfrac{(-1)^{\frac{k}{2}} \Gamma\left(\frac{1}{2}+\frac{a}{2}\right)}{2^{k}\Gamma\left(1+\frac{k}{2}\right)\Gamma\left(\frac{1}{2}+\frac{a}{2}+\frac{k}{2}\right)} \ _2F_1\left(-\frac{k}{2}, -\frac{k}{2}-\frac{a}{2}+\frac{1}{2}, \frac{1}{2}, -\frac{\langle X, e_n\rangle^2}{\langle X, e_y\rangle^2}\right)\langle X, e_y\rangle ^{k},
$$
as it was previously proved in Section \ref{section.blowup}.
\end{proof}
\section{Regularity of the Regular and Singular strata}\label{singular.sec}
In this Section we show some results about the regularity of the  regular and singular strata of the nodal set $\Gamma(u)$. As in Section \ref{sec.strat}, we will consider first the stratification in $\R^{n+1}$ of the whole nodal set $\Gamma(u)$, while in the second case we will focus the attention on the restriction $\Gamma(u)\cap \Sigma$ of the nodal set on the characteristic manifold.\\
The main idea of this stratification is to classify the nodal points and then to stratify the nodal set by the spines of the normalized tangent maps, i.e. the largest vector space that leaves the tangent map invariant. Indeed, we will introduce the subset $\Gamma^j_k(u)$ as the set of points at which every tangent map has at most $j$ independent directions of translation invariance in order
to correlate the nodal set of $u$ with the dimension of the set where the tangent map $\varphi^{X_0}$ vanishes with the same order of $u$. \\Moreover, by Theorem \ref{Federer.La} we already know that $\Gamma_k^j(u)$ is well defined for $j\leq n-1$.\\

More precisely, if $k\geq \min\{2,2-a\}$ given
$$
\Gamma_k (u) = \left\{ X_0 \in \Gamma(u)\colon N(X_0,u,0^+)=k\right\}
$$
for each $j= 0,\dots,n-1$ let us define
$$
\Gamma^j_k(u) = \left\{ X_0 \in \Gamma_k(u) \colon \mbox{dim}\,\Gamma_k(\varphi^{X_0}) = j \right\},
$$
where $\varphi^{X_0}$ is the unique normalized tangent limit of $u$ at $X_0$. Obviously, since the uniformly elliptic case is well studied, we focus on the structure of the nodal set $\Gamma(u)$ near $\Sigma$.\\
Before to continue our analysis, let us prove that the concept of dimension in well defined.
\begin{lemma}\label{subspace}
 Given $a \in (-1,1)$, for every $\varphi \in \mathfrak{B}^a_k(\R^{n+1})$, the singular set $\Gamma_k(\varphi)$ of order $k \geq \min\{2,2-a\}$ is the largest vector subspace on $\Sigma$ which leaves $\varphi$ and $N(\cdot, \varphi,0^+)$ invariant, i.e.
  $$
  \Gamma_k(\varphi)= \left\{Z \in \R^{n+1}\colon
  \varphi(X+Z)=\varphi(X) \ \mbox{for every }X \in \R^{n+1}\right\}.
  $$
 \end{lemma}
\begin{proof}
We can restrict our proof to the case $\varphi \in \mathfrak{sB}^a_k(\R^{n+1})$ for $k\geq 2$, since by Corollary \ref{anti.sym} we can easily extend the analysis to the antisymmetric case. Thus, we already know by Corollary \ref{gamma_k} that since $\varphi \in \mathfrak{sB}_k^a(\R^{n+1})$ we have
 $$
  \Gamma_k(\varphi)= \left\{X \in \R^{n+1}\colon
  D^\nu \varphi(X)=0 \ \mbox{for any }\abs{\nu} \leq k-1\right\}.
  $$
Obviously $0 \in \Gamma_k(\varphi)$ by the homogeneity of $\varphi$ and we claim that for every $Z \in \Gamma_k(\varphi)$
$$
\varphi(X)= \varphi(X+Z), \quad \mbox{for all }X \in \R^{n+1},
$$
in other words $\Gamma_k(\varphi)$ leaves the map $\varphi$ invariant. Hence, let $Z \in \Gamma_k(\varphi)$, i.e.
\begin{equation}\label{us1}
D^\nu \varphi(Z) = 0 \quad \mbox{for any }\abs{\nu}\leq k-1
\end{equation}
and write the homogenous polynomial $\varphi \in C^\infty$ as
$$
\varphi(X) = \sum_{\abs{\nu}=k} a_\nu X^\nu,
$$
where $X^\nu = x_1^{\nu_1} \cdot x_2^{\nu_2} \cdots y^{\nu_{n+1}}$ and $a_{\nu} \in \R$. By \eqref{us1} we directly obtain that
$$
\varphi(X)=\sum_{\abs{\nu}=k} a_\nu (X-Z)^\nu,
$$
which implies the claimed invariance. 
Since $\varphi$ is $k$-homogenous, for every $\lambda >0$ and $X \in \R^{n+1}$
\begin{align*}
\varphi(X)  &=\varphi (X-Z) \\
 &= (\lambda+1)^k \varphi\left(\frac{X-Z}{\lambda+1}\right)\\
 &= (\lambda+1)^k \varphi\left(Z+\frac{X-Z}{\lambda+1}\right)\\
 &= \varphi(X+\lambda Z),
\end{align*}
therefore, we obtain $D^\nu \varphi(\lambda Z)=0$ for any $\abs{\nu}\leq k-1$, i.e. $\lambda Z \in \Gamma_k(\varphi)$.\\
Similarly, noticing that for any $Z,W \in \Gamma_k(\varphi)$ we  have $\varphi(Z+W+X) = \varphi(W+X)=\varphi(X)$ for any $X \in \R^{n+1}$, we obtain $Z+W \in \Gamma_k(\varphi)$.
\end{proof}
\begin{definition}
Let $a \in (-1,1)$ and $u$ be an $L_a$-harmonic function in $B_1$. We call $d^{X_0}$ the dimension of $\Gamma_k(u)$ at $X_0 \in \Gamma_k(u)$ as
  \begin{align*}
  d^{X_0} &=\mbox{dim}\,\Gamma_k(\varphi^{X_0})\\
  &=\mbox{dim}\left\{ \xi \in \R^{n+1} \colon \langle \xi, \nabla_{X}\varphi^{X_0}(X)\rangle = 0 \mbox{ for all } X \in \R^{n+1}\right\}.
  \end{align*}
\end{definition}
Following the previous notations we obtain $\Gamma_k^j(u)= \{X_0 \in \Gamma_k(u)\colon d^{X_0}=j \}$.

Hence, given $a \in (-1,1)$ and $u$ be an $L_a$-harmonic function in $B_1$. Then, let us split the nodal set $\Gamma(u)$ in its regular part
$$
 \mathcal{R}(u) = \left\{X_0 \in \Gamma(u)\left\lvert
  \begin{array}{cc}
    N(X_0,u,0^+)=1 & \mbox{ if }X_0 \not\in \Sigma \\
    N(X_0,u_e,0^+)=1 \mbox{ or } N(X_0,u_o,0^+)=1-a & \mbox{ if }X_0 \in \Sigma
  \end{array} \right.\right\},
$$
 and its singular part
 $$
 \mathcal{S}(u)   = \left\{X_0 \in \Gamma(u)\left\lvert
  \begin{array}{cc}
    N(X_0,u,0^+)\geq 2 & \mbox{ if }X_0 \not\in \Sigma \\
    N(X_0,u_e,0^+)\geq 2 \mbox{ and } N(X_0,u_o,0^+)\geq 2-a & \mbox{ if }X_0 \in \Sigma
  \end{array} \right.\right\}.
 $$
 As we previously remarked, the regular set $\mathcal{R}(u)$ is well defined in such a way, for every $X_0\in \Gamma(u)\cap \Sigma$ such that $N(X_0,u_e,0^+)=1$ or $N(X_0,u_o,0^+)=1-a$ must exist a sequence of point $(X_i)_i \in \Gamma(u)\setminus\Sigma$ such that $N(X_i,u,0^+)=1$ and $X_i\to X_0$. The following result gives a generalization in the context of degenerate-singular operator of the concept of regular hypersurface as the set of points where the function vanishes away from its critical set.
 \begin{theorem}\label{regular.ext}
   Let $a \in (-1,1), a \neq 0$ and $u$ be an $L_a$-harmonic function in $B_1$. Then the regular set $\mathcal{R}(u)$ is locally a $C^{k,r}$ hypersurface on $\R^{n+1}$ in the variable $(x,y\abs{y}^{-a})$ with
   $$
   k=\left\lfloor \frac{2}{1-a}\right\rfloor \quad\mbox{and}\quad r = \frac{2}{1-a} - \left\lfloor \frac{2}{1-a}\right\rfloor.
   $$
   Moreover, we have that
   \begin{equation}\label{gen}
   \mathcal{R}(u)= \left\{ X \in \Gamma(u)\colon \abs{\nabla_x u(X)}^2 + \abs{\partial^a_y u(X)}^2 \neq 0 \right\}.
   \end{equation}
 \end{theorem}
 \begin{proof}
   Let us start by proving the characterization of the regular set in terms of the derivatives of the $L_a$-harmonic function $u$. By \eqref{decompos}, there exist $u_e^a\in H^{1,a}(B_1), u_e^{2-a} \in H^{1,2-a}(B_1)$ respectively $L_a$ and $L_{2-a}$-harmonic function in $B_1$, symmetric with respect to $\Sigma$, such that
   $$
   u(X)=u_e^a(X)+u_e^{2-a}(X)y\abs{y}^{-a}\quad\mbox{in }B_1.
   $$
   For every $i=1,\dots,h$, differentiating the previous equality, we obtain
   \begin{align}
   \partial_{x_i} u(X)&=\partial_{x_i} u_e^a(X)+\left(\partial_{x_i} u_e^{2-a}(X)\right)y\abs{y}^{-a}\label{deriv.i}\\
   \partial_{y}^a u(X)&=\left((1-a)u_e^{2-a}(X) + y\partial_y u_e^{2-a}\right) + \partial^a_y u_e^a, \label{deriv.y}
   \end{align}
   where we split the two functions as sum of their symmetric and antisymmetric part. If $X_0\in \mathcal{R}(u)\setminus\Sigma$ the condition in \eqref{gen} is obviously satisfied by the local uniformly elliptic regularity outside $\Sigma$. Instead, if $X_0 \in \mathcal{R}(u)\cap \Sigma$, if $N(X_0,u_e,0^+)=1$ it follows
   $$
   u_e(X)=\varphi^{X_0}_e(\nu^{X_0})\langle X-X_0, \nu^{X_0} \rangle + o(\abs{X-X_0}),
   $$
   for some $\nu^{X_0} \in S^{n-1} = S^n \cap \Sigma$, and by Theorem \ref{uniqueness} and \eqref{deriv.i} we obtain
   $$
   \partial_{x_i} u(X_0) = \partial_{x_i} u_e(X_0) = \varphi^{X_0}_e(\nu^{X_0})\langle e_i ,\nu^{X_0}\rangle,
   $$
   and by the nondegeneracy of the blow-up limit $\abs{\nabla u(X_0)} = \varphi^{X_0}_e(\nu^{X_0})\neq 0$ (for further details, we remaind to the proof of Theorem \ref{regular.sopra}). Similarly, taking care of the antisymmetric part, if $N(X_0,u_o,0^+)=1-a$ we obtain
   $$
   u_o(X)=\varphi^{X_0}_o(e_y)y\abs{y}^{-a} + o(\abs{X-X_0}^{1-a}),
   $$
   and consequently $\partial^a_y u(X_0)=\partial^a_y u_o(X_0) = (1-a)\varphi^{X_0}_o(e_y)\neq 0$, as we claimed.\\
   Now, let us consider the other part of the Theorem and let us study the regularity of the regular part $\mathcal{R}(u)$. Since the implicit function theorem implies that the nodal set of a smooth function is a smooth hypersurface away from the critical nodal set, we decide to introduce a suitable change of variable.\\
   More precisely, let us introduce the change of variable $\Phi \colon \R^{n+1}\to \R^{n+1}$ such that
   \begin{align*}
\Phi&\colon (x,z)\mapsto  \left(x,(1-a)\displaystyle z \abs{z}^{\frac{a}{1-a}}\right),\\
\Phi^{-1}&\colon (x,y)\mapsto \left(x,\frac{y \abs{y}^{-a}}{(1-a)^{1-a}}\right),
\end{align*}
with Jacobian $\abs{J_{\Phi^{-1}}(x,y)}=(1-a)^a \abs{y}^{-a}$ and $\Phi(X_0)=X_0$, for every $X_0 \in \Sigma$. By \eqref{decompos} and
This change of variable is well known in the literature since it allows to correlate our class of degenerate-singular operator with the class of Baouendi-Grushin Operators (see also \cite{MR2334826}). In particular, since $a\in (-1,1)$, we obtain by simple computations that $\Phi \in C^{k',r'}(\R^{n+1},\R^{n+1})$, with
$$
   k'=\left\lfloor \frac{1}{1-a}\right\rfloor \quad\mbox{and}\quad r' = \frac{1}{1-a} - \left\lfloor \frac{1}{1-a}\right\rfloor.
$$
The previous quantity are well defined since $(1-a)^{-1}>1/2$, for every $a \in (-1,1)$ and it blows up as $a$ approaches $1^-$. \\
Now, given $v(x,z)=u(\Phi(x,z))$, we obtain $\Gamma(u)=\Phi(\Gamma(v))$ and by \eqref{decompos}, \eqref{deriv.i} and \eqref{deriv.y}
\begin{align*}
  v(x,z)&=u_e^a(\Phi(x,z))+u_e^{2-a}(\Phi(x,z))z\\
  \partial_{x_i}v(x,z)&=(\partial_{x_i} u)(\Phi(x,z)),\mbox{ for every } i=1,\dots,h\\
  \partial_z v(x,z)&=(\partial^a_y u)(\Phi(x,z)),
\end{align*}
so in particular $\abs{\nabla v (x,z)}^2 = (\abs{\nabla_x u}^2 + \abs{\partial^a_y u }^2)(\Phi(x,z))$. By \eqref{decompos} and Proposition \ref{smooth} we obtain that given and $L_a$-harmonic function $u$ in $B_1$, since  $u_e^a(\Phi(x,z)),u_e^{2-a}(\Phi(x,z)) \in C^{k',r'}(B_{1/2})$ we obtain that $v \in C^{k',r'}(B_{1/2})$. 
Moreover, as we remarked in Section \ref{section.blowup}, since 
our change of variables $\Phi$ acts only in the $y$-direction, we obtain from Proposition \ref{smooth} and Theorem \ref{uniqueness} that actually $v\in C^{k,r}(B_{1/2})$ with
$$
k = \left\lfloor \frac{2}{1-a}\right\rfloor\geq 1 \quad\mbox{and}\quad r= \frac{2}{1-a} - \left\lfloor \frac{2}{1-a}\right\rfloor.
$$
Now, by the first part of the statement, since $X_0 \in \mathcal{R}(u)\cap\Sigma$ we obtain by Corollary \ref{taylor.generalized}
$$
\abs{\nabla v (X_0)}^2 = \abs{\nabla_x u (X_0)}^2 + \abs{\partial^a_y u(X_0)}^2 = \varphi^{X_0}_e(\nu^{X_0})^2 +(1-a)^2\varphi^{X_0}_o(e_y)^2 \neq 0,
$$
where $\varphi^{X_0}_e$ and $\varphi^{X_0}_o$ are respectively the tangent map of the symmetric and the antisymmetric part of $u$ with respect to $\Sigma$. Since the conclusion follows after an application of the implicit function theorem on the function $v$ and the relation $\Gamma(u)=\Phi(\Gamma(v))$, let us consider three different cases:
\begin{itemize}
  \item[(1)] $N(X_0,u_e,0^+)=1$ and $N(X_0,u_o,0^+)>1-a$, which implies that $\partial_z v(X_0)=0$ and $\nabla_x v(X_0)= \varphi^{X_0}(\nu^{X_0})\nu^{X_0}$. In this case, up to relabeling the $x$-variables, by the implicit function theorem we obtain that there exists $\rho>0$ and $g \in C^{k,r}(B_{\rho}(X_0))$ such that
      $x_1=g(x)=g(x_2,\dots,x_n,z)$ for every $(x,z)\in \Gamma(v)\cap B_\rho (X_0)$. Going back to the $(x,y)$ variables, we obtain $$
      x_1 = g(x_2,\dots,x_n,y\abs{y}^{-a})\mbox{ for every } X \in \Gamma(u)\cap B_{\rho/2}(X_0);
      $$
  \item[(2)] $N(X_0,u_e,0^+)>1$ and $N(X_0,u_o,0^+)=1-a$, in this case since $\partial_{x_i}v(X_0)=0$ for all $i=1,\dots,n$ and $\partial_z v(X_0)\neq 0$ we obtain that there exists $\rho>0$ and $g \in C^{k,r}(B_{\rho}(X_0))$ such that
      $z=g(x)=g(x_1,\dots,x_n)$ for every $(x,z)\in \Gamma(v)\cap B_\rho (X_0)$. Going back to the $(x,y)$ variables, we obtain $$
      y\abs{y}^{-a} = g(x)\mbox{ for every } X \in \Gamma(u)\cap B_{\rho/2}(X_0);
      $$
  \item[(3)] $N(X_0,u_e,0^+)=1$ and $N(X_0,u_o,0^+)=1-a$, we obtain that if $a<0$, by applying the implicit function theorem with respect to the $x$-variables as in case $(1)$, we obtain, up to a rotation on $\Sigma$, that
      $$
      x_1 = g(x_2,\dots,x_n, y \abs{y}^{-a})\mbox{ for every } X \in \Gamma(u)\cap B_{\rho/2}(X_0);
      $$
      where in this case $y\abs{y}^{-a} \in C^{1,-a}_\loc(B_1)$. Otherwise, if $a>0$ by applying the implicit function theorem on the $z$-variable as in $(2)$, we obtain
      $$
      y\abs{y}^{-a} = g(x)\mbox{ for every } X \in \Gamma(u)\cap B_{\rho/2}(X_0),
      $$
      where in the both cases $g \in C^{k,r}(B_\rho(X_0))$.
\end{itemize}
  We remark that the previous records can be changed considering the cases when the minimum between the Almgren frequency of the symmetric and the antisymmetric part of $u$ is achieved by the first or the second one. \\Thus, up to considering a smaller radius on the previous cases, the results on $\mathcal{R}(u)$ are a direct consequence of the local ones on $\Gamma(u)$ near $X_0$, since the regular set is relatively open in $\Gamma(u)$ and hence there exists $\rho>0$ such that $\Gamma(u)\cap B_\rho(X_0)=\mathcal{R}(u)\cap B_\rho(X_0)$.
 \end{proof}
 The previous result explains why the tangent map at a point of the restriction of the nodal set $\Gamma(u)\cap \Sigma$ does not allow to fully understand the geometric picture of the nodal set itself, since we need to take care of both the symmetric and antisymmetric part of $u$.\\
Furthermore, we can describe the local behaviour of the regular set $\mathcal{R}(u)$ near the characteristic manifold by using the tangent field $\Phi^{X_0}$, which contains all the geometric information of the regular set. More precisely, as a direct consequence of the previous reports we obtain
 \begin{corollary}
   Let $a\in (-1,1)$ and $u$ be an $L_a$-harmonic function in $B_1$. Then the regular part $\mathcal{R}(u)$ of the nodal set intersects the characteristic manifold $\Sigma$ either orthogonally or tangentially. More precisely, given $X_0\in \mathcal{R}(u)\cap \Sigma$
   \begin{itemize}
     \item if $N(X_0,u,0^+)=1$ the direction is orthogonal,
     \item if $N(X_0,u,0^+)=1-a$ the direction is tangential.
   \end{itemize}
   Moreover, independently on $a\in (-1,1)$ and on the value of $N(0,\Phi^{X_0},0^+)$, the restriction on $\Sigma$ of $\mathcal{R}(u)$ is completely described by $\varphi^{X_0}_e$.
 \end{corollary}
Instead, since the structure of the singular set is well known outside of the characteristic manifold $\Sigma$, we decided to postpone our analysis and to concentrate our attention to the intersection of the nodal set on $\Sigma$.\\
Hence, in this last part, we extend the previous analysis focusing on the restriction of the regular and singular set on the characteristic manifold. First, since Lemma \ref{subspace} relies on the homogeneity and the regularity of the homogenous polynomial $\varphi \in \mathfrak{sB}_k^a(\R^{n+1})$, we can reasonably introduce the concept of \emph{dimension restricted to $\Sigma$}.
\begin{definition}\label{d.sigma}
Given $a\in(-1,1)$, let $u$ be an $L_a$-harmonic function in $B_1$. We call $d_\Sigma^{X_0}$ the dimension of $\Gamma_k(u)\cap \Sigma$ at $X_0 \in \Gamma_k(u)\cap \Sigma$ as
  \begin{align*}
  d_\Sigma^{X_0} &=\mbox{dim}\,\Gamma_k(\varphi^{X_0})\cap \Sigma\\
  &=\mbox{dim}\left\{ \xi \in \Sigma \colon \langle \xi, \nabla_{x}\varphi^{X_0}(x,0)\rangle = 0 \mbox{ for all } x \in \Sigma\right\}.
  \end{align*}
\end{definition}
Following the previous notations, we define $\Gamma_k^j(u)\cap \Sigma = \{X_0 \in \Gamma_k(u)\cap \Sigma\colon d^{X_0}_\Sigma=j \}$.\\
 In the previous Section, we split the restriction on the nodal set on $\Sigma$ into its regular part
 $$
 \mathcal{R}(u)\cap \Sigma = \{X \in \Gamma(u)\cap \Sigma \colon N(X,u_e,0^+)=1 \},
 $$
 and its singular part
 $$
 \mathcal{S}(u)\cap \Sigma = \{ X \in \Gamma(u)\cap \Sigma \colon N(X,u_e,0^+)\geq 2 \}= \bigcup_{k\geq 2}\Gamma_k(u)\cap \Sigma.
 $$
 \begin{theorem}\label{regular.sopra}
   Let $a \in (-1,1)$ and $u$ be an $L_a$-harmonic function in $B_1$. Then the regular set $\mathcal{R}(u)$ on $\Sigma$ is locally a smooth hypersurface on $\Sigma$ and
   $$
   \mathcal{R}(u)\cap \Sigma = \left\{ X \in \Gamma(u)\cap \Sigma \colon \abs{\nabla_x u_e(X)} \neq 0 \right\}.
   $$
 \end{theorem}
 \begin{proof}
 By Proposition \ref{empty.interior} we already know that
$$
\Gamma(u)\cap \Sigma = \Gamma(u_e)\cap \Sigma,
$$
and it is either equal to $\Sigma$ or with empty interior in $\Sigma$. Inspired by this fact, we will concentrate our attention on the trace of $u$ on $\Sigma$, which is actually equal to the trace of $u_e$ itself. In order to simplify we will just write $u$ instead of $u_e$ assuming the symmetry with respect to $\Sigma$.\\ Suppose that $\Gamma(u_e) \neq \Sigma$, by Theorem \ref{uniqueness} and our blow-up classification, for every $X_0 \in \mathcal{R}(u)\cap \Sigma$ there exists a linear map $\varphi^{X_0} \in \mathfrak{sB}^a_1(\R^{n+1})$ such that
$$
u(X) = \varphi^{X_0}(X-X_0) + o(\abs{X-X_0}) = \varphi^{X_0}(\nu^{X_0})\langle X-X_0, \nu^{X_0} \rangle + o(\abs{X-X_0})
$$
for some $\nu^{X_0} \in S^{n-1} = S^n \cap \Sigma$.\\ Moreover, by Theorem \ref{continuation} we know that the map $X_0 \mapsto \varphi^{X_0}(\nu^{X_0})\nu^{X_0}$ is continuous. Passing through its trace on $\Sigma$, since $\nu \in \Sigma$ we obtain
$$
u(x,0) = \varphi^{X_0}(\nu)\langle x-x_0, \nu \rangle + o(\abs{x-x_0}).
$$
Since by Proposition \ref{smooth} the function $u \in C^\infty (B_{1/2})$, we can use the tangent map in order to compute the directional derivative of $u$, which will implies the nondegeneracy of the gradient on $\Sigma$ of $u$ at $X_0$. More precisely, for every $\xi \in S^{n-1}$
  $$
  \langle \nabla_x u(X_0),\xi\rangle = \left.\frac{d}{dt}u(X_0 +t \xi) \right\lvert_{t=0}
  = \lim_{t \to 0} \frac{u(X_0+ t\xi)}{t}
  = \varphi^{X_0}(\nu^{X_0}) \langle \xi , \nu^{X_0}\rangle,
  $$
  and hence $\nabla_x u(X_0) = \varphi^{X_0}(\nu^{X_0})\nu^{X_0}$ which is nonzero by Theorem \ref{nondegeneracy}. Finally, by the implicit function theorem we obtain the claimed result.
\end{proof}
As we already mentioned, since for $k\geq 2$ we have $\mathfrak{sB}^a_k(\R^{n+1})\setminus\mathfrak{B}^*_k(\R^{n+1}) \neq \emptyset$, we decide to introduce the following singular sets
$$
\mathcal{S}^*(u) = \bigcup_{k\geq 2}\Gamma^*_k(u)\quad \mbox{and}\quad \mathcal{S}^a(u) = \bigcup_{k\geq 2}\Gamma_k^a(u),
$$
where
$$
  \Gamma^*_k(u)=\left\{X_0 \in \Gamma_k(u)\cap \Sigma\colon \varphi^{X_0}_e \in \mathfrak{sB}^*_k(\R^{n+1})\right\}\ \mbox{ and } \ \Gamma^a_k(u)= \left(\Gamma_k(u)\cap \Sigma\right)\setminus \Gamma^*_k(u).
$$
The idea is to stratify the singular set taking care of both the dimension $d^{X_0}_\Sigma$ and the different classes of tangent map associated with the sets $\Gamma^*_k(u)$ and $\Gamma^a_k(u)$.
\begin{theorem}\label{regularity.singular}
Given $a\in(-1,1)$, let $u$ be an $L_a$-harmonic function in $B_1$. Then for $k \in 2+\N$ and $j=0,\cdots,n-1$ the sets $\Gamma_k^j(u)\cap \Sigma$ is contained in a countable union of $j$-dimensional $C^1$ manifolds.
\end{theorem}
\begin{proof}
  The proof of this result follows the strategy of \cite[Theorem 1.3.8]{MR2511747}.
  Since $\varphi^{X_0}$ is a polynomial of degree $k$ on $\Sigma$, we can write the following
$$
\varphi^{X_0}(x,0)=\sum_{\abs{\alpha}=k}\frac{a_\alpha(x_0,0)}{\alpha!}x^\alpha,
$$
where the coefficients $X \mapsto a_\alpha(X)$ are continuous on $\Gamma_k(u)\cap \Sigma$ and, since $u(X)=0$ on $\Gamma_k(u)$,there holds
$$
\abs{\varphi^{X_0}(X-X_0)}\leq \sigma\left(\abs{X-X_0}\right)\abs{X-X_0}^k\quad\mbox{for every }X,X_0 \in K.
$$
For any multi-index $\abs{\alpha}\leq k$, let us introduce for any $X \in \Gamma_k(u)$ the collection
$$
f_\alpha(X)=
\begin{cases}
  a_\alpha(X) & \mbox{if } \abs{\alpha}=k\\
  0 & \mbox{if } \abs{\alpha}<k
\end{cases}.
$$
Let us prove that the compatibility conditions for the Whitney's extension theorem are fully satisfied in order to guarantee the existence of a function $F \in C^k(\R^{n+1})$ such that
$$
\partial^\alpha F = f_\alpha \quad \mbox{on } E_j,
$$
for every $\alpha \leq k$.
More precisely, following \cite{Whi34} our claim is that, for any $X_0,X \in K$, there holds
$$
f_\alpha(X)= \sum_{\abs{\beta} \leq k - \abs{\alpha}} \frac{f_{\alpha+\beta}(X_0)}{\beta!}(X-X_0)^\beta +R_\alpha(X,X_0),
$$
with
\begin{equation}\label{modulus}
  \abs{R_\alpha(X,X_0)} \leq \sigma_\alpha\left(\abs{X-X_0}\right) \abs{X-X_0}^{k-\abs{\alpha}}
\end{equation}
where $\sigma_\alpha = \sigma_\alpha^K$ is a certain modulus of continuity.\\
If $\abs{\alpha} =k$, since $R_\alpha(X,X_0) = a_\alpha(X)-a_\alpha(X_0)$, we infer from the continuity of $X \mapsto \varphi^X$ on $K$ that $\abs{R_\alpha(X,X_0)} \leq \sigma_\alpha\left(\abs{X-X_0}\right)$.
Instead, for $0\leq\abs{\alpha}<k$ we have
\begin{equation}\label{serve2}
R_\alpha(X,X_0)= -\sum_{\substack{\gamma>\alpha\\\abs{\gamma}=k}}\frac{a_\gamma(X_0)}{(\gamma-\alpha)!}(X-X_0)^{\gamma-\alpha} = - \partial^a \varphi^{X_0}(X-X_0).
\end{equation}
By contradiction, suppose that there is no modulus of continuity $\sigma_\alpha$ such that \eqref{modulus} is satisfied for $X,X_0 \in K$. Then, must exist $\delta>0$ and two sequences $X^i,X_0^i \in K$ with $\rho_i = \abs{X^i- X^i_0} \searrow 0$ such that
$$
\abs{\sum_{\substack{\gamma>\alpha \\ \abs{\gamma}=k}}\frac{a_\gamma(X_0)}{(\gamma-\alpha)!}(X-X_0)^{\gamma-\alpha} }\geq \delta \abs{X^i- X^i_0}^{k-\abs{\alpha}}.
$$
Thus, consider the blow-up sequence associated with the sequences $(X_0^i)_i$ and $(\rho_i)_i$ given by
$$
u_i(X)= \frac{u(X_0^i + \rho_i X)}{\rho_i^k}, \quad \xi^i = \frac{X^i-X_0^i}{\rho_i},
$$
where it is not restrictive to assume that $X_0^i \to X_0 \in K$ and $\xi^i \to \xi_0 \in \partial B_1$. By Theorem \ref{continuation} we obtain $u^i \to \varphi^{X_0} \in \mathfrak{B}^a_k(\R^{n+1})$ uniformly on compact set and there exist a modulus of continuity such that
$$
\abs{u_i(X)-\varphi^{X_0^i}(X)} \leq \sigma\left(\rho_i\abs{X}\right) \abs{X}^k.
$$
In particular, since $X_0^i, X^i \in K=E_j$, the inequalities \eqref{serve} holds true for $u^i$ at $0$ and $\xi^i$. Thus, passing to the limit, we obtain that
$$
\frac{1}{j}\rho^k \leq \sup_{\abs{X-\xi_0}= \rho}{\abs{\varphi^{X_0}(X)}} < j \rho^k,
$$
for $0<\rho<+\infty$, which implies that $\xi_0 \in \Gamma_k(\varphi^{X_0})$. Finally, since $\partial^\alpha \varphi^{X_0}(\xi_0)=0$ for $\abs{\alpha}<k$, dividing both the left and the right hand side of \eqref{serve2} by $\rho_i^{k-\abs{\alpha}}$ and passing to the limit, we reach a contradiction since we obtain
$$
\abs{\partial^a \varphi^{X_0}(\xi_0)}= \abs{\sum_{\substack{\gamma>\alpha \\ \abs{\gamma}=k}}\frac{a_\gamma(X_0)}{(\gamma-\alpha)!}(X-X_0)^{\gamma-\alpha} } \geq \delta.
$$
\\
Finally, under the previous notations, let us consider $X_0=(x_0,0) \in \Gamma^j_k(u)\cap E_i$, where $E_i$ is defined in Lemma \ref{Fsigma}.
     Hence, by definition of $d^{X_0}_\Sigma$, there exists $n-d^{X_0}_\Sigma$ linearly independent unit vectors $(\nu_{i})_i \subset S^{n}$, such that
  $$
  \langle \nu_i, \nabla_X \varphi^{X_0}\rangle \neq 0 \mbox{ on }\Sigma,
  $$
  where $d^{X_0}_\Sigma=j$. Hence, there exist multi-indices $\alpha_i$ or order $\abs{\alpha_i}=k-1$ such that
$$
\partial_{\nu_i}D^{\alpha_i}\varphi^{X_0}(0,0) \neq 0.
$$
Since $\varphi^{X_0}$ is a polynomial of degree $k$ on $\Sigma$, we can write the following
$$
\varphi^{X_0}(x,0)=\sum_{\abs{\alpha}=k}\frac{a_\alpha(x_0,0)}{\alpha!}x^\alpha,
$$
where the coefficients $X \mapsto a_\alpha(X)$ are continuous on $\Gamma_k(u)\cap \Sigma$.
Thus, the nondegeneracy condition on $\varphi^{X_0}$ implies
\begin{equation}\label{impli}
\partial_{\nu_i}D^{\alpha_i} F(x_0,0) \neq 0, \quad i=1,\cdots, n-d^{X_0}_\Sigma.
\end{equation}
Finally, since
$$
\Gamma^j_k(u)\cap\Sigma\cap E_i\subset \bigcap_{i=1}^{n-j}\left\{ D^{\alpha_i} F=0\right\}\cap \Sigma,
$$
in view of the implicit function Theorem, the condition \eqref{impli} implies
that $\Gamma^j_k(u)\cap\Sigma\cap E_i$
is contained in a $j$-dimensional manifold in a neighborhood of $X_0$.\\
The results follows immediately from Lemma \ref{Fsigma}
\end{proof}
We remark that in this particular case of $L_a$-harmonic function symmetric with respect to $\Sigma$, since by the definition of tangent map at a point of the nodal set we have
$$
u(x,0) = \sum_{\abs{\alpha}=k}\frac{a_\alpha(x_0,0)}{\alpha!}x^\alpha + o(\abs{x-x_0}^k)
$$
and $u \in C^{\infty}(B_{1/2})$ thanks to Proposition \ref{smooth}, we obtain that $D^\alpha u(x_0,0) = 0$ for $\abs{\alpha}=k-1$ and $D^{\alpha}u(x_0,0)=a_\alpha(x_0,0)$
for $\abs{\alpha}=k$.
Thus, the nondegeneracy condition on $\varphi^{X_0}$ implies
\begin{equation}
\partial_{\nu_i}D^{\alpha_i} u(x_0,0) \neq 0, \quad i=1,\cdots, n-d^{X_0}_\Sigma.
\end{equation}
Hence, we can obtain conclusion just looking at the strata $\{D^{\alpha_i} u=0 \}$, with $i=1,\dots,n-j$. Instead, the previous proof is more general and it will be applied to a more general class of degenerate-singular operators in Section \ref{section.diverg}.\\
The following is the main Theorem of this stratification analysis, in particular it allows to emphasize the degenerate-singular attitude of the operator $L_a$ near the characteristic manifold $\Sigma$ by showing the presence of a $(n-1)$-dimensional singular stratum for $a\in (-1,1)$ with $a\neq 0$.
\begin{theorem}\label{general.singular}
  Given $a\in(-1,1)$, let $u$ be $L_a$-harmonic in $B_1$. Then there holds
  $$
  \mathcal{S}(u)\cap \Sigma = \mathcal{S}^*(u) \cup \mathcal{S}^a(u)
  $$
  where $\mathcal{S}^*(u)$ is contained in a countable union of $(n-2)$-dimensional $C^1$ manifolds and $\mathcal{S}^a(u)$ is contained in a countable union of $(n-1)$-dimensional $C^1$ manifolds. Moreover
  $$
   \mathcal{S}^*(u)=\bigcup_{j=0}^{n-2} \mathcal{S}^*_j(u)\quad\mbox{and}\quad
   \mathcal{S}^a(u)=\bigcup_{j=0}^{n-1}\mathcal{S}^a_j(u),
  $$
  where both $\mathcal{S}^*_j(u)$ and $\mathcal{S}^a_j(u)$ are contained in a countable union of $j$-dimensional $C^1$ manifolds.
\end{theorem}
\begin{proof}
The proof can be seen as an improvement of Proposition \ref{regularity.singular} since it consists on applying the previous strategy for the dimension and the regularity of the set $\Gamma^j_k(u)$ taking care on the case when the tangent map belongs to $\mathfrak{sB}^*_k(\R^{n+1})$ or not. Indeed, this two cases influence the upper bound on the dimension $d^{X_0}_\Sigma$ and consequently the dimension of the singular strata.\\
Hence, let us set
\begin{align*}
   \mathcal{S}^*(u) &=\bigcup_{j=0}^{n-2} \mathcal{S}^*_j(u) =\bigcup_{j=0}^{n-2}\bigcup_{k\geq 2}\{X \in \Gamma_k^*(u)\colon d^{X_0}_\Sigma=j\}, \\
   \mathcal{S}^a(u) &=\bigcup_{j=0}^{n-1}\mathcal{S}^a_j(u)
   =\bigcup_{j=0}^{n-1}\bigcup_{k\geq 2}\{X \in \Gamma_k^a(u)\colon d^{X_0}_\Sigma=j\}.
\end{align*}
Since for every $k\geq 2$ the functions $\varphi \in \mathfrak{sB}^*_k(\R^{n+1})$ are homogeneous polynomial harmonic in $\Sigma$, we have that $\mbox{dim}\left(\mathcal{S}(\varphi)\cap \Sigma\right)\leq n-2$, and consequently $d^{X_0}_\Sigma \leq n-2$ for every $X_0 \in \Gamma_k^*(u)$. \\Similarly, following Proposition \ref{example} and the remarks in the proof of Theorem \ref{federer.theorem}, since for every $k\geq 2$ there exists $\varphi \in \mathfrak{sB}^a_k(\R^{n+1})\setminus\mathfrak{sB}^*_k(\R^{n+1})$ such that $\mbox{dim}\left(\mathcal{S}(u)\cap \Sigma\right)=n-1$ we obtain that, for $X_0 \in \Gamma^a_k(u)$, there holds $d^{X_0}_\Sigma \leq n-1$.\\
Now, by applying the same argument in the proof of Proposition \ref{regularity.singular}, if we set
\begin{align*}
\mathcal{S}^*_j(u) &= \bigcup_{k\geq 2}\left\{X\in \Gamma^*_k(u)\colon d^{X_0}_\Sigma =j\right\} \quad\mbox{for }j=0,\cdots,n-2\\
\mathcal{S}^a_j(u) &= \bigcup_{k\geq 2}\left\{X\in \Gamma^a_k(u)\colon d^{X_0}_\Sigma =j\right\}\quad\mbox{for }j=0,\cdots,n-1,
\end{align*}
we obtain that $\mathcal{S}^*_j(u)$ and $\mathcal{S}^a_j(u)$ are contained in $j$-dimensional $C^1$ manifold.
\end{proof}
Furthermore, by Proposition \ref{example} we obtain that for any $X_0 \in \mathcal{S}^a_{n-1}(u)$  the leading polynomial of $u$ at $X_0$, i.e. the first term of the Taylor expansion of $u$ at $X_0$, is an homogenous polynomial of two variables of the form \eqref{even} or \eqref{odd}, up to a rotation on $\Sigma$.

\section{Fractional power of elliptic operator in divergence form}\label{section.diverg}
In this Section, we apply the previous analysis relating, via the extension technique, the study of the restriction of the nodal set on the characteristic manifold $\Sigma$ to the local properties of solutions of fractional power of  elliptic differential equations in divergence form.  We start by focusing  on the case of the fractional Laplacians $(-\Delta)^s$ and then we discuss the monotonicity formula and its consequences for
solutions of general fractional elliptic differential equations of the second order with Lipschitz leading coefficients.\\
Let $s\in (0,1)$ and $u\colon B_1\subset \R^n\to \R$ be a nontrivial $s$-harmonic function in $B_1$, that is
\begin{equation}\label{Ps}
(-\Delta)^s u(x) = 0 \quad \mbox{in }B_1.
\end{equation}
Here we define the $s$-Laplacian 
$$
(-\Delta)^s u(x)=  C(n,s) \mbox{ P.V.}\int_{\R^n}{\frac{u(x)-u(y)}{\abs{x-y}^{n+2s}}\mathrm{d}y}\;,
$$
where
\begin{equation}\label{eq:Cns}
C(n,s) = \frac{2^{2s}s\Gamma(\frac{n}{2}+s)}{\pi^{n/2}\Gamma(1-s)} \in \left(0, 4\Gamma\left(\frac{n}{2}+1\right) \right]\;.
\end{equation}
In general, the $s$-Laplacian can be defined in various ways, which we review now. First, in order to better understand these definitions, we introduce the spaces
$$
\tilde{H}^s(\R^n)= \left\{u \in L^2(\R^n) \colon \abs{\xi}^s(\mathcal{F}u)(\xi) \in L^2(\R^n) \right\},
$$
where $s\in (0,1)$ and $\mathcal{F}$ denotes the Fourier transform. In the literature, the spaces $\tilde{H}^s(\R^n)$ are called Bessel spaces and in particular they can be equivalently defined as a Sobolev-Slobodeckij spaces. More precisely, fixed $\Omega\subseteq \R^n$ an open set, for every fractional exponent $s \in (0,1)$ we define $H^s(\Omega)$ as the set of all functions $u$ defined on $\Omega$ with a finite norm
$$
\norm{u}{H^s(\Omega)}= \left(\int_{\Omega}{\abs{u}^2\mathrm{d}x} + \frac{C(n,s)}{2} \int_{\Omega}\int_{\Omega}{\frac{\abs{u(x)-u(z)}^2}{\abs{x-z}^{n+2s}}\mathrm{d}x\mathrm{d}z} \right)^{1/2},
$$
where the term
\begin{equation}\label{gagliardo}
\left[u\right]_{H^s(\Omega)}=\left(\frac{C(n,s)}{2}\int_{\Omega}\int_{\Omega}{\frac{\abs{u(x)-u(z)}^2}{\abs{x-z}^{n+2s}}\mathrm{d}x\mathrm{d}z} \right)^{1/2}
\end{equation}
is the so-called Gagliardo seminorm of $u$ in $H^s(\Omega)$. It can be proved that $\tilde{H}^s(\R^n)=H^s(\R^n)$ and in particular, for every $u \in H^s(\R^n)$ we obtain
$$
[u]^2_{H^s(\R^n)} = \int_{\R^n}{\abs{\xi}^{2s}\abs{\mathcal{F}u(\xi)}^2\mathrm{d}\xi} = \norm{(-\Delta)^{s/2}u}{L^2(\R^n)}^2.
$$
Note that one can also define the fractional Laplacian acting on spaces of functions with weaker regularity.\\ More precisely, following \cite{silvestrethesis}, let $\mathcal{S}$ be the Schwartz space of rapidly decreasing smooth functions in $R^n$ and $\mathcal{S}^s(\R^n)$ be the space of smooth function $u$ such that $(1+\abs{x}^{n+2s})D^k f(x)$ is bounded in $\R^n$, for every $k\geq 0$, endowed with the topology given by the family of seminorms
$$
\left[f\right]_k = \sup_{x\in\R^n} \left(1+\abs{x}^{n+2s}\right)D^k f(x).
$$
Under these notations, the fractional Laplacian of $f\in \mathcal{S}$ is well defined in $(-\Delta)^s f\in \mathcal{S}_s$ and, by duality, this allows to define the fractional Laplacian for functions in the space
\begin{align*}
\mathcal{L}^1_s(\R^n) &= \left\{u \in L^1_{\tiny{loc}}(\R^{n}) \colon \int_{\R^n}{\frac{\abs{u(x)}}{(1+\abs{x})^{n+2s}}\mathrm{d}x}< +\infty \right\}\\
&= L^1_{\loc}(\R^n)\cap \mathcal{S}'_s(\R^n),
\end{align*}
where $\mathcal{S}'_s(\R^n)$ stands for the dual of $\mathcal{S}_s(\R^n)$.
We remark that necessary a function in $\mathcal{L}^1_s(\R^n)$ needs to keep an algebraic growth of power strictly smaller than $2s$, in order to make the above expression meaningful, as was pointed out in \cite{banuelos,MR1671973,silvestrethesis} and recently in \cite{cones}.\\
In order to study the local behaviour of $u$, let us look at the extension technique popularized by Caffarelli and Silvestre (see \cite{CS2007}), characterizing the fractional Laplacian in $\mathbb{R}^n$ as the Dirichlet-to-Neumann map for a variable $v$ depending on one more space dimension. Namely for every $u\in H^s(\R^n)$, let us consider $v \in H^{1,a}(\R^{n+1}_+)$ satisfying
\begin{equation}\label{eq:Pextended}
\begin{cases}
\mbox{div}(y^{a}\nabla v)=0 &\mbox{in }\mathbb{R}^{n+1}_+,\\
v(x,0)=u(x) & \mbox{in }\Sigma\;.
\end{cases}
\end{equation}
with $a=1-2s \in (-1,1)$.
Such an extension exists unique and is given by the formula
\begin{equation*}
v(x,y)=\gamma(n,s)\int_{\R^n} \frac{y^{2s}u(x)}{(|x-\eta|^2+y^2)^{n/2+s}} \mathrm{d}\eta \qquad \mathrm{where }\; \gamma(n,s)^{-1}=:\int_{\R^n} \frac{1}{(|\eta|^2+1)^{n/2+s}} \mathrm{d}\eta\, ,
\end{equation*}
where the nonlocal operator $(-\Delta)^s$ translates into the Dirichlet-to-Neumann opeartor type
$$
(-\Delta)^s\colon H^s(\R^n)\to H^{-s}(\R^n),\quad u \longmapsto -\frac{C(n,s)}{\gamma(n,s)}\lim_{y \to 0^+}y^{1-2s}\partial_y v(x,y),
$$
with $C(n,s)$ the normalization constant deeply studied in \cite{cones}.
By \cite{Nekvinda}, it is known that the space $H^s(\R^n)$ coincides with the trace on $\partial \R^{n+1}$ of the weighted Sobolev space $H^{1,a}(\R^{n+1}_+)$ and in general
$$
[u]^2_{H^s(\R^n)} = {\frac{C(n,s)}{\gamma(n,s)}}\int_{\R^{n+1}}{\abs{y}^{1-2s} \abs{\nabla v}^2\mathrm{d}X},
$$
where $v$ is the $L_a$-harmonic extension of $u$ defined by \eqref{eq:Pextended}.
Since in the context of the extension problem the equation \eqref{Ps} translates in the homogeneous Neumann condition
$$
\partial^{a}_y v(x,0) =-\frac{C(n,s)}{\gamma(n,s)}\lim_{y \to 0^+} y^{1-2s}\partial_y v(x,y) = 0 \quad\mbox{on }B_1\subset \Sigma.
$$
Applying an even reflection through $\Sigma$, we can study the structure of the nodal set of $s$-harmonic function in $\R^{n}$ as the restriction of the nodal set $\Gamma(v)$ on the characteristic manifold $\Sigma$ of the solution
\begin{equation}\label{problem}
\begin{cases}
L_a v =0 & \mbox{in }B_1^+\\
  v(x,-y)=v(x,y)& \mbox{in }B_1^+\\
  v(x,0)=u(x) & \mbox{in }B_1
  \end{cases}
\end{equation}
where $a=1-2s \in (-1,1)$ and $B_1^+$ is the unitary $(n+1)$-dimensional ball in $\R^{n+1}$. \\ Moreover, by \cite{Nekvinda} is it well known that the class of trace on $B_1^+\cap\Sigma=B_1$ of function $L_a$-harmonic in $B_1^+$ is equal to the space $H^{s}(B_1)$.\\
Through this Section we will always identify as $v$ the $L_a$-harmonic extension of $u$ in $\R^{n+1}$ symmetric with respect to $\Sigma$ and with $B_r(x_0)^+$ the ball in $\R^{n+1}_+$ of radius $r>0$ and centered in the point $X_0=(x_0,0)\in \Sigma$ in the characteristic manifold associated with $\Sigma$.\\
The following results are a direct consequence of the ones obtained for purely symmetric $L_a$-harmonic function. For this reason the proof of the majority of them is skipped when the result is obtained just passing through the $L_a$-harmonic extension.
\begin{proposition}\label{doubling.s.prop}
  Given $s\in (0,1)$, let $u$ be $s$-harmonic in $B_1$. Then, there for every $x_0 \in B_1$,
  \begin{equation}\label{doubling.Sigma}
  \frac{1}{R^n}\int_{B_R(x_0)}{u^2\mathrm{d}x}\leq C(n,s)  \left(\frac{R}{r}\right)^{2N-1}\frac{1}{r^n}\int_{B_r(x_0)}{u^2\mathrm{d}x},
  \end{equation}
  for $0<r<R<1-\abs{x_0}$ and $N=N(X_0,v,1-\abs{X_0})$, with $v$ the $L_a$-harmonic extension of $u$.
  \end{proposition}
  \begin{proof}
  Let $v \in H^{1,a}(B_1)$ be the $L_a$-harmonic extension of $u$ in $\R^{n+1}$, symmetric with respect to $\Sigma$. The idea of this proof is to ``move'' the doubling condition on $\R^{n+1}$ to the characteristic manifold $\Sigma$.
  In \cite{Ruland} the author used a similar strategy to prove a so called ``bulk doubling property''.\\ In our case we improve the proof  by using our blow up analysis developed in Section \ref{section.blowup} and applying the correct factor of scaling in order to pass from a doubling condition in the dimension $n+a+1$ to the one on $\Sigma$.\\
  Let $X_0 \in B_1\cap \Sigma$, and $v$ the $L_a$-harmonic extension symmetric with respect to $\Sigma$. Integrating the inequality in Corollary \ref{doubling.corollary.Sigma}, see \eqref{doublin.ball}, we obtain that
  $$
  \int_{B_{R}^+(X_0)}{\abs{y}^a v^2\mathrm{d}X}\leq   \left(\frac{R}{r}\right)^{2C+n+a}\int_{B_{r}^+(X_0)}{\abs{y}^a v^2\mathrm{d}X}
  $$
  for every $0<r<R< 1-\abs{X_0}$, with $N=N(X_0,v,1-\abs{X_0})$. By the interpolation estimate in \cite{Ruland}, we obtain
  \begin{align*}
  \frac{1}{R^{n}}\int_{B_R(X_0)}{u^2\mathrm{d}x} & \leq C(n,a)\left(\frac{1}{R^{n+a+1}}\int_{B_R^+(X_0)}{\abs{y}^a v^2\mathrm{d}X}+\frac{1}{R^{n+a-1}}\int_{B_R^+(X_0)}{\abs{y}^a \abs{\nabla v}^2\mathrm{d}X}\right)\\
  & \leq C(n,a)\left(\frac{1}{R^{n+a+1}}\int_{B_R^+(X_0)}{\abs{y}^a v^2\mathrm{d}X}+\frac{1}{R^{n+a+1}}\int_{B_{2R}^+(X_0)}{\abs{y}^a v^2\mathrm{d}X}\right)\\
  & \leq C(n,a)2^{2N+n+a+1}
  \frac{1}{R^{n+a}}\int_{B_R^+(X_0)}{\abs{y}^a v^2\mathrm{d}X}
  \end{align*}
  where in the second inequality we used the Caccioppoli estimate \eqref{caccio} and in the last one the doubling condition. Since it yields the desired lower bound for the left hand side of the doubling condition on $\Sigma$, we left to prove the upper bound. Let us prove by contradiction the existence of $C>0$ and a radius $0<\overline{r}<R$ such that
  \begin{equation}\label{ruland}
  \int_{\partial B_r^+(X_0)}{\abs{y}^a v^2\mathrm{d}X} \leq C(n,a) r^{a+1} \int_{\partial B_r(X_0)}{u^2\mathrm{d}{x}}\quad\mbox{for all } 0<r\leq \overline{r},
  \end{equation}
  which will finally implies \eqref{doubling.Sigma} after a simple integration.\\
  Hence, suppose there exists a sequence $r_k\searrow 0^+$ such that
  \begin{equation}\label{absurd.ruland}
  \norm{v}{L^{2,a}(\partial B_{r_k}^+(X_0))}\geq k  r_k^{\beta/2}\norm{u}{L^2(\partial B_{r_k}(X_0))},
  \end{equation}
  with $\beta=a+1$. Then let us consider the blow-up sequence of $u$ centered at $X_0$ associated with $(r_k)_k$
  $$
  v_k(X) = \frac{v(X_0+r_k X)}{\rho_k}\quad \mbox{with }\,\rho_k^2 =  \frac{1}{r_k^{n+a}}\int_{\partial B_{r_k}^+(X_0)}{\abs{y}^a v^2\mathrm{d}X} = H(X_0,v,r_k).
  $$
  By definition we have $\norm{v_k}{L^{2,a}(\partial B_1^+)}=1$, and by Lemma \ref{bounded.H1} the sequence $(v_k)_k$ is uniformly bounded in $H^{1,a}(B_R^+)$ and $L^\infty(\overline{B_R})$, for every $R>0$. In particular, by \eqref{absurd.ruland}, we obtain
  $$
    \norm{u_k}{L^{2}(\partial B_1 )}
=  \norm{v_k}{L^{2}(\partial B_1 )} = \ddfrac{r_k^{-\frac{n-1}{2}}\norm{v}{L^2(\partial B_{r_k}(X_0))}}{r_k^{-\frac{n+a}{2}}\norm{v}{L^{2,a}(\partial B_{r_k}^+(X_0))}} \leq k^{-1} r_k^{\frac{a+1-\beta}{2}}=k^{-1}.
  $$
  Thus, up to a subsequence, by Theorem \ref{blowup.convergence} the blow-up sequence $(v_k)_k$ strongly converge in $H^{1,a}_\loc(\R^{n+1})$ and in $C^{0,\alpha}_\loc(\R^n)$, for every $\alpha \in (0,1)$ to some homogeneous blow-up limit $\overline{v}\in H^{1,a}_\loc(\R^{n+1})$ such that $\overline{v}=0$ on $B_1$, $\norm{\overline{v}}{L^{2,a}(\partial B_1^+)}=1$ and it satisfies
  $$
  \begin{cases}
    L_a \overline{v}=0 & \mbox{in } \R^{n+1} \\
    \partial^a_y \overline{v}=0 & \mbox{in }\Sigma. \end{cases}
  $$
  Hence, by Proposition \ref{empty.interior} we obtain that $\overline{v}\equiv 0$ in contradiction with $\norm{\overline{v}}{L^{2,a}(B_1^+)}=1$.
\end{proof}

  In order to justify the analysis of the local behaviour of $s$-harmonic functions, it is necessary to ensure the validity of the strong unique continuation property. 
It is known by \cite{FallFelli} that an $s$-harmonic function in $B_1$ enjoys the \emph{strong unique continuation property} in $B_1$, i.e. the only solutions which vanishes of infinite order at a point $X_0 \in \Gamma(u)$ is $u \equiv 0$. Similarly, an $s$-harmonic function in $B_1$ is said to satisfies the \emph{unique continuation
property} in $B_1$ if the only solution of $(-\Delta)^s u = 0$ in $H^{s}_\loc(B_1)$ which can vanish in an open
subset of $B_1$ is $u \equiv 0$. Indeed, as a direct consequence of Proposition \ref{empty.interior} we prove
\begin{corollary}\label{unique.continuation.s}
  Let $s\in (0,1)$ and $u$ be $s$-harmonic in $B_1$. Then the nodal set $\Gamma(u)$ has either empty interior in $B_1$ or $u\equiv 0$.
\end{corollary}
Hence, it is reasonable to define the notion of vanishing order of $u$ at $x_0 \in \Gamma(u)$. More precisely, the strong unique continuation property guarantees the existence of $k\in \R$ such that
$$
\limsup_{r \to 0^+}{\frac{1}{r^{n+2k}}\int_{B_r(x_0)}{u^2\mathrm{d}x}}>0.
$$
In order to correlate the notion of vanishing order of $s$-harmonic functions with the one for their $L_a$-harmonic extension, let us introduce the following common definition.
\begin{definition}
  Given $s\in (0,1)$, let $u$ be an $s$-harmonic function in $B_1$ and $x_0 \in \Gamma(u)$. The vanishing order of $u$ in $x_0$
is defined as the number $\mathcal{O}(u, x_0)\in \R$ such that
$$
\limsup_{r\to 0^+}\frac{1}{r^{n-1+2k}}\int_{\partial B_r(x_0)}{u^2\mathrm{d}x}=
\begin{cases}
  0 & \mbox{if } k<\mathcal{O}(u,x_0) \\
  +\infty & \mbox{if }k>\mathcal{O}(u,x_0).
\end{cases}
$$
\end{definition}
In particular, from Lemma \ref{gap.even} and Proposition \ref{doubling.s.prop} we obtain
\begin{corollary}
  Let $s\in (0,1)$ and $a=1-2s\in (-1,1)$. Given $u$ an $s$-harmonic function in $B_1$, then the vanishing order $\mathcal{O}(u,x_0)$ of $u$ in $x_0 \in \Gamma(u)$ satisfy
  $$
  \mathcal{O}(u,x_0) = N(X_0,v,0^+)= \lim_{r\to 0^+} \ddfrac{r\int_{B_r^+(X_0)}{\abs{y}^a \abs{\nabla v}^2 \mathrm{d}X} }{\int_{\partial B_r^+(X_0)}{\abs{y}^a v^2 \mathrm{d}\sigma}},
  $$
  where $v$ is the unique $L_a$-harmonic extension of $u$ symmetric with respect to $\Sigma$ and $X_0=(x_0,0)$.
\end{corollary}
Hence, for $k \in 1+\N$, we define the subsets
$$
\Gamma_k(u)  \coloneqq  \{x_0 \in \Gamma(u) \colon \mathcal{O}(u,x_0)=k\},
$$
which is coherent with the Definition for the $L_a$-harmonic case.
Indeed, inspired by the results in Section \ref{section.blowup}, we can prove a convergence result for the blow-up sequence associated with $x_0 \in \Gamma(u)$ to some blow-up limit $\varphi \in H^{s}_{\loc}(\R^n)$. \\Before proving the main convergence result, let us introduce  two different classes of tangent maps strictly related to the ones introduced in Definition \ref{blowup.class} and Definition \ref{blowup.class.y}. In particular, we will see that the structure of the nodal set is completely defined starting from these blow-up classes.
\begin{definition}
  Given $ s \in (0,1)$ and $k\in 1+\N$, we define the set of all possible blow-up limit of order $k$, i.e. the set of the traces of all $L_a$-harmonic polynomial of degree $k$ symmetric with respect to $\Sigma$, as
  $$
  \mathfrak{B}_k^s(\R^{n})
  = \left\{\varphi \in H^{s}_{\loc}(\R^{n}):
  \mbox{ the } L_a\mbox{-extension of }\varphi \in \mathfrak{sB}_k^a(\R^{n+1})
  \right\}.
  $$
  \end{definition}
  Moreover, by Lemma \ref{yornot} and Lemma \ref{garofalo}, the space $\mathfrak{B}_k^s(\R^{n})$ is the set of all possible homogenous polynomial of order $k$ in $\R^n$, which is, by the results in \cite{Nekvinda}, the space of traces on $\Sigma$ of $\mathfrak{sB}_k^a(\R^{n})$. Similarly, if we define with $\mathfrak{B}_k^*(\R^{n})$ the set of function $\varphi\in\mathfrak{B}_k^s(\R^{n})$ such that $\Delta \varphi=0$ in $\R^n$, namely the collection of homogeneous harmonic polynomial of order $k$, there holds that $\mathfrak{B}^*_k(\R^n)$ coincides with the set of traces of blow-up limits in $\mathfrak{sB}_a^*(\R^{n+1})$. \\
  The following result is a direct application of Theorem \ref{blowup.convergence}, Lemma \ref{nondegeneracy}, Theorem \ref{uniqueness} and Theorem \ref{continuation} on the $L_a$-harmonic extension of $u$ symmetric with respect to $\Sigma$ and it ensure the existence of a unique non trivial tangent map at every point of the nodal set of $u$.
\begin{proposition}\label{blowup.s}
  Given $s \in (0,1)$, let $u$ be an $s$-harmonic function in $B_1$ and $x_0\in \Gamma_k(u)$. Then there exists a unique $k$-homogenous polynomial $\varphi^{x_0} \in \mathfrak{B}_k^s(\R^n)$ such that
  \begin{equation*}
    u_{x_0,r}(x)= \ddfrac{u(x_0 + r_k x)}{r^k}\longrightarrow \varphi^{x_0}(x),
    \end{equation*}
    where the blow-up sequence $(u_{x_0,r})_r$ converges strongly in $H^{s}_{\loc}(\R^n)$ and in $C^{1,\alpha}_{\loc}(B_1)$, for every $\alpha \in (0,1)$. Moreover, the unique tangent map $\varphi^{x_0}$ is nontrivial and it satisfies the following generalized Taylor expansion
    $$
    u(x)= \varphi^{x_0}(x-x_0)+o(\abs{x-x_0}^k),
    $$
    where the map $x_0 \mapsto \varphi^{x_0}$ from $\Gamma_k(u)$ to the space $\mathfrak{B}_k^s(\R^n)$ is continuous.
\end{proposition}
Thus, let
\begin{align*}
  \mathcal{R}(u)&=  \{x_0 \in \Gamma(u)\colon \mathcal{O}(u,x_0)=1\}, \\
  \mathcal{S}(u)&= \bigcup_{k\geq 2}\Gamma_k(u) =\bigcup_{k\geq 2}\left\{x_0 \in \Gamma(u)\colon  \mathcal{O}(u,x_0)=k \right\},
  \end{align*}
be respectively the \emph{regular} and \emph{singular} part of $\Gamma(u)$. Moreover, by Corollary \ref{gamma_k} we can find a different characterization of the singular strata $\Gamma_k(u)$ for $k\geq 2$, i.e.
$$
\Gamma_k(u) = \left\{x_0 \in \Gamma_k(u)\left\lvert
\begin{array}{cl}
    D^{\nu}u(x_0)=0 &\mbox{ for every }\abs{\nu}\leq k-1 \vspace{0.1cm}\\
    D^{\nu_0}u(x_0)\neq 0 &\mbox{ for some }\abs{\nu_0}=k
  \end{array}
\right.\right\}.
$$
The following are the main theorems related to the regularity and the geometric structure of the nodal set: while in the first result we focus the attention on the regular part of the nodal set, proving a result similar to its local counterpart (see \cite{MR1305956,MR1090434}), in the ones related to the singular strata we highlight the presence of a singular subset $\mathcal{S}^s(u)$ strictly related to the nonlocal attitude of the fractional Laplacian.
\begin{theorem}
  Given $s\in (0,1)$, let $u$ be $s$-harmonic in $B_1$. Then the regular set $\mathcal{R}(u)$ is relatively open in $\Gamma(u)$ and is locally a smooth hypersurface on $\R^n$. Moreover
  $$
  \mathcal{R}(u) = \left\{x \in \Gamma(u) \colon \abs{\nabla u (x)}\neq 0 \right\}.
  $$
\end{theorem}
\begin{proof}
  By Corollary \ref{unique.continuation.s}, let us suppose that $u\not\equiv 0$ in $B_1$ and hence $\Gamma(u)$ has empty interior.
  Given $v$ the unique $L_a$-harmonic extension of $u$ symmetric with respect to $\Sigma$, well defined in $H^{1,a}(B_1^+)$ by \eqref{problem}, it is obvious to infer that
  $$
  \mathcal{R}(v)\cap \Sigma = \mathcal{R}(u).
  $$
  Moreover, by Lemma \ref{Fsigma} we already know that $\mathcal{R}(u)$ is relatively open in $\Gamma(u)$ and by the application of the Federer reduction principle in Theorem \ref{federer.theorem} we obtain
  $$
\mathrm{dim}_{\mathcal{H}}(\mathcal{R}(u)) = n-1.
$$
Now, by Theorem \ref{blowup.s} and our blow-up classification, for every $x_0 \in \mathcal{R}(u)\cap \Sigma$ there exists a linear map $\varphi^{x_0} \in \mathfrak{B}^s_1(\R^{n})$ such that
$$
u(x) = \varphi^{x_0}(x-x_0) + o(\abs{x-x_0}) = \varphi^{x_0}(\nu^{x_0})\langle x-x_0, \nu^{x_0} \rangle + o(\abs{x-x_0})
$$
for some $\nu^{x_0} \in S^{n-1}$.\\ Moreover, still by Theorem \ref{blowup.s} we know that the map $x_0 \mapsto \varphi^{x_0}(\nu^{x_0})\nu^{x_0}$ is continuous.
By Proposition \ref{smooth}, since $u \in C^\infty (B_{1/2})$ we can use the tangent map in order to compute the directional derivative of $u$, which will ensures the nondegeneracy of the gradient of $u$ at $x_0$. More precisely, for every $\xi \in S^{n-1}$
  $$
  \langle \nabla u(x_0),\xi\rangle = \left.\frac{d}{dt}u(x_0 +t \xi) \right\lvert_{t=0}
  = \lim_{t \to 0} \frac{u(x_0+ t\xi)}{t}
  = \varphi^{x_0}(\nu^{x_0}) \langle \xi , \nu^{x_0}\rangle,
  $$
  and hence $\nabla u(x_0) = \varphi^{x_0}(\nu^{x_0})\nu^{x_0}$ which is nonzero by the nondegeneracy of the tangent map. Finally, by the implicit function theorem we obtain the claimed result.
\end{proof}
As in Section \ref{singular.sec}, initially we will prove a stratification result for the singular set $\mathcal{S}(u)$.\\
The main idea of this stratification is to stratify the nodal set by the spines of the normalized tangent maps. Indeed, we will introduce the subset $\Gamma^j_k(u)$ as the set of points at which every tangent map has at most $j$ independent directions of translation invariance in order
to correlate the nodal set of $u$ with the dimension of the set where the tangent map $\varphi^{X_0}$ vanishes with the same order of $u$.\\ We remark that these result are a direct consequence of Theorem \ref{regularity.singular} and Theorem \ref{general.singular}, nevertheless, for the sake of completeness, we present some technical details.\\
From Definition \ref{d.sigma}, given $s\in (0,1)$  we call $d^{x_0}$ the dimension of $\Gamma_k(u)$ at $x_0\in \Gamma_k(u)$ as
$$
  d^{x_0}
  =\mbox{dim}\left\{ \xi \in \R^n \colon \langle \xi, \nabla\varphi^{x_0}(x)\rangle = 0 \mbox{ for all } x \in \R^n\right\}.
$$
Now, fixed $k\geq 2$, for each $j=0,\dots,n-1$ let us define
$$
\Gamma^j_k(u) = \left\{ x_0 \in \Gamma_k(u) \colon \mbox{dim}\,\Gamma_k(\varphi^{x_0}) = j \right\},
$$
where $\varphi^{x_0}$ is the unique tangent limit of $u$ at $x_0$. As we already mentioned, since for $k\geq 2$ we have $\mathfrak{B}^s_k(\R^{n})\setminus\mathfrak{B}^*_k(\R^{n}) \neq \emptyset$, we decide to introduce the following singular sets
$$
\mathcal{S}^*(u) = \bigcup_{k\geq 2}\Gamma^*_k(u)\quad \mbox{and}\quad \mathcal{S}^s(u) = \bigcup_{k\geq 2}\Gamma_k^s(u),
$$
where
$$
  \Gamma^*_k(u)=\left\{x_0 \in \Gamma_k(u)\colon \varphi^{x_0} \in \mathfrak{B}^*_k(\R^{n})\right\}\ \mbox{ and } \ \Gamma^s_k(u)= \Gamma_k(u)\setminus \Gamma^*_k(u).
$$
The idea is to stratify the singular set taking care of both the dimension $d^{x_0}$ and the different classes of tangent map associated with the sets $\Gamma^*_k(u)$ and $\Gamma^s_k(u)$.
\begin{theorem}
Given $s\in (0,1)$ let $u$ be $s$-harmonic in $B_1$. Then there holds
  $$
  \mathcal{S}(u)= \mathcal{S}^*(u) \cup \mathcal{S}^s(u)
  $$
  where $\mathcal{S}^*(u)$ is contained in a countable union of $(n-2)$-dimensional $C^1$ manifolds and $\mathcal{S}^s(u)$ is contained in a countable union of $(n-1)$-dimensional $C^1$ manifolds. Moreover
  $$
   \mathcal{S}^*(u)=\bigcup_{j=0}^{n-2} \mathcal{S}^*_j(u)\quad\mbox{and}\quad
   \mathcal{S}^s(u)=\bigcup_{j=0}^{n-1}\mathcal{S}^s_j(u),
  $$
  where both $\mathcal{S}^*_j(u)$ and $\mathcal{S}^s_j(u)$ are contained in a countable union of $j$-dimensional $C^1$ manifolds.
\end{theorem}
\begin{proof}
  The proof is based on a combination of Theorem \ref{regularity.singular} Theorem \ref{general.singular}.
  Since for every $k\geq 2$ the functions $\varphi \in \mathfrak{sB}^*_k(\R^{n+1})$ are homogeneous polynomial harmonic in $\Sigma$, we have that $\mbox{dim}\left(\mathcal{S}(\varphi)\cap \Sigma\right)\leq n-2$, and consequently $d^{X_0}_\Sigma \leq n-2$ for every $X_0 \in \Gamma_k^*(u)$. \\Similarly, following Proposition \ref{example} and the remarks in the proof of Theorem \ref{federer.theorem}, since for every $k\geq 2$ there exists $\varphi \in \mathfrak{sB}^a_k(\R^{n+1})\setminus\mathfrak{sB}^*_k(\R^{n+1})$ such that $\mbox{dim}\left(\mathcal{S}(u)\cap \Sigma\right)=n-1$ we obtain that for $X_0 \in \Gamma^a_k(u)$ there holds $d^{X_0}_\Sigma \leq n-1$.\\
Now, by applying the same argument in the proof of Theorem \ref{general.singular}, if we set
 \begin{align*}
\mathcal{S}^*_j(u) &= \bigcup_{k\geq 2}\left\{x\in \Gamma^*_k(u)\colon d^{x_0} =j\right\} \quad\mbox{for }j=0,\cdots,n-2\\
\mathcal{S}^s_j(u) &= \bigcup_{k\geq 2}\left\{x\in \Gamma^s_k(u)\colon d^{x_0} =j\right\}\quad\mbox{for }j=0,\cdots,n-1,
\end{align*}
we obtain the claimed result.
\end{proof}
Furthermore, by Proposition \ref{example} we obtain that for any $x_0 \in \mathcal{S}^s_{n-1}(u)$  the leading polynomial of $u$ at $x_0$ is a monomial of degree $k$ with $k\in 2+\N$ depending only on one variable of $\R^{n}$.\\
In order to show the optimality of the result, we will now present an explicit example of $s$-harmonic function in $B_1=(-1,1)\subset \R$ with vanishing order $k\geq 2$. More precisely, the following construction allows to exhibit an $s$-harmonic in $B_1\subset \R^n$ with $\Gamma(u) = \mathcal{S}^s_{n-1}(u)$.\\
Fixed $s\in (0,1)$, let $B_1=(-1,1)\subset \R$ be the unitary ball in the real line and $f\in \mathcal{L}^1_s(\R)\cap C(\R)$ an admissible function. By the classical potential theory is it known that the unique solution of
$$
\begin{cases}
(-\Delta)^s u = 0 & \mbox{in } B_1\\
u= f & \mbox{in } \R \setminus B_1
\end{cases}
$$
an be computed explicitly as
$$
u(x)= \int_{\R \setminus B_1}{P(x,y)f(y)\mathrm{d}y}=\frac{\Gamma(1/2)\sin{\pi s}}{\pi^{3/2}}\left(1-\abs{x}^2\right)^s\int_{\R\setminus B_1}{\frac{1}{(\abs{y}^2-1)^s}\frac{f(y)}{\abs{x-y}}\mathrm{d}y}.
$$
We remark that several results and reference about the Poisson kernel can be found in the classical book of Landkof \cite{Landkof}.\\
Now, given $f\in \mathcal{L}^1_s(\R)\cap C(\R)$, let us consider $f_e,f_o\in \mathcal{L}^1_s(\R)\cap C(\R)$ respectively the even and odd part of $f$ uniquely defined as
$$
f_e(x)=\frac{f(x)+f(-x)}{2}\quad\mbox{and}\quad f_o(x)=\frac{f(x)-f(-x)}{2}.
$$
Under this notations, we obtain for $x\in (-1,1)$
$$
u(x)= \frac{2\Gamma(1/2)\sin{\pi s}}{\pi^{3/2}}\left(1-\abs{x}^2\right)^s\left[\int_1^{+\infty}{\frac{f_e(y) y}{(\abs{y}^2-1)^s (y^2-x^2)}\mathrm{d}y}+ x \int_1^{+\infty}{\frac{f_o(y) }{(\abs{y}^2-1)^s (y^2-x^2)}\mathrm{d}y}\right].
$$
Since for every $y \in \R\setminus B_1$ we have $\abs{y}>\abs{x}$, using the series expression
$$
\frac{1}{y^2-x^2} = \frac{1}{y^2}\sum_{n=0}^\infty \frac{x^{2n}}{y^{2n}},
$$
we obtain
$$
u(x)= \frac{2\Gamma(1/2)\sin{\pi s}}{\pi^{3/2}}\left(1-\abs{x}^2\right)^s
\left[\sum_{n=0}^\infty A_{2n}(f) x^{2n} +\sum_{n=0}^\infty A_{2n+1}(f) x^{2n+1}\right],
$$
where for every $n\in \N$
$$
A_{2n}(f)= \int_1^{+\infty}{\frac{f_e(y)}{y(\abs{y}^2-1)^s y^{2n}}\mathrm{d}y}\quad \mbox{and}\quad A_{2n+1}(f)= \int_1^{+\infty}{\frac{f_o(y)}{y(\abs{y}^2-1)^s y^{2n+1}}\mathrm{d}y}.
$$
In particular, if we consider $f(x)=(\abs{x}^2-1)^s g(x^{-1})$ we obtain by a simple change of variables
$$
A_{2n}(f)=
\int_0^{1}{\frac{g_e(y)}{y} y^{2n}\mathrm{d}y}\quad \mbox{and}\quad A_{2n+1}(f)= \int_0^{1}{g_o(y) y^{2n}\mathrm{d}y}.
$$
Hence, for every fixed order of vanishing $k \in 2+ \N$, there exists a polynomial function $g(x)$ such that $A_i(f)=0$, for every $i\leq k-1$. We remark that all these coefficients can be computed explicitly. Moreover, this construction implies that for every vanishing order $k\in 2+\N$ there exists an $s$-harmonic function in $(-1,1)$ which vanishes at zero with order $k$, which shows the purely nonlocal behaviour of the singular set of $s$-harmonic functions.

\section{Some more general nonlocal equations}

In this part we will generalize the previous result to a more general class of fractional power of divergence form operator following the change of variables first introduced in \cite{MR0140031} and deeply popularized in the works \cite{MR833393,MR882069} .
Inspired by works, we consider solutions of homogeneous linear elliptic differential equations of the second order with Lipschitz leading coefficients and no lower order terms. We remark that in general the regularity assumption on the coefficient is optimal thanks to the counterexample of \cite{miller}.\\
Let $A(x)=(a_{ij}(x))$ be a symmetric $n\times n$ matrix-valued function in $\overline{B}_1$ satisfying the following assumptions:
\begin{enumerate}
  \item there exists $\lambda \in (0,1)$ such that
  $$
  \lambda \abs{\xi}^2 \leq a_{ij}(x) \xi_i \xi_j \leq \lambda^{-1}\abs{\xi}^2 \quad\mbox{for any }x \in \overline{B}_1 \mbox{ and } \xi \in \R^{n};
  $$
  \item there exists $\gamma>0$ such that for any $1\leq i,j\leq n$
  $$
  \abs{a_{ij}(x)-a_{ij}(z)}\leq \gamma\abs{x-z}\quad \mbox{for any }x,z \in B_1.
  $$
\end{enumerate}
Hence, consider now the uniformly elliptic operator
\begin{equation}\label{Ldefi}
L u = \mbox{div}\left(A(x)\nabla u(x)\right) = \frac{\partial}{\partial x_i}\left( a_{ij}(x) \frac{\partial }{\partial x_j}u\right) =0 \quad \mbox{in } B_1.
\end{equation}
By \cite{stingatorrea}, we already know the existence of characterization for the fractional powers of second order partial differential operators in some suitable class.
\begin{proposition}
  Let $s\in (0,1)$ and $u \in \mathcal{L}^1_s(\R^n)$. Given $a=1-2s \in (-1,1)$, a solution of the extension problem
  \begin{equation}\label{L.extendend}
  \begin{cases}
  L v + \frac{a}{y}\partial_{y} v + \partial^2_{yy}v=0 & \emph{in }\R^{n+1}_+\\
  v(x,0)=u(x) & \emph{in }\R^{n};
  \end{cases}
  \end{equation}
  is given by
  $$
  v(x,y)=\frac{1}{\Gamma(s)}\int_{0}^\infty{(e^{tL}(-L)^s f)(x)e^{-\frac{y^2}{4t}}\frac{\mathrm{d}t}{t^{1-s}}},
  $$
  and
  $$
  (-L)^s u(x) = -\frac{\Gamma(s)}{2^{1-2s} \Gamma(1-s)}\lim_{y \to 0^+}{y^{a}\partial_y u(x,y)}.
  $$
\end{proposition}
A similar extension can be constructed in the context of fractional powers $(-\Delta_M)^s$ of the Laplace-Beltrami operator on a Riemannian manifold $M$ and to conformal fractional Laplacian on conformally compact Einstein manifolds and asymptotically hyperbolic manifold, thanks to the extension technique developed in \cite{conformal} and the asymptotic expansion of their geodesic boundary defining function.\\
In this Section, we just consider the case of divergence form operator $L$ in order to show how to deal with the limit case of Lipschitz coefficients. Therefore, this analysis will extend the results also to the case of Laplace-Beltrami with Lipschitz metric. \\
As we did for the fractional Laplacian, in order to study the local behaviour of solution of fractional elliptic equation associated with operator $L$ in divergence form, let $s \in (0,1)$ and $u$ be a solution of the extended problem \eqref{L.extendend} associated with $L$, even with respect to the $y$-direction, i.e. such that
\begin{equation}\label{L.ext}
\begin{cases}
  \mbox{div}_{x,y}(\abs{y}^a \overline{A}(x)\nabla_{x,y} u) = 0, & \mathrm{in}\, \R^{n+1} \\
  u(x,y)=u(x,-y), & \mathrm{in}\, \R^{n+1}.
\end{cases}
\end{equation}
where $\overline{A}(x)$ is a symmetric $(n+1) \times (n+1)$ matrix-valued function in $B_1$ such that
\begin{equation}\label{overline.A}
\overline{A}(x)=
\left(
\begin{array}{ccc|c}
&&&  \\
&\hspace{-0.1cm}A(x)\vspace{-0.1cm}\hspace{-0.1cm}&& 0\\
&&& \\
\hline
  & 0 & & 1
\end{array}
\right).
\end{equation}
Inspired by Definition \ref{energy} we define the natural generalization of notion of $L_a$-harmonicity in the context of divergence form operator $L$ with Lipschitz leading coefficient.
\begin{definition}
Let $a \in (-1,1)$, we say $u\in H^{1,a}(B_1)$ is $L_a^A$-harmonic in $B_1$ if for every $\varphi \in C^\infty_c(B_1)$ we have
  $$
  \int_{B_1}{\abs{y}^a \langle \overline{A}(x)\nabla u , \nabla \varphi\rangle \mathrm{d}X} = 0,
  $$
where $\overline{A}(x)$ is the symmetric $(n+1)\times (n+1)$ matrix-valued function defined in \eqref{overline.A}.
\end{definition}
Through this Section we will state all the result in the contxt of $L_a^A$-harmonic function in $B_1$ symmetric with respect to $\Sigma$, since the nodal set of the fractional powers $(-L)^s$ is completely defined as the restriction of the nodal set of $L_a^A$-harmonic function symmetric with respect to $\Sigma$, as we did in the previous part of the Section.\\ Obviously, in order to better understand the behaviour of general degenerate operator with Lipschitz leading coefficient, one could consider general $L_a^A$-harmonic solution and apply the ideas and the decomposition of the previous Sections.\\
In order to develop a blow-up analysis, we need to extend our monotonicity formula based on a geometrical reduction introduced in \cite{MR0140031} and deeply used in the local case \cite{MR833393,MR882069}. Hence, for $n\geq 3$, define a Lipschitz metric $\overline{g}=\overline{g}_{ij}(x,y)\mathrm{d}x_i \otimes\mathrm{d}x_j + \overline{g}_{yy}(x,y) \mathrm{d}y \otimes \mathrm{d}y$ on $B_1$ by setting
\begin{equation}\label{over.g}
\overline{g}_{ij} = \overline{a}^{ij}\left(\det \overline{A} \right)^{\frac{1}{n-1}}=
\begin{cases}
a^{ij}\abs{A}^{\frac{1}{n-1}}, & \mbox{if } 1\leq i,j\leq n\\
\abs{A}^{\frac{1}{n-1}}, & \mbox{otherwise}
\end{cases}
\end{equation}
where $\overline{a}^{i,j}$ and $a^{i,j}$ denote respectively the entries of $\overline{A}^{-1}$ and ${A}^{-1}$. Letting similarly $\overline{g}^{ij}$ be the entries of the inverse metric of $\overline{g}$, consider
\begin{align*}
r(x,y)^2 &= \overline{g}_{ij}(0)x_i x_j + \overline{g}_{yy}(0)y^2\\
&= \abs{A}^{\frac{1}{n-1}}\left(a^{ij}(0)x_i x_j + y^2\right)
\end{align*}
and
\begin{align*}
\eta(x,y) &= \frac{1}{r^2(x,y)}\left(\overline{g}^{kl}(x)\overline{g}_{ik}(0)\overline{g}_{jl}(0)x_i x_j + \overline{g}^{yy}(x)\overline{g}_{yy}(0)\overline{g}_{yy}(0) y^2\right)\\
& =\frac{a_{kl}(x)a^{ik}(0)a^{jl}(0)x_i x_j + y^2}{a^{ij}(0)x_i x_j + y^2}.
\end{align*}
We can easily verify that $\eta$ is a positive Lipschitz function in $B_1$, whose Lipschitz constant depends on $n,\lambda,\Gamma$ but not on $a \in (-1,1)$.\\ Next, we introduce a new metric tensor $g=g_{ij}(x,y)\mathrm{d}x_i \otimes\mathrm{d}x_j + g_{yy}(x,y) \mathrm{d}y \otimes\mathrm{d}y$  in $B_1$ by defining  $g = \eta(x,y)\overline{g}$. In the intrinsic geodesic polar coordinates with pole at zero of the Riemannian manifold $(B_1, g_{ij})$, the metric tensor takes the form
$$
g= \mathrm{d}r \otimes \mathrm{d}r + r^2 b_{ij}(r,\theta) \mathrm{d}\theta_i \otimes \mathrm{d}\theta_{j},
$$
where
\begin{equation}\label{condition1}
b_{ij}(0,0)=\delta_{ij},\quad \abs{\partial_r b_{ij}(r,\theta)}\leq \Lambda(n,\lambda,\Gamma),\,\,\mbox{for }1\leq i,j\leq n.
\end{equation}
Moreover, if we denote $\abs{g}=\abs{\det g}$ we obtain
\begin{equation}\label{det.metric}
\sqrt{\abs{g}}=\eta^{\frac{n+1}{2}}\sqrt{\abs{\overline{g}}}= \eta^{\frac{n+1}{2}} \abs{A}^{\frac{1}{n-1}}.
\end{equation}
Here we denote by $\nabla_g u$ and $\mbox{div}_g X$ respectively the intrinsic gradient of a function $u$ and the intrinsic divergence of a vector field $X$ on $B_1$ in the metric $g$, i.e.
$$
\nabla_g u = g^{ij}\frac{\partial u }{\partial{x_i}}\frac{\partial}{\partial x_j}, \quad \mbox{div}_g X = \frac{1}{\sqrt{\abs{g}}}\left(\sum_{i=1}^n\frac{\partial  }{\partial{x_i}}(\sqrt{\abs{g}}X_i) + \frac{\partial}{\partial y}(\sqrt{\abs{g}}X_y)\right).
$$
Finally, in this new metric we rewrite the divergence form equation in \eqref{L.ext} as
$$
\mbox{div}_g\left(\abs{y}^a\mu \nabla_g u \right) = \frac{1}{\sqrt{\abs{g}}}\frac{\partial}{\partial y}\left[ \left(1-\sqrt{\abs{g}}g^{yy}\mu\right) u \frac{\partial }{\partial y}\abs{y}^a\right]
$$
where $\mu=\mu(x,y)$ is a positive Lipschitz function given by
$$
\mu(x,y)=\eta(x,y)^{-\frac{n-1}{2}}
$$
bounded in $\overline{B}_1$ and such that, in polar coordinates, it satisfies
\begin{equation}\label{condition2}
\mu(0,0)=1, \quad \abs{\frac{\partial}{\partial r} \mu(r,\theta)}\leq \Lambda(n,\lambda,\Gamma).
\end{equation}
By \eqref{over.g}, \eqref{det.metric} and the definition of $\mu$, for every $(x,y)\in B_1$
$$
\sqrt{\abs{g}}g^{yy}\mu= \eta^{\frac{n+1}{2}}\abs{A}^{\frac{1}{n-1}}\abs{A}^{-\frac{1}{n-1}}\eta^{-1}\eta^{-\frac{n-1}{2}}=1.
$$
To proceed, given $u\in H^{1,a}(B_1, \mathrm{d}V_{g})$ a solution of
\begin{equation}\label{div-mu}
\mbox{div}_g\left(\abs{y}^a\mu \nabla_g u \right)=0\quad \mbox{in } B_1
\end{equation}
symmetric with respect to $\Sigma$, let us define for any $r \in (0,1)$
\begin{align*}
  E_g(u,r) & =\frac{1}{r^{n+a-1}}\int_{B_g(r)}{\abs{y}^a \mu \abs{\nabla_g u}^2 \mathrm{d}V_g} \\
  H_g(u,r) & =\frac{1}{r^{n+a}}\int_{\partial B_g(r)}{\abs{y}^a \mu u^2 \mathrm{d}V_{\partial B_r}}
\end{align*}
where here $B_g(r)$ represents the geodesic ball in the metric $g$ of radius $r$ centered at the origin. We remark that by the polar decomposition of $g$, $B_g(r)$ coincides with the usual Euclidian ball.\\
In \cite{MR2334826} the authors introduced a new monotonicity formula tailored for a class of generalized Baouendi-Grushin operators.  As well known, such operators are strictly related to ours; thus, their result gives an analogue counterpart in the context of our weighted degenerate operator, firstly introduced in the pioneering papers \cite{fks,fkj}. More precisely, let us introduce the change of variable $\Phi \colon \R^{n+1}_+\to \R^{n+1}_+$ such that
$$
(x,z)=\Phi(x,y)= \left(x,\frac{y^{1-a}}{(1-a)^{1-a}}\right),
$$
with inverse $\Phi^{-1}(x,z)=\left(x,(1-a)z^{\frac{1}{1-a}}\right)$. Given a function $u(x,z)$ defined for $(x,z)\in \R^{n+1}_+$, we define a function $\tilde{u}(x,y)$ with $(x,y)\in \R^{n+1}_+$ as $\tilde{u}(x,y)=u(\Phi(x,y))$. A simple computations gives
$$
L \tilde{u}(x,y) + \partial_{yy}\tilde{u}(x,y) + \frac{a}{y}\partial_y\tilde{u}(x,y) =  z^{-\frac{2a}{1-a}}\left[\partial_{zz} u(x,z) + z^{\frac{2a}{1-a}}L u(x,z)\right].
$$
As we can see, the operator within square brackets in the right-hand side of the previous equation is a special case of the family of operators in $\R^{n}_x\times \R^1_z$ known as generalized Baouendi-Grushin operator.\\
Nevertheless, our problem does not satisfy the assumptions in \cite{MR2334826} and consequently we need to construct a new monotonicity formula, which does extend the class of generalized Baouendi-Grushin operator for which a unique continuation principle holds true.\\
Under the previous notations, for $r \in (0,1)$, we define the Almgren type monotonicity formula as
$$
N_g(u,r) = \frac{E_g(u,r)}{H_g(u,r)} = \ddfrac{r\int_{B_g(r)}{\abs{y}^a \mu \abs{\nabla_g u}^2 \mathrm{d}V_g} }{\int_{\partial B_g(r)}{\abs{y}^a \mu u^2 \mathrm{d}V_{\partial B_r}}}.
$$
\begin{theorem}
  Let $a\in(-1,1)$ and $u$ be a solution of \eqref{div-mu} symmetric with respect to $\Sigma$. Then there exist a constant $C>0$ such that the map $ r\mapsto e^{Cr} N_g(u,r)$ is absolutely continuous and monotone nondecreasing on $(0,1)$.
  Hence, there always exists finite the limit
$$
N_g(u,0^+) = \lim_{r\to 0^+}N_g(u,r),
$$
which we will refer to as the \emph{(Almgren) limiting frequency}.
\end{theorem}
\begin{proof}
  By assumption, both $r\mapsto E_g(u,r)$ and  $r\mapsto H_g(u,r)$ are locally absolutely continuous functions on $(0,1)$, that is both their derivative are $L^1_\loc(0,1)$. First, passing to the logarithmic derivatives, the monotonicity of $r\mapsto N_g(u,r)$ is a direct consequence of the claim
$$
\frac{d}{dr}\log{N(X_0,u,r)} = \frac{1}{r}+\ddfrac{\frac{d}{dr}\int_{B_g(r)}{\abs{y}^a \mu\abs{\nabla u}^2 \mathrm{d}V_{g}} }{\int_{B_g(r)}{\abs{y}^a \mu \abs{\nabla u}^2 \mathrm{d}V_{g}} } - \ddfrac{\frac{d}{dr}\int_{\partial B_g(r)}{\abs{y}^a \mu u^2 \mathrm{d}V_{\partial B_r}}}{\int_{\partial B_g(r)}{\abs{y}^a\mu u^2 \mathrm{d}V_{\partial B_r}}} \geq 0
$$
for $r \in (0,1)$. First, by setting $b(r,\theta)= \abs{\det b_{ij}(r,\theta)}$, we obtain $\sqrt{g(r,\theta)}=r^n\sqrt{b(r,\theta)}$ and we can rewrite the denominator $H_g(u,r)$ of the Almgren quotient as
$$
H_g(u,r)= \int_{\partial B_g(1)}{\abs{\theta_n}^a \mu(r,\theta) u^2(r,\theta)\sqrt{b(r,\theta)}\mathrm{d}\theta},
$$
where $\theta_n$ is the spherical coordinate associated with the $y$-direction. By differentiating respect to $r\in (0,1$, we obtain
$$
\frac{d}{dr}H_g(u,r)=\frac{2}{r^{n+a}}\int_{\partial B_g(r)}{\abs{y}^a\mu u \partial_\rho u \mathrm{d}V_{\partial B_r}} +\frac{1}{r^{n+a}} \int_{\partial B_g(r)}{\frac{\abs{y}^a}{\sqrt{b}}\frac{\partial}{\partial \rho}\left( \mu \sqrt{b}\right)u^2\mathrm{d}V_{\partial B_r}}
$$
where $\partial_\rho u$ denotes the radial differentiation $\partial_\rho u = \langle \nabla_g u, X/\rho \rangle$ for $X \in \R^{n+1}$. Finally, by \eqref{condition1} and \eqref{condition2} we obtain
\begin{equation}\label{deriv1}
  \frac{d}{dr}H_g(u,r) = \frac{2}{r^{n+a}}\int_{\partial B_g(r)}{\abs{y}^a\mu u \partial_\rho u \mathrm{d}V_{\partial B_r}} + O(1)H_g(u,r),
\end{equation}
with $O(1)$ a function bounded in absolute value by a constant $C=C(n,\Lambda)$. On the other hand, the divergence theorem gives
$$
\int_{B_g(r)}{\abs{y}^a \mu \abs{\nabla u}^2 \mathrm{d}V_{g}} = -\int_{B_g(r)}{u \mbox{div}_g\left(\abs{y}^a \mu \nabla u\right) \mathrm{d}V_{g}} + \int_{\partial B_g(r)}{\abs{y}^a \mu u \partial_\rho u \mathrm{d}V_{\partial B_r}}
$$
and hence we can rewrite \eqref{deriv1} as
\begin{equation}\label{derivataH}
\frac{d}{dr} H_g(u,r) = \frac{2}{r}E_g(u,r) + O(1) H_g(u,r).
\end{equation}
We now focus on the derivative of $r\mapsto E_g(u,r)$, following the idea of the radial deformation in \cite{MR833393,MR882069}: for $0<r,\Delta r< 1/2$ fixed, we define $w_t\colon\R^+\to \R^+$ by
$$
w_t(\rho) =
\begin{cases}
  t, & \mbox{if } \rho\leq r \\
  1, & \mbox{if } \rho \geq r+\Delta r \\
  t\ddfrac{r+\Delta r - \rho}{\Delta r} + \ddfrac{\rho-r}{\Delta r}, & \mbox{if } r \leq \rho \leq r+\Delta r.
\end{cases}
$$
Now, for $0<t<1+\Delta r/(r+\Delta r)$, we define the bi-Lipschitz map $l_t\colon \R^{n+1}\to \R^{n+1}$ as
$$l_t(X)=w_t(\rho(X)) X,$$
 with $\rho(X)=\mbox{dist}_g(0,X)$, and consequently the radial deformation $u_t$ of $u$ as
$$
u^t(X)= u(l_t^{-1}(X)) \in H^{1,a}(B_1,\mathrm{d}V_g).
$$
By definition we have $u^t(Z)=u(X)$, with $Z=l_t(X)$.
Since $u$ is a solution of \eqref{div-mu}, given the functional $I(t)=E_g(u^t,1)$ we have
\begin{equation}\label{I(t)}
\frac{d}{dt}I(t)\bigg\lvert_{t=1} \!\!=0.
\end{equation}
In order to ease the notations, through the following computations we will simply use $B_r$ instead of $B_{g}(r)$. Inspired by the definition of $w(t)$,  let us set
\begin{align*}
  I(t) & =  \int_{B_{rt}}{\abs{y}^a \mu \abs{\nabla u^t}^2 \mathrm{d}V_{B_r}} + \int_{B_{r+\Delta r} \setminus B_{rt}}{\abs{y}^a \mu \abs{\nabla u^t}^2 \mathrm{d}V_{B_r}} + \int_{B_1 \setminus B_{r+\Delta r}}{\abs{y}^a \mu \abs{\nabla u^t}^2 \mathrm{d}V_{B_r}} \\
   & = I_1(t) + I_2(t) + I_3(t).
\end{align*}
It is easy to see that
$$
I_3(t)=\int_{B_1 \setminus B_{r+\Delta r}}{\abs{y}^a \mu \abs{\nabla u^t}^2 \mathrm{d}V_{B_r}} = \int_{B_1 \setminus B_{r+\Delta r}}{\abs{y}^a \mu \abs{\nabla u}^2 \mathrm{d}V_{B_r}}
$$
and consequently that $I_3(t)$ does not give contribution to the derivative of $I(t)$. Next, we have
\begin{align*}
  I_1(t) =&\,  \int_{B_{rt}}{\abs{y}^a \mu \abs{\nabla u^t}^2 \mathrm{d}V_{B_r}}  \\
   =&\,  \int_0^r \int_{\partial B_{1}}{t^a\abs{(\rho,\theta_n)}^a \mu(t\rho,\theta) \partial_\rho u^2(\rho,\theta)\frac{\sqrt{g(t\rho,\theta)}}{t}\mathrm{d}\theta \mathrm{d}\rho} \\
   &\, + \int_0^r \int_{\partial B_{1}}{t^a\abs{(\rho,\theta_n)}^a \mu(t\rho,\theta) b^{ij}(t\rho,\theta)\partial_{\theta_i}u(\rho,\theta)\partial_{\theta_j}u(\rho,\theta) t\sqrt{g(t\rho,\theta)}\mathrm{d}\theta \mathrm{d}\rho},
\end{align*}
where obviously $b^{ij}$ are the entries of the inverse of $(b_{ij})_{ij}$ associated with the metric $g$. By \eqref{condition2}, we obtain
$$
\abs{\frac{\partial}{\partial t}\mu(t\rho,\theta)}\leq \Lambda(n,\lambda,\Gamma) \rho.
$$
Furthermore, we can rewrite
\begin{equation}\label{expansions}
 \begin{cases}
\sqrt{g(t\rho,\theta)}= t^n \rho^n \sqrt{b(t\rho,\theta)}\\
b^{ij}(t\rho,\theta) \sqrt{g(t\rho,\theta)} = t^{n-2}\rho^{n-2}\left[\delta_{ij} + \varepsilon_{ij}(t \rho, \theta)\right]
\end{cases}
\end{equation}
for some $(\varepsilon_{ij}(t\rho,\theta))_{ij}$. Since that \eqref{condition1}, we have
$$
\abs{\frac{\partial}{\partial t}\sqrt{b(t\rho,\theta)}}\leq C(n,\Lambda)\rho, \quad
\abs{\frac{\partial}{\partial t}\sqrt{\varepsilon_{ij}(t\rho,\theta)}}\leq C(n,\Lambda)\rho
$$
which gives
\begin{align*}
  I_1(t) =&\, t^{n+a-1} \left[\int_0^r \int_{\partial B_{1}}{\abs{(\rho,\theta_n)}^a \mu(t\rho,\theta) \partial_\rho u^2(\rho,\theta)\rho^n\sqrt{b(t\rho,\theta)}\mathrm{d}\theta \mathrm{d}\rho}\right. \\
   &\, + \left.\int_0^r \int_{\partial B_{1}}{\abs{(\rho,\theta_n)}^a \mu(t\rho,\theta) \rho^{n-2} (\delta_{ij}+ \varepsilon(t\rho,\theta))\partial_{\theta_i}u(\rho,\theta)\partial_{\theta_j}u(\rho,\theta) \mathrm{d}\theta \mathrm{d}\rho}\right],
\end{align*}
and consequently
\begin{equation}\label{I1(t)}
\frac{d}{dt}I_1(t)\bigg\lvert_{t=1}  = (n+a-1)\int_{B_r}{\abs{y}^a \mu \abs{\nabla_g u}^2 \mathrm{d}V_g} + O(r) \int_{B_r}{\abs{y}^a \mu \abs{\nabla_g u}^2 \mathrm{d}V_g},
\end{equation}
with $O(r)$ a function of $(r,\theta)$ whose absolute value is bounded by $C(n,\Lambda) r$. \\Finally, in order to estimate the second term of $I(t)$, we need to introduce the following notations. Hence, given  $X \in B_{r+\Delta r}\setminus B_r$ and $Z=l_t(X) \in B_{r+\Delta r}\setminus B_{rt}$ let us consider their expression in the intrinsic geodesic polar coordinates associated with $g$, namely $X=(\rho,\theta)$ and $ Z=(\gamma_t(\rho),\theta)$, where
$$
\gamma_t(\rho)=\mbox{dist}_g(Z,0)=w_t(X)\rho = \rho\left[ t\frac{r + \Delta r -\rho}{\Delta r} + \frac{\rho -r }{\Delta r}\right].
$$
and
$$
\frac{\partial}{\partial\rho} \gamma_t(\rho)= t \frac{r+\Delta r -2\rho}{\Delta r}+\frac{2\rho -r}{\Delta r}.
$$
Then, still using the polar coordinates, we have
\begin{align*}
  \abs{\nabla_g u^t(Z)}^2 & =\abs{\partial_s u^t(s,\theta)}^2 + \frac{1}{s^2}b^{ij}(s,\theta)\partial_{\theta_i}u^t(s,\theta) \partial_{\theta_j}u^t(s,\theta)\Big\lvert_{s=\gamma_t(\rho)} \\
   & = \abs{\partial_\rho u(\rho,\theta)}^2 \left( \frac{\partial}{\partial s}\gamma_t^{-1}(s)\Big\lvert_{s=\gamma_t(\rho)}\right)^2+ \frac{1}{\gamma_t(\rho)^2}b^{ij}(\gamma_t(\rho),\theta)\partial_{\theta_i}u(\rho,\theta) \partial_{\theta_j}u(\rho,\theta),\\
   & = \abs{\partial_\rho u(\rho,\theta)}^2 h_t(\rho)^2+ \frac{1}{\gamma_t(\rho)^2}b^{ij}(\gamma_t(\rho),\theta)\partial_{\theta_i}u(\rho,\theta) \partial_{\theta_j}u(\rho,\theta),
\end{align*}
and similarly the volume element is given by
$$
\mathrm{d}V_{B_{r}} (Z) = \gamma_t(\rho)^{n}\sqrt{g(\gamma_t(\rho),\theta)}\frac{\partial }{\partial \rho} \gamma_t(\rho) \mathrm{d}\rho \mathrm{d}\theta.
$$
By the previous computations and the expansions in \eqref{expansions}, we obtain
\begin{align*}
  I_2(t) & = \int_{B_{r+\Delta r} \setminus B_{rt}}{\abs{y}^a \mu \abs{\nabla u^t}^2 \mathrm{d}V_{B_r}(Z)}\\
   =&\,  \int_r^{r+\Delta r} \int_{\partial B_{1}}{s^{n}\abs{(s,\theta_n)}^a h_t(\rho)\mu(s,\theta) \partial_\rho u^2(\rho,\theta)\sqrt{b(s,\theta)}\,\bigg\lvert_{s=\gamma_t(\rho)}\hspace{-0.5cm}\mathrm{d}\theta \mathrm{d}\rho} \\
   &\, + \int_r^{r+\Delta r} \int_{\partial B_{1}}{s^{n-2}\abs{(s,\theta_n)}^a \frac{\partial}{\partial \rho}\gamma_t(\rho)\mu(s,\theta) (\delta_{ij}+ \varepsilon(s,\theta))\partial_{\theta_i}u(\rho,\theta)\partial_{\theta_j}u(\rho,\theta) \,\bigg\lvert_{s=\gamma_t(\rho)}\hspace{-0.5cm}\mathrm{d}\theta \mathrm{d}\rho}.
\end{align*}
Since
$$
h_t(\rho) 
=\frac{\Delta r + t \rho -  \rho}{t(r+\Delta r - \rho) + \rho -r}, \quad \frac{\partial}{\partial t}h_t(\rho) \Big\lvert_{t=1}= -\frac{\Delta r +r -2\rho}{\Delta r},
$$
we can conclude
\begin{align*}
\frac{d}{dt} I_2(t)\bigg\lvert_{t=1} = &  \int_{B_{r+\Delta r}\setminus B_r}{\abs{y}^a \mu \left[\left(n+a+O(\rho)\right) \frac{r+\Delta r - \rho}{\Delta r} - \frac{r+\Delta r - 2\rho}{\Delta r} \right] (\partial_\rho u)^2\mathrm{d}V_{B_{r}}}\\
& +\int_{B_{r+\Delta r}\setminus B_r}{\abs{y}^a \mu \left[\left(n+a-2+O(\rho)\right) \frac{r+\Delta r - \rho}{\Delta r} + \frac{r+\Delta r - 2\rho}{\Delta r} \right] \left(\abs{\nabla_g u}^2 - (\partial_\rho u)^2\right)\mathrm{d}V_{B_{r}}}.
\end{align*}
Finally, by letting $\Delta r \to 0^+$ we obtain
\begin{equation}\label{I3}
\frac{d}{dt} I_2(t)\bigg\lvert_{t=1} = 2r\int_{\partial B_r}{\abs{y}^a\mu (\partial_\rho u)^2\mathrm{d}V_{\partial B_r}} - r\int_{\partial B_r}{\abs{y}^a\mu \abs{\nabla_g u}^2\mathrm{d}V_{\partial B_r}}.
\end{equation}
From \eqref{I(t)}, \eqref{I1(t)} and \eqref{I3}, we obtain
$$
r\frac{d}{dr}\int_{B_r}{\abs{y}^a \mu \abs{\nabla_g u}^2\mathrm{d}V_{\partial B_r}} -\left(n+a-1+O(r)\right)\int_{B_r}{\abs{y}^a \mu \left(\nabla_g u\right)^2\mathrm{d}V_{B_r}} = 2r\int_{\partial B_r}{\abs{y}^a\mu (\partial_\rho u)^2\mathrm{d}V_{\partial B_r}},
$$
which implies with \eqref{derivataH} that
$$
\frac{d}{dr}\log{N(X_0,u,r)}  = O(1)+\ddfrac{2\int_{\partial B_r}{\abs{y}^a\mu (\partial_\rho u)^2\mathrm{d}V_{\partial B_r}} }{\int_{\partial B_r}{\abs{y}^a \mu u\partial_\rho u \mathrm{d}V_{\partial B_r}} } - \ddfrac{2\int_{\partial B_r}{\abs{y}^a \mu u\partial_\rho u \mathrm{d}V_{\partial B_r}}}{\int_{\partial B_g(r)}{\abs{y}^a\mu u^2 \mathrm{d}V_{\partial B_r}}} \geq -C(n,\Lambda),
$$
where the inequality is a consequence of Schwarz's inequality. It follows immediately that the map $r\mapsto \exp(C(n,\Lambda)r)N_g(u,r)$ is a monotone nondecreasing function on $r\in (0,1)$ as required.
\end{proof}
Returning to the formulation of the problem in the euclidian metric, for $X_0 \in \Sigma$ and $r\in (0,1-\abs{X_0})$ we set
\begin{align*}
E(X_0,u,r) &= \frac{1}{r^{n+a-1}}\int_{B_r(X_0)}{\abs{y}^a \langle \overline{A}(X) \nabla u, \nabla u\rangle \mathrm{d}X}\\
H(X_0,u,r) &= \frac{1}{r^{n+a}}\int_{\partial B_r(X_0)}{\abs{y}^a \mu_0 u^2\mathrm{d}\sigma},
\end{align*}
and consequently
$$
N(X_0,u,r)= \frac{E(X_0,u,r)}{H(X_0,u,r)},
$$
with $\mu_0$ is a positive Lipschitz function bounded in $B_1$ satisfying \eqref{condition1} with $\Lambda$ depending only on $n,\lambda$ and $\Gamma$. \begin{corollary}\label{almgren.A}
  Let $a \in (-1,1)$ and $u$ be a solution of \eqref{L.ext} in $B_1$ symmetric with respect to $\Sigma$. Then there exist a constant $C>0$ such that for every $X_0\in B_1 \cap \Sigma$ the map
  $$r \mapsto e^{Cr}N(X_0,u,r)$$ is absolutely continuous and monotone nondecreasing on $(0,1-\abs{X_0})$.Hence, there exists finite the \emph{Almgren limiting frequency} defined as
$$
N(X_0,u,0^+) = \lim_{r\to 0^+}N(X_0,u,r) = \inf_{r>0}N(X_0,u,r).
$$
\end{corollary}
Now, we can finally apply the previous analysis to the general case $(-L)^s$, by proving the validity of a doubling condition, a compactness result for blow-up sequences and a general Theorem on the structure of the nodal set itself.
\begin{proposition}
  Let $a \in (-1,1)$ and $u$ be a solution of \eqref{L.ext} in $B_1$. Hence, there exists a constant $C=C(n,\Lambda)$ such that, for every $X_0 \in B_1\cap \Sigma$,
  $$
H(X_0,u,r_2)\leq C H(X_0,u,r_1) \left(\frac{r_2}{r_1}\right)^{\!\!2\tilde{C}}
$$
for $0<r_1<r_2< 1-\abs{X_0}$, where $\tilde{C}= N(X_0,u,R)e^{C(n,\Lambda) R}$.
\end{proposition}
\begin{proof}
Fixed $R=1-\abs{X_0}$, by Corollary \ref{almgren.A} we have that $N(X_0,u,r) \leq e^{C R}N(X_0,u,R)$ for every $r \in (0,R)$. By \eqref{derivataH} we obtain
\begin{align*}
\frac{d}{dr}\log H(X_0,u,r) &= \frac{2}{r}N(X_0,u,r)+ O(1)\\
&\leq \frac{2}{r}N(X_0,u,R)e^{C(n,\Lambda) R}+ C(n,\Lambda),
\end{align*}
for every $0<r<R$. Now we integrate between $0<r_1<r_2< R$, obtaining
$$
\log\frac{H(X_0,u,r_2)}{H(X_0,u,r_1)} \leq 2N(X_0,u,R)e^{C(n,\Lambda) R}\log \frac{r_2}{r_1} + C(n,\Lambda)(r_2-r_1)
$$
and finally
$$
\frac{H(X_0,u,r_2)}{H(X_0,u,r_1)} \leq e^{C(n,\Lambda)R}\left(\frac{r_2}{r_1}\right)^{\!\!2\tilde{C}}
$$
with $\tilde{C}= N(X_0,u,R)e^{C(n,\Lambda) R}$.
\end{proof}
Moreover, since we are dealing with the extension $L_a^A$ of operator uniformly elliptic in divergence form with Lipschitz coefficent, we can easily extend Corollary \ref{lower.N} to our new class of operator following the technique developed in \cite{susste}. Indeed, since the lower bound on the Almgren frequency formula is based on the H\"older regularity of $L_a^A$-harmonic function, we easily obtain
\begin{corollary}
      Let $u$ be $L_a^A$-harmonic on $B_1$, then for every $X_0 \in \Gamma(u) \cap \Sigma$ we have
      \begin{equation}
        N(X_0,u,0^+)\geq \min\{1,1-a \}.
      \end{equation}
      More precisely
      \begin{itemize}
        \item if $u$ is symmetric with respect to $\Sigma$, we have $N(X_0,u,0^+)\geq 1$,
        \item if $u$ is antisymmetric with respect to $\Sigma$ we have $N(X_0,u,0^+)\geq 1-a$.
      \end{itemize}
    \end{corollary}
In particular, since in this Section we are focusing on the symmetric case, we directly obtain $N(X_0,u,0^+)\geq 1$, for any $X_0 \in \Gamma(u)\cap \Sigma$.
All techniques presented in this manuscript involve a local analysis of the solutions, which will be performed via a blow-up procedure. The following result are a generalization of the ones in Section \ref{section.blowup}. Fixed $a\in (-1,1)$ and $u$ an $L_a^A$-harmonic function in $B_1$, for every $X_0=(x_0,0) \in \Gamma(u)\cap \Sigma$ and $r_k \downarrow 0^+$ we define as the blow-up sequence the collection
    \begin{equation*}
    u_k(X)= \ddfrac{u(X_0 + r_k X)}{\sqrt{H(X_0,u,r_k)}}\quad \mbox{for } X \in X\in B_{X_0,r_k}=\frac{B_1 - X_0}{r_k},
    \end{equation*}
    such that $L_a^{A_k} u_k = 0$ and $\norm{u_k}{L^{2,a}(\partial B_1)}=1$, where
    $$
    L_a^{A_k} = \mbox{div}_{x,y}\left(\abs{y}^a \overline{A}_k(x)\nabla_{x,y} \right),\quad\mbox{with } \overline{A}_k(x)=\overline{A}(x_0+r_k x),
    $$
    for every $X \in B_{X_0,r_k}$.
    \begin{proposition}
    Let $a \in (-1,1)$. Given $X_0 \in \Gamma(u)\cap \Sigma$ and a blow-up sequence $u_k$ centered in $X_0$ and associated with some $r_k \downarrow 0^+$, there exists $p \in H^{1,a}_{\loc}(\R^n)$ such that, up to a subsequence, $u_k\to p$ in $C^{0,\alpha}_{\loc}(\R^n)$ for every $\alpha \in (0,1)$ and strongly in $H^{1,a}_{\loc}(\R^n)$. In particular, the blow-up limit is and entire solution of following elliptic equation with constant coefficient
    $$
    \mbox{div}_{x,y}\left(\abs{y}^a \overline{A}(x_0) \nabla_{x,y}p\right) = 0\quad\mathrm{ in }\,\,\R^{n+1}.
    $$
    \end{proposition}
The proof of this result is a straightforward adaption of the one of Theorem \ref{blowup.convergence}. In particular, since the coefficient of $\overline{A}$ are Lipschitz continuous and uniformly elliptic, all the computations of the blow-up argument follows the line of the local counterpart in
\cite{MR1305956,MR1090434,MR882069,MR833393,MR2984134,MR0140031}.\\
Moreover, since for every $X_0 \in \Gamma(u)\cap \Sigma$ the blow- up limit satisfies a degenerate-singular equation with constant coefficients, it is not restrictive to suppose that $\overline{A}(x_0)=\mathrm{Id}$, since by trivial transformation we can rewrite the equation in a canonical form.\\
Therefore, all the results on the structure of the singular strata, proved in the previous part of the Section for the nodal set of $s$-harmonic functions, remain valid for the nodal set of fractional power of divergence form operator with Lipschitz leading coefficients. Indeed, as we already pointed out, in the proof of Theorem \ref{regularity.singular} and Theorem \ref{general.singular} we never used Proposition \ref{smooth} in order to attain the result on the structure of the singular strata on $\Sigma$. The crucial idea is that the Whitney extension allows to study the structure of the nodal set just by using the generalized Taylor expansion \eqref{taylor.generalized} for symmetric function without the high-order differentiability of the function itself. In this way the results can be easily generalized to our class of operators.
 \begin{proposition}
  Given $s\in (0,1)$, let $u$ be a solution of
   $$
   (-L)^s u = 0 \quad\mbox{in }B_1,
   $$
   with $L$ a uniformly elliptic operator with Lipschitz coefficient defined as \eqref{Ldefi}. Then the nodal set $\Gamma(u)$ splits into its regular and singular part
   $$
   \mathcal{R}(u) = \{x \in \Gamma(u)\colon \abs{\nabla u(x)}\neq 0 \}\quad\mbox{and}\quad
      \mathcal{S}(u)= \{x \in \Gamma(u)\colon \abs{\nabla u(x)}=0\}.
   $$
   Moreover, if $u\in C^1(B_{1/2})$, on one hand $\mathcal{R}(u)$ is locally a smooth hypersurface and on the other one there holds
   $$
  \mathcal{S}(u)= \mathcal{S}^*(u) \cup \mathcal{S}^s(u)
  $$
  where $\mathcal{S}^*(u)$ is contained in a countable union of $(n-2)$-dimensional $C^1$ manifolds and $\mathcal{S}^s(u)$ is contained in a countable union of $(n-1)$-dimensional $C^1$ manifolds. Moreover
  $$
   \mathcal{S}^*(u)=\bigcup_{j=0}^{n-2} \mathcal{S}^*_j(u)\quad\mbox{and}\quad
   \mathcal{S}^s(u)=\bigcup_{j=0}^{n-1}\mathcal{S}^s_j(u),
  $$
  where both $\mathcal{S}^*_j(u)$ and $\mathcal{S}^s_j(u)$ are contained in a countable union of $j$-dimensional $C^1$ manifolds.
\end{proposition}

In this following, we extend the previous results to the case of nonlocal elliptic equation with a potential. In particular,
let $s\in (0,1)$ and $u\colon B_1\subset \R^n\to \R$ be a nontrivial solution of
\begin{equation}\label{Ps.potential}
(-\Delta)^s u(x) = V(x)u \quad \mbox{in }B_1,
\end{equation}
where $V\in W^{1,q}(B_1)$, for some $q \geq n/2s$. Following, the same strategy of Section \ref{section.diverg}, by the local realisation of the fractional Laplacian, let us consider $u \in H^{1,a}(B_1^+)$ a solution of
\begin{equation}\label{potential}
\begin{cases}
  L_a u= 0 & \mbox{in } B_1^+ \\
  -\partial^a_y u = V(x)u & \mbox{in } B_1 \subset \Sigma,
\end{cases}
\end{equation} with
$V \in W^{1,q}(B_1)$, for some $q \in [n/2s,+\infty]$.
Hence, for every $X_0 \in B_1$ and $r\in (0,1-\abs{X_0})$, let us consider
\begin{align*}
  E(X_0,u,r) 
  &=  \frac{1}{r^{n+a-1}}\left[\int_{B_r^+(X_0)}{\abs{y}^a \abs{\nabla u}^2 \mathrm{d}X} - \int_{B_r(X_0)}{V u^2\mathrm{d}x}\right],\\
  H(X_0,u,r) & = \frac{1}{r^{n+a}}\int_{\partial^+ B^+_r(X_0)}{\abs{y}^a u^2 \mathrm{d}\sigma}
\end{align*}
and the Almgren quotient
\begin{equation}\label{Almgren.a}
N(X_0,u,r)= \ddfrac{E(X_0,u,r)}{H(X_0,u,r)}.
\end{equation}
Before to state the main result on the monotonicity of the Almgren quotient, we recall a general class of Poho\v{z}aev type identities, which will allow to compute the derivative of the functionals previously defined. Indeed, by multiplying the equation \eqref{potential} with $\langle X, \nabla u\rangle $, and integrating by parts over $B^+_r(X_0)$, for some $X_0 \in B_1$ and $r\in (0,1-\abs{X_0})$, we obtain
\begin{align}\label{pohoz}
\begin{aligned}
\frac{1-n-a}{2}\int_{B^+_r(X_0)}{\abs{y}^a \abs{\nabla u}^2\mathrm{d}X} + \frac{r}{2}\int_{\partial^+ B^+_r(X_0)}{\abs{y}^a \abs{\nabla u}^2\mathrm{d}\sigma} = &r \int_{\partial^+ B^+_r(X_0)}{\abs{y}^a (\partial_r u)^2 \mathrm{d}\sigma}+\\
&+ \int_{B_r(X_0)}{-\partial^a_y u \langle x, \nabla u\rangle\mathrm{d}x}.
\end{aligned}
\end{align}
The following lemmata is a simple generalization of similar result obtained for the case of the Laplacian in \cite{tvz1}. First, let $a\in (-1,1)$ and $u\in H^{1,a}(B_r^+(X_0))$ for some $X_0 \in B_1$ and $r\in (0,\mbox{dist}(X_0,\partial B_1))$. Then, for every $p \in [2, p^\#]$, where $p^\# = 2n/(n-2s)$ there exists a constant $C(n,p,s)$ such that
\begin{equation}\label{poincare}
\left[\frac{1}{r^{n}}\int_{B_r(X_0)}{\abs{u}^p\mathrm{d}x} \right]^{\frac{2}{p}} \leq C(n,p,s) \left[ \frac{1}{r^{n-1+a}}\int_{B^+_r(X_0)}{\abs{y}^a\abs{\nabla u}^2\mathrm{d}X} + \frac{1}{r^{n+a}}\int_{\partial^+ B^+_r(X_0)}{\abs{y}^a u^2\mathrm{d}X}\right]
\end{equation}
This result is a direct consequence of the characterization of the class of trace of $H^{1,a}(B^+_1)$ in \cite{Nekvinda} and the critical Sobolev exponent for the trace embedding in the context of fractional Sobolev-Slobodeckij spaces $W^{s,2}(K)$, with $s\in (0,1)$ and $K\subset \R^n$.
\begin{lemma}\label{lemma1}
  Let $a \in (-1,1)$. Then, given $u \in H^{1,a}(B^+_1)$ a solution of \eqref{potential}, for every $p\in [2,p^\#]$ and $X_0\in B_1\cap \Sigma$ there exist constants $C>0, \overline{r}>0$ such that
  $$
  \left[\frac{1}{r^{n}}\int_{B_r(X_0)}{\abs{u}^p\mathrm{d}x} \right]^{\frac{2}{p}} \leq C\left(E(X_0,u,r)+H(X_0,u,r)\right),
  $$
  for every $r\in (0,\overline{r})$.
\end{lemma}
\begin{proof}
  Since $u\in L^\infty(B^+_1)$ and  $V \in W^{1,q}(B_1)$ for some $q \in [n/2s,+\infty]$, we obtain
  \begin{align*}
  \abs{\frac{1}{r^{n-1+a}}\int_{B_r(X_0)}{V u^2 \mathrm{d}x}}& \leq \norm{V}{L^q}r^{1-a}\left[\frac{1}{r^n}\int_{B_r(X_0)}{\abs{u}^{q^*}\mathrm{d}x}\right]^{\frac{2}{q^*}}\\
  &\leq C r^{1-a}\left[\frac{1}{r^{n-1+a}}\int_{B^+_r(X_0)}{\abs{y}^a\abs{\nabla u}^2\mathrm{d}X} + \frac{1}{r^{n+a}}\int_{\partial^+ B^+_r(X_0)}{\abs{y}^a u^2\mathrm{d}X}\right],
  \end{align*}
  where we used the trace inequality in the case $q^*=2q/(q-1)$, since $q^*\leq p^\#$. Finally, since $a\in(-1,1)$ we obtain
  \begin{equation}\label{tipotav}
  E(X_0,u,r) + H(X_0,u,r) \geq (1 - C r^{1-a})\left[\frac{1}{r^{n-1+a}}\int_{B^+_r(X_0)}{\abs{y}^a\abs{\nabla u}^2\mathrm{d}X} + \frac{1}{r^{n+a}}\int_{\partial^+ B^+_r(X_0)}{\abs{y}^a u^2\mathrm{d}X}\right],
  \end{equation}
  the result follows by taking into account the trace inequality and choosing $\overline{r}>0$ sufficiently small.
\end{proof}
Following the same idea in \cite{tvz1} for the case $s=1/2$, let introduce  for $ p\in (2, p^\#]$ the auxiliary function
$$
\psi(X_0,u,r) = \left(\frac{1}{r^n}\int_{B_r(X_0)\cap \Sigma}{\abs{u}^2\mathrm{d}X}\right)^{1-\frac{2}{p}}
$$
which is bounded for $r\in (0,\mbox{dist}(X_0,\partial B_1))$.  Under this notations, for $a\in (-1,1)$ consider
$$
\Psi(X_0,u,r) = C(n,s)\int_0^r{t^{-a}\left(1+\frac{d}{dt}\left(t\psi(X_0,u,t\right)\right)\mathrm{d}t},
$$
which is well defined on $r\in (0,\mbox{dist}(X_0,\partial B_1))$ such that $\lim_{r\to 0^+}\Psi(X_0,u,r)=0$, since $\psi(X_0,u,r)$ is bounded for $r$ sufficiently small.
In order to simplify the notations, through the Section we will just use the notation $\psi(r)$ and $\Psi(r)$ for the auxiliary functions previously defined.
\begin{lemma}\label{lemma2}
  Let $a\in (-1,1)$ and $u \in H^{1,a}(B^+_1)$ be a solution of \eqref{potential}. Then, for every $p\in (2,p^\#]$ and $X_0\in \partial^0 B^+_1$ there exist constants $C>0, \overline{r}>0$ such that
  $$
  \frac{1}{r^{n-1}}\int_{S^{n-1}_r(X_0)}{\abs{u}^p\mathrm{d}\sigma} \leq C\left(E(X_0,u,r)+H(X_0,u,r)\right)\frac{d}{dr}\left(r\psi(r)\right),
  $$
  for every $r \in (0,\overline{r})$.
\end{lemma}
\begin{proof}
The proof follows it is the same of \cite[Lemma 9.5]{tvz1} make exception in our case is based on the generalized Poincar$\grave{\mbox{e}}$ inequality \eqref{poincare}. Hence, a direct computation yields the identity
$$
\frac{d}{dr}\left(r\psi(r)\right) = \psi(r)\left[r\left(1-\frac{2}{p}\right) \ddfrac{\int_{S^{n-1}_r}{\abs{u}^p\mathrm{d}\sigma}}{\int_{\partial^0 B^+_r}{\abs{u}^p\mathrm{d}\sigma}} + \left(1-n\left(1-\frac{2}{p}\right)\right) \right],
$$
and, since $p\leq p^\#$ implies $n(1-2/p)\leq 1$, we infer
$$
\frac{d}{dr}\left(r\psi(r)\right)\geq r\psi(r)\left(1-\frac{2}{p}\right) \ddfrac{\int_{S^{n-1}_r}{\abs{u}^p\mathrm{d}\sigma}}{\int_{\partial^0 B^+_r}{\abs{u}^p\mathrm{d}\sigma}}.
$$
Finally, recalling the definition of $\psi$ and using Lemma \ref{lemma1}, we deduce
$$
\left(E(X_0,u,r)+H(X_0,u,r)\right)\frac{d}{dr}\left(r\psi(r)\right)\geq C \frac{1}{r^{n-1}}\int_{S^{n-1}_r}{\abs{u}^p\mathrm{d}\sigma}.
$$
\end{proof}
We are now ready to prove the boundedness of the Almgren quotient, rather than its monotonicity, considering a modified version of the quotient.
\begin{proposition}\label{mono.potential}
  Given $a\in (-1,1), u\in H^{1,a}(B^+_1)$ a solution of \eqref{potential} and $\Omega^+\subset\subset B^+_1$, there exist constants $C, \overline{r}>0$ such that, for every $X_0 \in \Omega \cap \Sigma$ and $r \in (0,\overline{r})$ such that $B_r^+(X_0)\subset B^+_1$, we have that $H(X_0,u,r)>0$ and $N(X_0,u,r)>0$ for every $r\in (0,\overline{r})$.
  Moreover, the map $$r\mapsto e^{C\Psi(X_0,u,r)}\left(N(X_0,u,r) +1 \right)$$ is monotone non decreasing on $(0,\overline{r})$, which ensures the existence of limit
  $$
  N(X_0,u,0^+)=\lim_{r\to 0^+}N(X_0,u,r),
  $$
  which is finite and called the Almgren frequency of $u$ at $X_0$.
\end{proposition}
\begin{proof}
Let $X_0\in B_1\cap\Sigma$ and $\overline{r}>0$ be such that $\overline{r} < \mbox{dist}(\Omega^+, B^+_1)$ and Lemma \ref{lemma1} and Lemma \ref{lemma2} hold true. First, let us consider the following modified Almgren frequency formula
\begin{equation}\label{new.almgren}
\widetilde{N}(X_0,u, r) = \frac{E(X_0, u,r)}{H(X_0,u,r)}+ 1 = N(X_0,u,r)+1.
\end{equation}
Under this notations, we obtain by Lemma \ref{lemma1}
$$
E(X_0,u,r)+ H(X_0,u,r) \geq 0 \longrightarrow \widetilde{N}(X_0,u, r) = \frac{E(X_0, u,r)}{H(X_0,u,r)}+ 1\geq 0,
$$
whenever $H(X_0,u,r)\neq 0$. By continuity of $r\mapsto H(X_0,u,r)$ we can consider a reasonable neighborhood of $r$ where it does not vanish. Now, taking into account the Poho\v{z}aev identity \eqref{pohoz}, if we differentiate the map $r\mapsto E(X_0,u,r)$ we obtain
\begin{align*}
  \frac{d}{dr}E(X_0,u,r) & = \frac{2}{r^{n+a-1}}\int_{\partial^+ B^+_r(X_0)}{\abs{y}^a (\partial_r u)^2 \mathrm{d}\sigma} + R(X_0,u,r),\\
  \frac{d}{dr}H(X_0,u,r) & = \frac{2}{r^{n+a}}\int_{\partial^+ B^+_r(X_0)}{\abs{y}^a u\partial_r u \mathrm{d}\sigma},
\end{align*}
where the remainder $R(X_0,u,r)$ takes the form
\begin{align*}\label{utile}
R(X_0,u,r) = 
&\frac{2}{r^{n+a}}\int_{B_r(X_0)}{V u \langle x, \nabla u\rangle \mathrm{d}x} - \frac{1-n-a}{r^{n+a}}\int_{B_r(X_0)}{V u^2\mathrm{d}x}+\\
&- \frac{1}{r^{n-1+a}}\int_{S^{n-1}_r(X_0)}{V u^2\mathrm{d}x}.
\end{align*}
In particular, if $V \in W^{1,q}(B_1)$, for some $q \in [n/2s,+\infty]$, we obtain that the following integrals are well defined
$$
2\int_{B_r(X_0)}{V u \langle x, \nabla u\rangle \mathrm{d}x} = \int_{S^{n-1}_r(X_0)}{Vu^2\mathrm{d}\sigma} - n\int_{B_r(X_0)}{Vu^2\mathrm{d}x} + \int_{B_r(X_0)}{u^2 \langle \nabla V,x\rangle \mathrm{d}x},
$$
which it implies the estimate on the remainder
$$
\abs{R(X_0,u,r)} \leq C(n,a)\norm{V}{W^{1,q}}\left[\frac{1}{r^{n+a}}\int_{B_r(X_0)}{u^2\mathrm{d}x}+\frac{1}{r^{n+a-1}}\int_{S^{n-1}_r(X_0)}{ u^2\mathrm{d}\sigma}\right],
$$
where we used Lemma \ref{lemma1} and Lemma \ref{lemma2}. Therefore, differentiating the Almgren quotient and using the Cauchy-Schwarz inequality on $\partial^+ B^+_r$, we obtain
\begin{align*}
\frac{d}{dr}\widetilde{N}(X_0,u,r) =&\, \ddfrac{\ddfrac{d}{dr}E(X_0,u,r)+\ddfrac{d}{dr}H(X_0,u,r)}{E(X_0,u,r)+H(X_0,u,r)} - \ddfrac{\ddfrac{d}{dr}H(X_0,u,r)}{H(X_0,u,r)} \\
\geq&\, \frac{2 H(X_0,u,r)}{r^{2n+2a-1}}\left[\int_{\partial^+B^+_r}{\abs{y}^a (\partial_r u)^2\mathrm{d}\sigma}\int_{\partial^+B^+_r}{\abs{y}^a u^2\mathrm{d}\sigma} - \left(\int_{\partial^+ B^+_r}{\abs{y}^a \langle u,\partial_r u\rangle\mathrm{d}\sigma}\right)^2\right]+\\
&\,-C(n,s) \widetilde{N}(X_0,u,r) r^{-a}\left(1+\frac{d}{dr}(r\psi(r))\right)\\
\geq&\,-C(n,s) \widetilde{N}(X_0,u,r) r^{-a}\left(1+\frac{d}{dr}(r\psi(r))\right).
\end{align*}
which implies that the function
$$
r\mapsto e^{C\Psi(X_0,u,r)}\widetilde{N}(X_0,u,r)
$$
is nondecreasing as far as $H(X_0,u,r)\neq 0$. Passing to the logarithmic derivative of $r\mapsto H(X_0,u,r)$ we infer for $r\in (r_1,r_2)$ we obtain
\begin{equation}\label{int}
\frac{d}{dr}\log H(X_0,u,r)=\frac{2}{r}N(X_0,u,r).
\end{equation}
More precisely, we can choose $r_1 = 0, r_2 = +\infty$. On one hand, the above equation provides that, if $\log H(X_0,u,R) >-\infty$ then $\log H(X_0,u,r) >-\infty$ for every $r >R$, so
that $r_2 = \mbox{dist}(X_0,\partial B_1)$. Now, on the other hand assume by contradiction that
$$
r_1 = \inf\left\{r \colon H(X_0,u,r) > 0 \mbox{ in } (r,r_2)\right\} > 0.
$$
By the monotonicity result on the modified Almgren quotient \eqref{new.almgren}, we have that
$$
N(X_0,u,r) < e^{C\Psi(2 r_1)}\left(N(X_0,u,2 r_1)+1\right)-1,
$$
for every $r_1<r\leq 2r_1$.
Hence, integrating \eqref{int} between $r$ and $2r_1$, we obtain
$$
\frac{H(X_0,u,2r_1)}{H(X_0,u,r)}\leq \left(\frac{2r_1}{r}\right)^{2\left(e^{C\Psi(2 r_1)}\left(N(X_0,u,2 r_1)+1\right)-1\right)}
$$
and, since $r\mapsto H(X_0,u,r)$ is continuous we deduce the absurd $H(X_0,u,r_1) > 0$.
\end{proof}
Inspired by the previous part of the paper, we will prove a compactness result in order to perform a blow-up analysis of the nodal set $\Gamma(u)$. Hence, fixed $a\in (-1,1)$ and $u$ a solution of \eqref{poincare}, consider now $X_0 \in \Gamma(u)$ a point on the nodal set of $u$, then for any $r_k \downarrow 0^+$ we define as the blow-up sequence the collection
    \begin{equation*}\label{blow.up.almgren}
    u_k(X)= \ddfrac{u(X_0 + r_k X)}{\sqrt{H(X_0,u,r_k)}}\quad \mbox{for } X \in X\in B_{X_0,r_k}=\frac{B_1 - X_0}{r_k},
    \end{equation*}
    such that $L_a u_k = 0$ and $\norm{u_k}{L^{2,a}(\partial B_1)}=1$.
    \begin{proposition}\label{blowup.convergence.potential}
    Let $a \in (-1,1)$ and $\alpha^*=\min\{1,1-a\}$. Given $X_0 \in \Gamma(u)\cap \Sigma$ and a blow-up sequence $u_k$ centered in $X_0$ and associated with some $r_k \downarrow 0^+$, there exists $p \in H^{1,a}_{\loc}(\R^n)$ such that, up to a subsequence, $u_k\to p$ in $C^{0,\alpha}_{\loc}(\R^n)$ for every $\alpha \in (0,\alpha^*)$ and strongly in $H^{1,a}_{\loc}(\R^n)$. In particular, the blow-up limit is and entire $L_a$-harmonic function symmetric with respect to $\Sigma$, i.e.
    $$
    \begin{cases}
      L_a p = 0  & \mathrm{in }\,\R^{n+1} \\
      p(x,y)=p(x,-y) & \mathrm{in }\,\R^{n+1}.
    \end{cases}
    $$
    \end{proposition}
    We will prove the result in a series of lemmata.
    \begin{lemma}\label{bounded.H1}
  Let $X_0 \in \Gamma(u)\cap\Sigma$. For any given $R>0$, we have  $$\norm{u_{k}}{H^{1,a}(B^+_R)}\leq C\quad\mbox{and}\quad\norm{u_{k}}{L^\infty(\overline{B^+_R})}\leq C,
$$
where $C>0$ is a constant independent of $k>0$.
\end{lemma}
\begin{proof}
Let us consider $\rho^2_k = H(X_0,u,r_k)$, then by definition of the blow-up sequence $u_k$  and Corollary \ref{doubling.corollary.Sigma} we obtain
 \begin{align*}
    \int_{\partial^+ B^+_R}{\abs{y}^a u^2_k \mathrm{d}\sigma} & = \frac{1}{\rho_k^2}\int_{\partial^+ B^+_{R}}{\abs{y}^a u^2(X_0 + r_k X) \mathrm{d}\sigma} \\
&=\frac{1}{\rho_k^2 r_k^{n+a}}\int_{\partial^+ B^+_{Rr_k}(X_0)}{\abs{y}^a u^2 \mathrm{d}\sigma}\\
    &= R^{n+a}\frac{H(X_0,u,Rr_k)}{H(X_0,u,r_k) }\\
    & \leq R^{n+a}\left( \frac{R r_k}{r_k}\right)^{\!2\tilde{C}}
\end{align*}
which gives us $\norm{u_k}{L^{2,a}(\partial B_R)}^2\leq C(R) R^{n+a}$. Similarly
\begin{align}\label{c.overline}
\begin{aligned}
\int_{B^+_{R}}{\abs{y}^a \abs{\nabla u_k}^2 \mathrm{d}\sigma} & = N(0,u_k,R)\frac{1}{R}\int_{\partial^+ B^+_R}{\abs{y}^a u^2_k \mathrm{d}\sigma}\\
&\leq C(R) R^{n-1+a} N(X_0,u,Rr_k)\\
& \leq C(R) R^{n-1+a} N(X_0,u,R)
\end{aligned}
\end{align}
where in the last inequality we used the monotonicity result of Proposition \ref{Almgren.formula}. Since the map $u_k$ is $L_a$- harmonic, by \cite[Lemma A.2]{ww} we obtain
\begin{align*}
  \sup_{\overline{B_{R/2}}} u_k &\leq C(n,s) \left(\frac{1}{R^{n+1+a}}\int_{B^+_R}{\abs{y}^a u_k^2\mathrm{d}X} \right)^{1/2} \\
  &\leq   C(n,s) \left(\frac{H(0,u_k,R)}{n+a+1}\right)^{1/2},
\end{align*}
where in the second inequality we used the monotonicity of $r\mapsto H(0,u_k,r)$ in $(0,R)$. Finally, the estimate follows directly from the one the $L^{2,a}(\partial B_R)$-norm.
\end{proof}
\begin{lemma}\label{growthestim}
  Let $a \in (-1,1)$ and $u \in H^{1,a}(B_1)$ be a solution of \eqref{potential}. Then, given $X_0 \in \Gamma_k(u)\cap \Sigma$, there exists $C>0$ such that
  $$
  \frac{1}{r^{n+a}}\int_{\partial^+B^+_r(X_0)}{\abs{y}^a u^2\mathrm{d}\sigma} \leq C r^{2k} \quad\mathrm{and}\quad   \frac{1}{r^{n-1+a}}\int_{B^+_r(X_0)}{\abs{y}^a \abs{\nabla u}^2\mathrm{d}X}\leq C r^{2k},
  $$
  for every $r \in (0,\mathrm{dist}(X_0,\partial B_1)/2)$. Moreover, by \eqref{poincare}, for every $p \in [2, p^\#]$, where $p^\# = 2n/(n-2s)$ there exists a constant $C>0$ such that
\begin{equation}\label{poincare-lavendetta}
\left[\frac{1}{r^{n}}\int_{B_r(X_0)}{\abs{u}^p\mathrm{d}x} \right]^{\frac{2}{p}} \leq C r^{2k}.
\end{equation}
\end{lemma}
\begin{proof}
By Proposition \ref{mono.potential} and \eqref{int}, there exists a constant $C>0$ and $\overline{r}>0$ such that, for every $r \in (0,\overline{r})$ we obtain
\begin{align*}
\frac{d}{dr}\log \frac{H(X_0,u,r)}{r^k} & = \frac{2}{r}(N(X_0,u,r)-k)\\
& \geq  \frac{2}{r} \left(e^{-C\Psi(r)}e^{C \Psi(r)}(N(X_0,u,r)+1)-1-k\right)\\
& \geq  \frac{2(k+1)}{r} \left(e^{-C\Psi(r)}-1\right),
\end{align*}
which implies, after an integration by part, that
$$
\frac{H(X_0,u,r)}{r^{2k}}\leq \frac{H(X_0,u,\overline{r})}{\overline{r}^{2k}} \exp\left(\int_0^{\overline{r}}{\frac{2(k+1)}{\rho} \left(e^{-C\Psi(\rho)}-1\right)\mathrm{d}\rho}\right)\leq C,
$$
which it implies the first inequality.
Now, by the Cacciopoli inequality associated with the $L_a$-opertor and the monotonicity of $r\mapsto H(X_0,u,r)$ due to \eqref{int}, we obtain
\begin{align*}
\frac{1}{r^{n-1+a}}\int_{B^+_{r/2}(X_0)}{\abs{y}^a\abs{\nabla u}^2\mathrm{d}X} &\leq \frac{1}{r^{n-1+a}}\frac{C}{r^2} \int_{B^+_r(X_0)}{\abs{y}^a u^2 \mathrm{d}\sigma}\\
&\leq
\frac{1}{r^{n-1+a}}\frac{C}{r^2} \int_0^r{\rho^{n+a}H(X_0,u,\rho)\mathrm{d}\rho}\\
&\leq C H(X_0,u,r),
\end{align*}
which yields the claimed conclusion.
\end{proof}
Let $r\mapsto W_k(X_0,u,r)$ be the $k$-Weiss function defined as
$$
W_k(X_0,u,r)=\frac{1}{r^{n+a+2k}}\int_{\partial^+ B^+_r(X_0)}{\abs{y}^a u \left( \langle \nabla u , X-X_0\rangle - u \right)\mathrm{d}\sigma},
$$
or, equivalently
$$
W_k(X_0,u,r) = \frac{1}{r^{2k}}\left(E(X_0,u,r) - k H(X_0,u,r)\right).
$$
\begin{proposition}\label{weiss.formula.potential}
Let $a\in (-1,1)$ and $u$ be a nontrivial solution of \eqref{potential}. For $X_0 \in \Gamma_k(u)\cap \Sigma$, given the $k$-Weiss function
\begin{align*}
W_k(X_0,u,r) = & \frac{1}{r^{n+a-1+2k}}\left[\int_{B_r^+(X_0)}{\abs{y}^a \abs{\nabla u}^2 \mathrm{d}X} - \int_{B_r(X_0)}{V u^2\mathrm{d}x}\right]+\\
& - \frac{k}{r^{n+a+2k}}\int_{\partial^+ B^+_r(X_0)}{\abs{y}^a \abs{u}^2\mathrm{d}\sigma}.
\end{align*}
For $r\in (0,1-\abs{X_0})$ we have
$$
\frac{d}{dr}W_k(X_0,u,r) \geq -C(n,a,k)\norm{V}{W^{1,q}(B_1)} r^{-a}.
$$
\end{proposition}
\begin{proof}
By a direct computation, we have
\begin{equation}\label{weiss.potential}
\frac{d}{dr}W_k(X_0,u,r)=\frac{2}{r^{n+a+1+2k}}\int_{\partial^+ B^+_r(X_0)}{\abs{y}^a \left(\langle \nabla u, X-X_0\rangle - k u \right)^2 \mathrm{d}\sigma}+ R_k(X_0,u,r),
\end{equation}
with
$$
R_k(X_0,u,r) = \frac{4k-1-a}{r^{n+a+2k}}\int_{B_r(X_0)\cap \Sigma}{V u^2\mathrm{d}x} - \frac{1}{r^{n+a+2k}} \int_{B_r(X_0)\cap \Sigma}{u^2\langle \nabla V,x\rangle \mathrm{d}x}.
$$
By Lemma \ref{lemma1}, Lemma \ref{lemma2} and Lemma \ref{growthestim} it follows
\begin{align*}
\abs{R_k(X_0,u,r)} &\leq \frac{C(n,a,k)\norm{V}{W^{1,q}(B_1)}}{r^{a+2k}}\left[\frac{1}{r^n}\int_{B_r(X_0)\cap \Sigma}{\abs{u}^{q^*}\mathrm{d}x}\right]^{\frac{1}{q^*}}\\
&\leq \frac{C(n,a,k,V)}{r^{a+2k}}\left[ \frac{1}{r^{n-1+a}}\int_{B^+_r(X_0)}{\abs{y}^a\abs{\nabla u}^2\mathrm{d}X} + \frac{1}{r^{n+a}}\int_{\partial^+ B^+_r(X_0)}{\abs{y}^a u^2\mathrm{d}X}\right]\\
& \leq C(n,a,k,V) r^{-a},
\end{align*}
where $q^*=2q/(q-1)$, which implies that
$$
r\mapsto W_k(X_0,u,r) + C(n,a,k)\norm{V}{W^{1,q}(B_1)} r^{1-a}
$$ is monotone nondecreasing in $(0,1-\abs{X_0})$.
\end{proof}
\begin{proposition}\label{monneau.pot}
  Let $a\in (-1,1), u \in H^{1,a}(B_1^+)$ be a solution of \eqref{potential} and $X_0 \in \Gamma_k(u)\cap \Sigma$. For every homogenous $L_a$-harmonic polynomial $p\in \mathfrak{sB}^a_k(\R^{n+1})$, the map
  $$
  r \mapsto \frac{H(X_0,u-p_{X_0},r)}{r^{2k}}= \frac{1}{r^{n+a+2k}}\int_{\partial B_r(X_0)}{\abs{y}^a \left(u -p_{X_0}\right)^2\mathrm{d}\sigma}
  $$
  satisfies
  $$
  \frac{d}{dr}\frac{H(X_0,u-p_{X_0},r)}{r^{2k}} \geq -C\norm{p_{X_0}}{L^\infty(B_1)} \norm{V}{W^{1,q}(B_1)}r^{-a},
  $$
  for every $r \in (0,1-\abs{X_0})$, with $p_{X_0}(X)=p(X-X_0)$.
\end{proposition}
\begin{proof}
 Let $w = u - p_{X_0}$, then on one hand we have
 \begin{align*}
 \frac{d}{dr}\left(\frac{1}{r^{n+a+2k}}\int_{\partial^+ B^+_r(X_0)}{\abs{y}^a w^2 \mathrm{d}\sigma}\right) =& \frac{d}{dr} \int_{\partial^+ B^+_1}{\abs{y}^a\frac{w(X_0+rX)^2}{r^{2k}}\mathrm{d}\sigma}\\
 =& 2\int_{\partial^+ B^+_1}{\abs{y}^a\frac{w(X_0+rX)\left(\langle\nabla w(X_0+rX) , rX \rangle -k w(X_0+rX)\right)}{r^{2k+1}}\mathrm{d}\sigma}\\
 =& \frac{2}{r^{n+a+1+2k}}\int_{\partial^+ B^+_r(X_0)}{\abs{y}^a w(\langle X-X_0, \nabla w\rangle -k w)\mathrm{d}\sigma}\\
 =& \frac{2}{r}W_k(X_0,w,r).
  \end{align*}
On the other hand, looking at the expression of the $k$-Weiss function, we have
\begin{align*}
    W_k(X_0,u,r) =&\, W_k(X_0,w+p_{X_0},r)\\ 
     =&\, \frac{1}{r^{n+a-1+2k}}\left(\int_{B^+_r(X_0)}{\abs{y}^a (\abs{\nabla w}^2 +2  \langle \nabla w,\nabla p\rangle)\mathrm{d}X}- \frac{k}{r}\int_{\partial^+ B^+_r(X_0)}{\abs{y}^a (w^2 +2  w p) \mathrm{d}\sigma}\right) \\
     &+ \frac{1}{r^{n+a-1+2k}}\int_{B_r(X_0)}{(w+p_{X_0})\partial^a_y (w+p_{X_0}) \mathrm{d}x}\\
     =&\, W_k(X_0,w,r) + \frac{1}{r^{n+a-1+2k}}\int_{B_r(X_0)}{p_{X_0}\partial^a_y w\mathrm{d}x} +\\
     &+ \frac{2}{r^{n+a+2k}}\int_{\partial^+ B^+_r(X_0)}{\abs{y}^a w( \langle \nabla p_{X_0},X-X_0\rangle - k p) \mathrm{d}\sigma}   \\
    =&\, W_k(X_0,u-p_{X_0},r)+ \frac{1}{r^{n+a-1+2k}}\int_{B_r(X_0)}{p_{X_0}\partial^a_y u\mathrm{d}x} ,
  \end{align*}
  where in the second equality we used the $k$-homogeneity of $p_{X_0}\in \mathfrak{sB}^a_k(\R^{n+1})$.
Hence 
we finally infer
\begin{align*}
\frac{d}{dr} \frac{H(X_0,u-p_{X_0},r)}{r^{2k}} &=\frac{2}{r}W_k(X_0,u-p_{X_0},r)\\ &=\frac{2}{r}W_k(X_0,u,r)+\frac{2}{r^{n+a+2k}}\int_{B_r(X_0)}{p_{X_0}V u\mathrm{d}x}.
\end{align*}
On one hand, by Proposition \ref{weiss.formula.potential} we have
$$
W_k(X_0,u,r) =W_k(X_0,u,r) -W_k(X_0,u,0^+) \geq -C(n,a,k)\norm{V}{W^{1,q}(B_1)} r^{1-a},
$$
while on the other one
$$
\abs{\frac{2}{r^{n+a+2k}}\int_{B_r(X_0)}{p_{X_0}V u\mathrm{d}x}}  \leq C(n,a,k)\norm{V}{L^q(B_1)}\norm{p_{X_0}}{L^\infty(B_1)} r^{-a},
$$
where we used Lemma \ref{lemma1} and the growth estimate of Lemma \ref{growthestim}. Hence, since $a \in (-1,1)$, it follows immediately that
$$
r\mapsto \frac{H(X_0,u-p_{X_0},r)}{r^{2k}} + C\norm{p_{X_0}}{L^\infty(B_1)} \norm{V}{W^{1,q}(B_1)}r^{1-a},
$$
is monotone nondecreasing in $(0,1-\abs{X_0})$.
\end{proof}
Now, in the same spirit of Section \ref{sec.weiss}, we prove the nondegeneracy of the functions at the singular points and then the existence and uniqueness of the tangent map.
\begin{lemma}
\label{nondegeneracy.potential}
Let $a\in(-1,1)$ and $u$ be a solution of \eqref{potential}. Then, for every $X_0 \in \Gamma_k(u)\cap \Sigma$  there exists $C>0$ such that
$$
\sup_{\partial B_r(X_0)}{\abs{u(X)}}\geq C r^{k}\quad \mathrm{for}\,\,0<r<R
$$
where $R=1-\mathrm{dist}(X_0,\partial B_1)$.
\end{lemma}
\begin{proof}
Fix $X_0\in \Gamma_k(u)$ and suppose by contradiction, given a decreasing sequence $r_j\downarrow 0$, that
$$
\lim_{j\to \infty}\frac{H(X_0,u,r_j)^{1/2}}{r_j^{k}}=\lim_{j\to\infty}{\bigg(\frac{1}{r_j^{n+a+2k}}\int_{\partial B_{r_j}(X_0)}{\abs{y}^a u^2\,d\sigma\bigg)^{1/2}}}\!\!=0.
$$
Consider now the blow-up sequence
$$
u_j(X)=\frac{u(X_0+r_j X)}{\rho_j}\quad \mbox{where }\, \rho_j = H(X_0,u,r_j)^{1/2}
$$
constructed starting from $r_j$ and centered in $X_0 \in \Gamma_k(u)$. By Theorem \ref{blowup.convergence.potential}, up to a subsequence $u_j \to p$ uniformly, where $p \in \mathfrak{sB}^a_k(\R^{n+1})$ is a nontrivial homogenous polynomial of degree $k$, such that $H(0,p,1)=1$.\\
Let us focus our attention on the functional $M(X_0,u,p_{X_0},r)$ with $p_{X_0}$ as above. By the assumption on the growth of $u$ it follows, as before, that
$$
M(X_0,u,p_{X_0},0^+) = \int_{\partial B_1}{\abs{y}^a p^2\,d\sigma} =\frac{1}{r^{n+a+2k}}\int_{\partial B_r(X_0)}{\abs{y}^a {p}_{X_0}^2\,d\sigma}.
$$
By the monotonicity result of Proposition \ref{monneau.pot} on the map $$r\mapsto M(X_0,u,p_{X_0},r) +C\norm{p_{X_0}}{L^\infty(B_1\cap \Sigma)} \norm{V}{L^\infty(B_1\cap \Sigma)}r^{1-a},$$ we obtain
$$
\frac{1}{r^{n+a-1+2k}}\int_{\partial B_r(X_0)}{\abs{y}^a(u- p_{X_0})^2\,d\sigma}+C(n,p_{X_0},V)r^{1-a}
\geq\frac{1}{r^{n+a-1+2k}}\int_{\partial B_r(X_0)}{\abs{y}^a p_{X_0}^2\,d\sigma}
$$
and similarly
$$
\frac{1}{r^{n+a-1+2k}}\int_{\partial B_r(X_0)}{\abs{y}^a (u^2-2u p_{X_0})\,d\sigma}\geq -C(n,p_{X_0},V)r^{1-a}.
$$
On the other hand, rescaling the previous inequality and using the blow-up sequence $u_j$ defined as above, we obtain
\begin{equation}\label{ineq}
\int_{\partial B_1}{\abs{y}^a\left(\frac{H(X_0,u,r_j)}{r^k_j}u_j^2-2H(X_0,u,r_j)^{1/2}  u_j p\right)\,d\sigma}\geq -C(n,p_{X_0},V)\frac{{r_j}^{k+1-a}}{H(X_0,u,r_j)}.
\end{equation}
Now, since $N(X_0,u,0^+)=k$, by Proposition \ref{mono.potential} we obtain that $H(X_0,u,r)> C r^{2A}$, for every $A > k$. Hence, simply by taking $A=k+1$, we have that the right hand side of \eqref{ineq} goes to zero and,
passing to the limit for $j\to\infty$, by the previous inequality we obtain
$$
\int_{\partial B_1}{\abs{y}^a p^2\,d\sigma}\leq 0
$$
in contradiction with $p \not\equiv 0$.
\end{proof}

\begin{theorem}
\label{uniqueness.potential}
Given $a \in (-1,1)$ and $u$ be a solution of \eqref{potential}, let us consider $X_0\in \Gamma_k(u)\cap \Sigma$, i.e. $N(X_0,u,0^+)=k$. Then there exists a unique nonzero $p \in \mathfrak{sB}_k^a(\R^{n+1})$ blow-up limit such that
\begin{equation}\label{blow.up.homogenous.potential}
  u_{X_0,r}(X) =\frac{u(X_0+rX)}{r^k} \longrightarrow  p(X).
\end{equation}
\end{theorem}
\begin{proof}
Up to a subsequence $r_j\to 0^+$, we have that $u_{X_0,r_j}\to p$ in $\mathcal{C}^{0,\alpha}_{\loc}$. The existence of such limit follows directly from the growth estimate $\abs{u(X)} \leq C\abs{X}^k$; moreover, by Lemma \ref{nondegeneracy.potential} we know $p$ is not identically zero. Now, for any $r>0$ we have
$$
W_k(0,p,r) = \lim_{j\to \infty}{W_k(0,u_{X_0,r_j},r)}= \lim_{j\to \infty}{W_k(X_0,u,r r_j)}= W_k(X_0,u,0^+)=0.
$$
In particular, Proposition \ref{weiss.formula.potential} implies that the $L_a$-harmonic function $p$ is $k$-homogeneous and symmetric with respect to $\Sigma$, i.e. $p \in \mathfrak{sB}^a_k(\R^{n+1})$. By Proposition \ref{monneau.pot}, the limit $M(X_0,u,p_{X_0},0^+)$ exists and can be computed by
\begin{align*}
M(X_0,u,p_{X_0},0^+) &= \lim_{j\to\infty}{M(X_0,u,p_{X_0},r_j)}\\
&= \lim_{j\to\infty}{M(0,u_{X_0,r_j},p,1)}\\
&= \lim_{j\to\infty}{\int_{\partial B_1}{\abs{y}^a(u_{X_0,r_j} -p)^2\,d\sigma}}=0.
\end{align*}
Moreover, let us suppose by contradiction that there exists another sequence $r_i\to 0^+$ for which the associated blow-up sequence converges to a different  limit, i.e. $u_{X_0,r_i}\to q\in \mathfrak{sB}^a_k(\R^{n+1}), q\not\equiv p$, then
\begin{align*}
0=M(X_0,u,p_{X_0},0^+) &= \lim_{i\to \infty}M(X_0,u,p_{X_0},r_i)\\
&=\lim_{i\to\infty}\int_{\partial B_1}{\abs{y}^a (u_{r_i} - p)^2\,d\sigma}\\
&=\int_{\partial B_1}{\abs{y}^a (q- p)^2\,d\sigma}.
\end{align*}
Since $q$ and $p$ are both homogenous of degree $k$ they must coincide in $\R^n$.
\end{proof}

\section{Measure estimates of nodal sets of $s$-harmonic functions}\label{sec.measure}
In this last Section, we estimate the measure of the nodal set $\Gamma(u)$ of $s$-harmonic functions. Our result can be seen as the nonlocal counterpart of the conjecture that Lin proposed in \cite{MR1090434}. Indeed, following his strategy, we will give an explicit estimate on the $(n-1)$-Hausdorff measure of the nodal set in terms of the Almgren frequency of its $L_a$-extension.\\

We will  keep the notations previously introduced: more precisely, through this Section we will denote with $v \in H^{1,a}(B_1^+)$ the restriction on the unitary ball in $\R^{n+1}_+$ of the $L_a$-harmonic extension, defined by \eqref{eq:Pextended}, symmetric with respect to $\Sigma$ (see \eqref{problem}).
Since the fractional Laplacian $(-\Delta)^s$ admits a representation formula, we directly have that the analyticity assumption, which is fundamental in order to apply a strategy developed in \cite{MR1090434}, is fully satisfied on every compact set $K\subset\subset B_1$. Moreover, by Proposition \ref{doubling.s.prop} we already have at our disposal a quantitative doubling condition for $s$-harmonic functions strictly related to the one in the extended space $\R^{n+1}$.\\

In order to achieve the estimate on the Hausdorff measure of the nodal set $\Gamma(u)$ we use the following lemma, introduced in \cite{MR943927}, relating the growth of a complex analytic function with the number of its zeros.
\begin{lemma}\label{tech}
Let $f\colon B_1\subset \CCC \to \CCC$ be an  analytic function such that
$$
\abs{f(0)}=1\quad\mbox{and}\quad \sup_{B_1}\abs{f} \leq 2^N,
$$
for some positive constant $N$. Then for any $r \in (0, 1)$
$$
\#\left\{z \in B_r \colon  f(z) = 0  \right\}\leq cN
$$
and
$$
\#\left\{z \in B_{1/2} \colon  f(z) = 0  \right\}\leq N,
$$
where $C$ is a positive constant depending only on the radius $r$.
\end{lemma}
Before stating the main result on the measure of the nodal sets $\Gamma(u)$ 
in terms of the Almgren frequency of the $L_a$-harmonic extension, let us start with an example in the setting of tangent maps $\mathfrak{B}^s_k(\R^n)$ that emphasizes how the measure of the nodal set is strictly related to the class of tangent maps that we are considering. 
First, it is not restrictive to assume that $\varphi \in \mathfrak{B}^s_k(\R^{2})$ for some $k\in 1+\N$. Hence, consider the case $n=2$ with the notation $(x,z)\in \R^2$. Since every $\varphi\in \mathfrak{B}_k^*(\R^{2})$ is harmonic in $\R^2$, it is known that
$$
\mathcal{H}^1\left(\Gamma(\varphi)\cap B_1\right)=2 k.
$$
In contrast, whenever $\varphi \in \mathfrak{B}_k^s(\R^{2})\setminus\mathfrak{B}_k^*(\R^{2})$ with $k\geq 2$, the previous bound turns out to be not optimal. More precisely, given the constant $k'= \#\{t \in \R\colon \varphi(t,1)=0\}$ , we obtain
$$
\mathcal{H}^1\left(\Gamma(\varphi)\cap B_1\right)=2 k',
$$
where, by the Fundamental Theorem of Algebra, it is obvious to see that $0\leq k'\leq k$.\\
In general, we prove the following result which is based on an argument first introduced in \cite{MR1090434} in the context of solution of second order elliptic equation with analytic coefficient.\\ More recently, in \cite{MR1714341} the author constructs a similar estimate in a more general context connecting the Hausdorff measure of the nodal set of smooth functions with their finite vanishing order, which can be also applied to our case. Unfortunately, the remarkable difference between the case $\mathfrak{B}^*_k(\R^{n})$ and $\mathfrak{B}^s_k(\R^{n})\setminus \mathfrak{B}^*_k(\R^{n})$ ( or similarly $\mathfrak{sB}^*_k(\R^{n+1})$ and $\mathfrak{sB}^a_k(\R^{n+1})\setminus \mathfrak{sB}^*_k(\R^{n+1})$  ) implies the non optimality of  B\"{a}r's result in our setting.
\begin{theorem}\label{hau}
    Given $s\in (0,1)$, let $u$ be an $s$-harmonic function in $B_1$ and $0 \in \Gamma(u)$. Then
  $$
  \mathcal{H}^{n-1}\left(\Gamma(u)\cap B_{\frac{1}{2}}\right) \leq C(n,s) N,
  $$
  where $N=N(0,v,1)$ is the frequency of the $L_a$-harmonic extension $v$ in $B_1^+$ defined by
  $$
  N=\ddfrac{\int_{B_1^+}{\abs{y}^a \abs{\nabla v}^2\mathrm{d}X}}{\int_{\partial B_1^+}{\abs{y}^a v^2\mathrm{d}\sigma}}.
  $$
\end{theorem}
\begin{proof}
Let $(B_R(p_i))_{i}$ be a finite cover of $B_{1/2}$ with $R< 1/8$ and $p_i \in B_{1/2}$. Moreover, up to a normalization, it is not restrictive to assume that
$$
\fint_{B_1}{u^2\mathrm{d}x}= 1.
$$
By Proposition \ref{import} and Proposition \ref{doubling.s.prop}, for every $p_i \in B_{1/2}$ we have
$$
\fint_{B_r(p_i)}{u^2\mathrm{d}x}\geq 4^{-C(n,s)N}
\fint_{B_{2r}(p_i)}{u^2\mathrm{d}x},
$$
with $0<r< 1/4$ and $N=N(0,v,1)$. Moreover, using the normalization hypothesi, we obtain
$$
\fint_{B_{R}(p_i)}{u^2\mathrm{d}x}\geq 4^{-C(n,s)N}.
$$
Given $p_1,\dots,p_j \in B_{1/2}$ the collection of points associated with the covering, let us consider $(x_{p_i})_i\in B_{R}(p_i)$ such that
$$
\abs{u(x_{p_i})}\geq 2^{-C(n,s)N}, \quad \mbox{for any }i=1,\dots,j.
$$
In order to apply Lemma \ref{tech}, for $i=1,\dots,j$ consider the collection of analytic functions of one complex variable defined as
$$
f_i(w,z)=u(x_{p_i}+4R z w), \quad\mbox{for } w \in S^{n-1}, z \in B^{\CCC}_1
$$
Then, by construction, we have
$$
\abs{f_i(w,0)} \geq 2^{-C(n,s)N}\quad\mbox{and}\quad \abs{f_i(w,z)}\leq C,
$$
for some positive dimensional constant $C>0$. Since, by Lemma \ref{tech} we have
\begin{align*}
N_i(w) & = \#\left\{x \in B_{2R}(x_{p_i})\colon u(x)=0 \quad \mbox{for } (x-x_{p_j})\parallel w \right\}\\
&\leq \#\left\{z \in B_{1/2}^{\CCC}\colon f_i(w,z)=0\right\}\\
&\leq c(n,s,N)N,
\end{align*}
for every $i=1,\cdots,j$, by the integral geometric formula in \cite[Theorem 3.2.27]{Federer}, we finally obtain
$$
 \mathcal{H}^{n-1}\left(\Gamma(u)\cap B_{1/2}\right) \leq \sum_{i=1}^j \mathcal{H}^{n-1}\left(\Gamma(v)\cap B_{R}(p_i)\right)
\leq  c(n,s,N)\sum_{i=1}^j\int_{S^{n-1}}{N_i(w)\mathrm{d}w}
\leq C(n,s,N)N
$$
where in the second inequality we used $B_R(p_i)\subset B_{2R}(x_{p_i})$ for every $i=1,\dots,j$.
\end{proof}

\bibliography{Biblio}

\begin{thebibliography}{10}

\bibitem{almgren}
F.~J. Almgren, Jr.
\newblock Dirichlet's problem for multiple valued functions and the regularity
  of mass minimizing integral currents.
\newblock In {\em Minimal submanifolds and geodesics ({P}roc. {J}apan-{U}nited
  {S}tates {S}em., {T}okyo, 1977)}, pages 1--6. North-Holland, Amsterdam-New
  York, 1979.

\bibitem{MR0140031}
N.~Aronszajn, A.~Krzywicki, and J.~Szarski.
\newblock A unique continuation theorem for exterior differential forms on
  {R}iemannian manifolds.
\newblock {\em Ark. Mat.}, 4:417--453 (1962), 1962.

\bibitem{audrito}
A.~{Audrito} and S.~{Terracini}.
\newblock {On the nodal set of solutions to a class of nonlocal parabolic
  reaction-diffusion equations}.
\newblock {\em ArXiv e-prints}, July 2018.

\bibitem{toroengel}
M.~Badger, M.~Engelstein, and T.~Toro.
\newblock Structure of sets which are well approximated by zero sets of
  harmonic polynomials.
\newblock {\em Anal. PDE}, 10(6):1455--1495, 2017.

\bibitem{banuelos}
R.~Banuelos and K.~Bogdan.
\newblock Symmetric stable processes in cones.
\newblock {\em Potential Analysis}, 21(3):263--288, 2004.

\bibitem{MR1714341}
C.~B\"ar.
\newblock Zero sets of solutions to semilinear elliptic systems of first order.
\newblock {\em Invent. Math.}, 138(1):183--202, 1999.

\bibitem{MR1671973}
K.~Bogdan and T.~Byczkowski.
\newblock Potential theory for the {$\alpha$}-stable {S}chr\"odinger operator
  on bounded {L}ipschitz domains.
\newblock {\em Studia Math.}, 133(1):53--92, 1999.

\bibitem{MR2393430}
L.~Caffarelli and F.-H. Lin.
\newblock Singularly perturbed elliptic systems and multi-valued harmonic
  functions with free boundaries.
\newblock {\em J. Amer. Math. Soc.}, 21(3):847--862, 2008.

\bibitem{MR2367025}
L.~Caffarelli, S.~Salsa, and L.~Silvestre.
\newblock Regularity estimates for the solution and the free boundary of the
  obstacle problem for the fractional {L}aplacian.
\newblock {\em Invent. Math.}, 171(2):425--461, 2008.

\bibitem{CS2007}
L.~Caffarelli and L.~Silvestre.
\newblock An extension problem related to the fractional {L}aplacian.
\newblock {\em Comm. Partial Differential Equations}, 32(7-9):1245--1260, 2007.

\bibitem{conformal}
S.-Y.~A. Chang and M.~d.~M. Gonz\'alez.
\newblock Fractional {L}aplacian in conformal geometry.
\newblock {\em Adv. Math.}, 226(2):1410--1432, 2011.

\bibitem{CNV}
J.~Cheeger, A.~Naber, and D.~Valtorta.
\newblock Critical sets of elliptic equations.
\newblock {\em Comm. Pure Appl. Math.}, 68(2):173--209, 2015.

\bibitem{MR2146353}
M.~Conti, S.~Terracini, and G.~Verzini.
\newblock Asymptotic estimates for the spatial segregation of competitive
  systems.
\newblock {\em Adv. Math.}, 195(2):524--560, 2005.

\bibitem{MR943927}
H.~Donnelly and C.~Fefferman.
\newblock Nodal sets of eigenfunctions on {R}iemannian manifolds.
\newblock {\em Invent. Math.}, 93(1):161--183, 1988.

\bibitem{kufner}
D.~E. Edmunds, A.~Kufner, and J.~R\'akosn\'\i~k.
\newblock Embeddings of {S}obolev spaces with weights of power type.
\newblock {\em Z. Anal. Anwendungen}, 4(1):25--34, 1985.

\bibitem{fjk2}
E.~Fabes, D.~Jerison, and C.~Kenig.
\newblock The {W}iener test for degenerate elliptic equations.
\newblock {\em Ann. Inst. Fourier (Grenoble)}, 32(3):vi, 151--182, 1982.

\bibitem{fks}
E.~Fabes, C.~Kenig, and R.~Serapioni.
\newblock The local regularity of solutions of degenerate elliptic equations.
\newblock {\em Comm. Partial Differential Equations}, 7(1):77--116, 1982.

\bibitem{fkj}
E.~B. Fabes, C.~E. Kenig, and D.~Jerison.
\newblock Boundary behavior of solutions to degenerate elliptic equations.
\newblock In {\em Conference on harmonic analysis in honor of {A}ntoni
  {Z}ygmund, {V}ol. {I}, {II} ({C}hicago, {I}ll., 1981)}, Wadsworth Math. Ser.,
  pages 577--589. Wadsworth, Belmont, CA, 1983.

\bibitem{FallFelli}
M.~M. Fall and V.~Felli.
\newblock Unique continuation property and local asymptotics of solutions to
  fractional elliptic equations.
\newblock {\em Comm. Partial Differential Equations}, 39(2):354--397, 2014.

\bibitem{Federer}
H.~Federer.
\newblock {\em Geometric measure theory}.
\newblock Die Grundlehren der mathematischen Wissenschaften, Band 153.
  Springer-Verlag New York Inc., New York, 1969.

\bibitem{MR833393}
N.~Garofalo and F.-H. Lin.
\newblock Monotonicity properties of variational integrals, {$A_p$} weights and
  unique continuation.
\newblock {\em Indiana Univ. Math. J.}, 35(2):245--268, 1986.

\bibitem{MR882069}
N.~Garofalo and F.-H. Lin.
\newblock Unique continuation for elliptic operators: a geometric-variational
  approach.
\newblock {\em Comm. Pure Appl. Math.}, 40(3):347--366, 1987.

\bibitem{MR2511747}
N.~Garofalo and A.~Petrosyan.
\newblock Some new monotonicity formulas and the singular set in the lower
  dimensional obstacle problem.
\newblock {\em Invent. Math.}, 177(2):415--461, 2009.

\bibitem{GaRo}
N.~Garofalo and X.~Ros-Oton.
\newblock Structure and regularity of the singular set in the obstacle problem
  for the fractional laplacian.
\newblock {\em preprint}, to appear, 2017.

\bibitem{MR2334826}
N.~Garofalo and D.~Vassilev.
\newblock Strong unique continuation properties of generalized
  {B}aouendi-{G}rushin operators.
\newblock {\em Comm. Partial Differential Equations}, 32(4-6):643--663, 2007.

\bibitem{GZ}
C.~R. Graham and M.~Zworski.
\newblock Scattering matrix in conformal geometry.
\newblock {\em Invent. Math.}, 152(1):89--118, 2003.

\bibitem{MR1305956}
Q.~Han.
\newblock Singular sets of solutions to elliptic equations.
\newblock {\em Indiana Univ. Math. J.}, 43(3):983--1002, 1994.

\bibitem{MR2351641}
Q.~Han.
\newblock Nodal sets of harmonic functions.
\newblock {\em Pure Appl. Math. Q.}, 3(3, Special Issue: In honor of Leon
  Simon. Part 2):647--688, 2007.

\bibitem{MR1639155}
Q.~Han, R.~Hardt, and F.-H. Lin.
\newblock Geometric measure of singular sets of elliptic equations.
\newblock {\em Comm. Pure Appl. Math.}, 51(11-12):1425--1443, 1998.

\bibitem{MR1010169}
R.~Hardt and L.~Simon.
\newblock Nodal sets for solutions of elliptic equations.
\newblock {\em J. Differential Geom.}, 30(2):505--522, 1989.

\bibitem{Landkof}
N.~S. Landkof.
\newblock {\em Foundations of modern potential theory}.
\newblock Springer-Verlag, New York-Heidelberg, 1972.
\newblock Translated from the Russian by A. P. Doohovskoy, Die Grundlehren der
  mathematischen Wissenschaften, Band 180.

\bibitem{MR1090434}
F.-H. Lin.
\newblock Nodal sets of solutions of elliptic and parabolic equations.
\newblock {\em Comm. Pure Appl. Math.}, 44(3):287--308, 1991.

\bibitem{logu2}
A.~Logunov.
\newblock Nodal sets of {L}aplace eigenfunctions: polynomial upper estimates of
  the {H}ausdorff measure.
\newblock {\em Ann. of Math. (2)}, 187(1):221--239, 2018.

\bibitem{logu1}
A.~Logunov.
\newblock Nodal sets of {L}aplace eigenfunctions: proof of {N}adirashvili's
  conjecture and of the lower bound in {Y}au's conjecture.
\newblock {\em Ann. of Math. (2)}, 187(1):241--262, 2018.

\bibitem{LM}
A.~Logunov and E.~Malinnikova.
\newblock Ratios of harmonic functions with the same zero set.
\newblock {\em Geom. Funct. Anal.}, 26(3):909--925, 2016.

\bibitem{miller}
K.~Miller.
\newblock Nonunique continuation for uniformly parabolic and elliptic equations
  in self-adjoint divergence form with {H}\"older continuous coefficients.
\newblock {\em Arch. Rational Mech. Anal.}, 54:105--117, 1974.

\bibitem{Nekvinda}
A.~Nekvinda.
\newblock Characterization of traces of the weighted {S}obolev space
  {$W^{1,p}(\Omega,d^\epsilon_M)$} on {$M$}.
\newblock {\em Czechoslovak Math. J.}, 43(118)(4):695--711, 1993.

\bibitem{MR2599456}
B.~Noris, H.~Tavares, S.~Terracini, and G.~Verzini.
\newblock Uniform {H}\"older bounds for nonlinear {S}chr\"odinger systems with
  strong competition.
\newblock {\em Comm. Pure Appl. Math.}, 63(3):267--302, 2010.

\bibitem{Ruland}
A.~R\"uland.
\newblock On quantitative unique continuation properties of fractional
  {S}chr\"odinger equations: doubling, vanishing order and nodal domain
  estimates.
\newblock {\em Trans. Amer. Math. Soc.}, 369(4):2311--2362, 2017.

\bibitem{silvestrethesis}
L.~Silvestre.
\newblock {\em Regularity of the obstacle problem for a fractional power of the
  {L}aplace operator}.
\newblock ProQuest LLC, Ann Arbor, MI, 2005.
\newblock Thesis (Ph.D.)--The University of Texas at Austin.

\bibitem{Simon83}
L.~Simon.
\newblock {\em Lectures on geometric measure theory}, volume~3 of {\em
  Proceedings of the Centre for Mathematical Analysis, Australian National
  University}.
\newblock Australian National University, Centre for Mathematical Analysis,
  Canberra, 1983.

\bibitem{sttv}
Y.~Sire, S.~Terracini, and S.~Vita.
\newblock Liouville type theorems and local behaviour of solutions to
  degenerate-singular problems, 2018, In preparation.

\bibitem{stingatorrea}
P.~R. Stinga and J.~L. Torrea.
\newblock Extension problem and {H}arnack's inequality for some fractional
  operators.
\newblock {\em Comm. Partial Differential Equations}, 35(11):2092--2122, 2010.

\bibitem{MR2984134}
H.~Tavares and S.~Terracini.
\newblock Regularity of the nodal set of segregated critical configurations
  under a weak reflection law.
\newblock {\em Calc. Var. Partial Differential Equations}, 45(3-4):273--317,
  2012.

\bibitem{cones}
S.~Terracini, G.~Tortone, and S.~Vita.
\newblock On s-harmonic functions on cones.
\newblock {\em Anal. PDE}, 11(7):1653--1691, 2018.

\bibitem{tvz2}
S.~Terracini, G.~Verzini, and A.~Zilio.
\newblock Uniform {H}\"older regularity with small exponent in
  competition-fractional diffusion systems.
\newblock {\em Discrete and Continuous Dynamical Systems- Series A},
  34(6):2669--2691, 2014.

\bibitem{tvz1}
S.~Terracini, G.~Verzini, and A.~Zilio.
\newblock Uniform {H}\"older bounds for strongly competing systems involving
  the square root of the laplacian.
\newblock {\em Journal of the European Mathematical Society},
  18(12):2865--2924, 2016.

\bibitem{susste}
S.~Terracini and S.~Vita.
\newblock On the asymptotic growth of positive solutions to a nonlocal elliptic
  blow-up system involving strong competition.
\newblock {\em Ann. Inst. H. Poincar\'e Anal. Non Lin\'eaire}, 35(3):831--858,
  2018.

\bibitem{tesivita}
S.~Vita.
\newblock {\em Strong competition systems ruled by anaomalous diffusions}.
\newblock PhD thesis, Politecnico di Torino and Universit$\grave{\mbox{a}}$ di
  Torino, 2018.

\bibitem{ww}
K.~Wang and J.~Wei.
\newblock On the uniqueness of solutions of a nonlocal elliptic system.
\newblock {\em Mathematische Annalen}, 365(1-2):105--153, 2016.

\bibitem{Whi34}
H.~Whitney.
\newblock Analytic extensions of differentiable functions defined in closed
  sets.
\newblock {\em Trans. Amer. Math. Soc.}, 36(1):63--89, 1934.

\end{thebibliography}
\bibliographystyle{abbrv}


\end{document}